%% file: arxiv.tex
\documentclass{article}
\usepackage{hyperref}
\usepackage{natbib}
\usepackage{amstext}
\usepackage{amssymb}
\usepackage{amsmath}
\usepackage{amsthm}
\usepackage{stmaryrd}
\usepackage[all,2cell,v2]{xy}
\UseAllTwocells
\usepackage{diagxy}
\usepackage{bussproofs}

\input reiterMacros.tex

\def\entailsHPV{\entails^{\textsf{h}}_{\partEq}}

\def\homSet#1#2#3{\textnormal{Hom}_{#1}(#2,#3)}
\def\converges#1{#1\!\downarrow}
\def\ccont{\mathfrak{C}}
\def\ccontT{\ccont_{\textsf{tot}}}
\def\econt{\mathfrak{E}}
\def\econtT{\mathfrak{E}_{\textsf{tot}}}
\def\domf{\mathop{\mathfrak{dom}}}
\def\someCat{\mathbb{C}}
\def\comma{\olessthan}

\def\extendedBy{\sqsubseteq}
\def\extends{\sqsupseteq}
\def\someTheory{\mathbb{T}}
\def\freeCat#1#2{\ccont_{(#1,#2)}}
\def\freeUnsplitCat#1#2{\ccont_{(#1,#2)}^0}
\def\totalCat#1#2{\ccont_{(#1,#2)}^t}
\def\freeHey#1#2{\ccont^{\textsf{h}}_{(#1,#2)}}
\def\freeUnsplitHey#1#2{\ccont_{(#1,#2)}^{\textsf{h}0}}

\def\ob{\mathop{\textsf{Ob}}}
\def\cart{\textsc{Cart}}
\def\domains{\mathcal{M}}
\def\mCat{\domains\cart}
\def\rCat{\textsc{rCat}}
\def\rCatS{\rCat_s}
\def\rCatl{\textsc{rCatl}}
\def\bpm{\textsc{bpm}}

\def\bpms{\bpm_s}
\def\asl{\land \textsc{SeLa}}
\def\poset{\textsc{Poset}}
\def\SplitRtoMCat{\mathop{\domains\textsc{total}}}
\def\MCattoSplitR{\mathop{\textsc{Par}}}
\def\MCattoBPM{\mathop{\textsc{bPar}}}
\def\BPMtoMCat{\mathop{\textsc{bTot}}}

\def\heyting#1#2{{}^{#2}#1}

\def\coheyting#1#2{#1_{#2}}
\def\lapb{\coprod}
\def\rapb{\prod}
\def\id{\textsf{Id}}
\def\dom#1{\mathord{\textsf{dom}}(#1)}
\def\cod#1{\mathord{\textsf{cod}}(#1)}
\def\obC#1#2{\{\vec{#2} | \vec{#1}(\vec{#2})\}}
\def\morphC#1#2{\langle\vec{#1} | \vec{#2}\rangle}

\newlength{\erh}
\newtheorem{lemma}{Lemma}
\newtheorem{proposition}{Proposition}
\newtheorem{theorem}{Theorem}
\newtheorem{example}{Example}
\newtheorem{corollary}{Corollary}
\theoremstyle{definition}
\newtheorem{remark}{Remark}
\newtheorem{definition}{Definition}

\title{Reasoning about Unreliable Actions} 
\author{Graham White\\ 
School of Electronic Engineering and 
Computer Science\\ 
Queen Mary, University of London\\ 
London E1 4NS }
\begin{document} 
\maketitle 
\begin{abstract} 
	
We analyse the philosopher
Davidson's semantics of actions, using a strongly typed logic with contexts
given by sets of partial equations between the outcomes of actions. This
provides a perspicuous and elegant treatment of reasoning about action,
analogous to Reiter's work on artificial intelligence.  We define a sequent
calculus for this logic, prove cut elimination, and give a semantics based on
fibrations over partial cartesian categories: we give a structure theory for
such fibrations.  The existence of lax comma objects is necessary for the proof
of cut elimination, and we give conditions on the domain fibration of a partial
cartesian category for such comma objects to exist.  
\end{abstract}
\tableofcontents 
\section{Introduction} 
\subsection{Background} 

In this paper we describe a logical system for reasoning about unreliable
actions, or, to be precise, actions which can succeed or fail: it continues the
programme, begun in \citep{white:_david_reiter}, of developing strongly typed
logical systems for reasoning about actions. As well as the motivations for the
project as a whole, there are several purely technical reasons why  this system
in particular might be worth investigating: the notion of success or failure of
actions means that, in our fibred categorical semantics, the base category is
order-enriched, and this makes the proof theory quite interesting and, so far,
somewhat unexplored. For example, the Beck-Chevalley condition -- which we will
need for cut elimination --  applies to comma squares rather than, as with the
non-enriched case, to Cartesian squares.

There are, however, also non-technical grounds which make such a project
interesting. The first is that it bears on the semantics of adverbs: adding an
adverb to a verb modifies the success conditions of the action denoted by the
verb (singing in tune, for example, has more restrictive success conditions
than merely singing), and a logic which can handle these success conditions
directly would seem to be important for the semantics of adverbs. Adverbs (or,
more generally, verbal adjuncts) are a large and disparate class
\citep{ernst:synAdj}, and the logic studied here can only handle a small
subclass of these (for example, it can only handle adjuncts whose success
conditions are a subset of the success conditions of the unmodified verb: it
could not deal with an adjunct like `apparently', for example).  However, it
is, at least, a start, and it gives some idea of what a more adequate theory
might look like.

There are also  reasons specific to the case of reasoning about \emph{action}
which make the success or failure of actions an interesting concept. Saying
that actions succeed or fail is an example of \emph{normativity}, that is, of
dividing a set of entities into normal and deviant examples. Normative contexts
can typically not be defined using purely physical vocabulary, and, for this
reason, there has recently been a great deal of philosophical interest in the
commonsense use of normative concepts:  see
\citep{mcdowell:96a,mcdowell1982criteria} and the author's own papers on
normativity in the philosophy of computer science
\citep{WhiteGG:bootNorm,WhiteGG:descAR}.  It is normativity that we are aiming
at in the concept of unreliability of actions: it does not necessarily entail
nondeterminacy, merely some notion of normativity.

Indeed, the concept of success or failure of actions has been recognised as
important from the early days of artificial intelligence: it is usually
referred to, using McCarthy's terminology, as the problem of determining
\emph{action qualifications} \citep{mccarthy62:_towar_abstr_scien_of_comput,%
McCarthyJ:cirfnr,McCarthyJ:epipai,McCarthyJ:appcfc,McCarthyJ:cirfnr} and
\citep[Appendix~B]{reiter01:_knowl_action}.  However, although this concept has
been much discussed in the AI community, the technical results have not been
very illuminating: for example, in Reiter's treatment
\citeyearpar{reiter01:_knowl_action}, the success of actions is represented as
a first-order predicate $\textsc{executable}(\cdot)$ of sequences of actions,
and the special logical role of success and failure does not really come to the
fore. 

When we do develop a formalism in which success and failure play their
appropriate role, we discover important connections with other issues. There is
a long-standing argument, due to Davidson, about the importance of equality in
reasoning about actions \citep{00DavidsonD:logfas}. Our formalism supports
equational reasoning in the appropriate way: this was almost apparent in our
previous paper \citep{white:_david_reiter}, but in this paper the role of
equality becomes more perspicuous. Indeed, one can define an equality predicate
using merely the order-enrichment together with an appropriately structured
category of types (that of a \emph{partial Cartesian category} or
\emph{bicategory of partial maps} \citep{carboni87:_bicat_of_partial_maps}.
Equalities between actions, then, are implicit in the normative concept of
success or failure of actions, together with appropriate and plausible
structure in our type theory.

There is a final and more technical reason for this research, which we
alluded to above. A locally posetal
2-category with comma objects and final object is, in fact, a locally posetal
2-category closed under PIE limits.  These limits \cite[\S6.6]{lack:2catComp}
are in many ways the natural 2-categorical generalisation of finite limits:
just as we showed in \cite{white:_david_reiter} that analogous fibrations over
1-categories with finite limits have cut elimination, so too we can prove here
that 2-fibrations over locally posetal 2-categories with PIE limits have cut
elimination. One could conjecture, then, that 2-fibrations over general
2-categories have cut elimination: a proof of this, however, would require a
certain amount of additional machinery. 

This final reason may be technical, but it is not \emph{merely} technical. As
\cite{sundholm:_proof_theor_meanin} argues, if we regard proof theory as
specifying the meaning of the connectives (that is, if we regard its left and
right rules as a description of the meaning of a connective), then cut
elimination says gives a sort of closure property for these specifications: it
says that no more components of meaning will emerge if we compose the
connectives with cut. Our results say that, provided the base category (i.e.\
the category of actions) is closed under certain limits, then we have cut
elimination. So it says that our logic of actions will have nice closure
properties provided that the actions themselves have suitable closure
properties.  

\subsection{States and Possible Worlds} 

As we have said, our logic will be strongly typed: propositions will have
types, and the types will be the objects of a category, with the category of
propositions fibred over it. 

In this section we describe the intuitive meaning of our fibrations.  As in
\citep{white:_david_reiter}, we will start with Reiter's treatment of action
\citeyearpar{reiter01:_knowl_action}.  His work can be regarded as a
phenomenology of reasoning about action, together with a logical formulation of
that phenomenology: we will retain his phenonomenology, but develop a
formalisation of our own.

Reasoning about action has two sides, which we will, following philosophical
terminology, call the intensional and the extensional.  The intensional side is
the agent's view of actions: what actions are performed, in what sequence, and
so on. It is this view of actions which is sometimes referred to as the
``knowledge level'' \citep{newell}. We can think of this view as giving us a
labelled transition system: the nodes of the system will be called
\emph{states} (in AI terminology, \emph{situations}), the arrows will be, in
philosophical terminology, action \emph{tokens}, and the arrows will be
labelled with action \emph{types} (for the type/token distinction, see
\citep{DavidsonD:inde,davidson80:_event_partic},
\citep{hornsby:_anomal_action,hornsby:_action}, and
\citep{wetzel:_type_token_distin}\nocite{RescherN:esshcg}).  Our actions will
be deterministic -- that is, there will be at most one action token of a given
action type starting from a given state.

However, as well as their intensional aspect, actions also have an effect on
the world. This is the \emph{extensional} side of action and it will be
important also to talk about it: we are concerned about what actually happens
when we act, not merely about the actions that we performed, and so we need to
represent the gap between the intensional and the extensional.  We will
represent the extensional side of actions by propositional assertions about
states.  If, like  Reiter, we use classical logic, ``the way the world is'' can
be described by assigning truth values to propositions: that is, by what is
called, in logical jargon, a \emph{possible world}, and we can, therefore,
think of the effect of an action as a function from possible worlds to possible
worlds. 

Extensions and intensions will be related as follows. States encode intensional
information, and such information will, in general, only yield partial
knowledge of the world: thus, each (intensional) state will, in general,
correspond to several different possible worlds.  However, the agent's
epistemic state will be part of the world, so that each possible world will
correspond to a unique such state.  So, each state will have, associated to it,
a set of possible worlds, and these sets of possible worlds will be disjoint.

Pulling back predicates by these functions will give us a weakest preconditions
map: this is what Reiter calls \emph{regression}.  Reiter also requires that
there should be a Reiter also requires that regression should have a left
adjoint, which he calls \emph{progression}. It solves the problem: given a
transition $\transition{s}{\alpha}{t}$ between situations $s$ and $t$, and
given a proposition $P$ at $t$ -- what Reiter would describe as a \emph{fluent}
-- the regression problem is to find a proposition $P'$ at $s$ which which will
be true iff $P$ is true at $t$.

Reiter also requires a solution to the following problem, which he calls
\emph{progression}: given an action $\transition{s}{\alpha}{t}$, and given a
theory $P$ describing the state $s$, find the theory $Q$ describing $t$.
Regression turns out to be a left adjoint to regression; we will, then, require
that our substitution (or regression) operators should have left adjoints. 

Now actions, as we have said and as \citet{reiter01:_knowl_action} emphasises,
are \label{ref:actPartFun} not usually performable in all circumstances:
furthermore, whether an action is performable or not will, in general, depend
on circumstances unknown to the agent (for example, I may try to open a door,
not knowing whether it is locked or not, or I may try to unlock a door not
knowing that it is not locked).  So whether an action is performable or not is
a matter of the extensional side of things, in which we are representing
actions as functions from possible worlds to possible worlds: and we can
conveniently represent this by having these functions be partially defined. An
action will be performable in precisely those worlds in which the corresponding
function is defined.  We should notice that partiality gives us a partial order
on functions, namely the order given by extension ($f \extendedBy g$ iff $fg$
is defined whenever $f$ is, and, where both defined, $f$ and $g$ agree: think
of the relation between murdering and murdering elegantly). It is this partial
order that we will work with in the remainder of this paper.

This concept of success or failure can, it turns out, be internalised in our
logic: an action will be performable in a situation provide that $\lnot
f^*\lfalse$ is true.  Equality between actions can, likewise, be given a
similar internalisation. We should note, here, that this definition of equality
assumes classical logic: constructively, we do not get equality between
actions, but \emph{apartness} (and the corresponding logic in the fibres is
given by co-Heyting semilattices). 

\section{Cartesian Bicategories and Comma Objects}

\subsection{Outline}

The ultimate goal of this paper is to define a logic whose types and
substitutions come from the objects and 1-cells of a locally posetal base
category, or category of contexts. The semantics of this logic will be a
category fibred over our category of contexts: thus, we will be to investigate
such fibred categories. First, however, we investigate the structure in the
base. 

The appropriate structure on the category of contexts for the case where
actions always succeed seems to be that of  a cartesian category, i.e. a
category with finite limits \citep{white:_david_reiter}: we can construct from
this a locally posetal bicategory by taking its bicategory of partial maps
\citep{robinson88:_categ_of_partial_maps}, and we have argued above that  the
partial order on such a bicategory will give an appropriate notion of success
or failure of actions.  We can characterise these bicategories more abstractly:
\citet{carboni87:_bicat_of_partial_maps} gives conditions for  a locally
posetal bicategory to be the category of partial maps in a cartesian  category. 

So we have two descriptions of a possible base category, one 2-categorical --
as a functionally complete partial Cartesian category, in Carboni's sense
\citeyearpar{carboni87:_bicat_of_partial_maps} -- and one categorical, as a
finite limit category.  The two are naturally related: the finite limit
category, $\ccontT$, is the category of total morphisms of the partial
Cartesian category, $\ccont$, and this induces an equivalence of 2-categories
between, on the one hand, the 2-category of finite limit categories, functors,
and natural transformations, and, on the other hand, the 2-category of partial
cartesian categories, 2-functors, and natural transformations whose components
are total. All of these results are well-known in the literature: I summarise
them in Section~\ref{section:correspondences}.  

Consider now a partial cartesian category  $\ccont$ (or, alternatively, its
category of total morphisms $\ccontT$). We can
\citep[following][]{hermida99:_some_proper_of_fib_as_fibred_categ} define a
notion of 2-fibration over a partial cartesian category: the restriction of a
2-fibration to the subcategory of total morphisms yields a fibration in the
normal sense, and this gives an equivalence of categories between 2-fibrations
over $\ccont$ and fibrations over $\ccontT$. So we can use the theory of
fibrations over $\ccontT$ to guide our investigations of 2-fibrations over
$\ccont$. In particular, we can show that the Frobenius properties correspond
under the equivalence, and that a Beck-Chevalley condition over $\ccontT$
corresponds to a somewhat modified Beck-Chevalley condition over $\ccont$. So
this will give us enough category theory to be able to define our logic and
prove soundness, completeness, and cut elimination.

\subsubsection{Notation}

I have made a few unorthodox choices of notation. Comma objects I write with
$\comma$, because it has an analogous role to $\otimes$: furthermore, it is
probably superior to the standard notations (it is asymmetric, unlike
$\downarrow$, and it is legible, unlike the comma, and it can also be reversed
easily, unlike the comma). We need a Heyting operation on the (distinguished)
subobjects of an object of our categories, and for this I have used
$\heyting{A}{B}$: it is not a wonderful choice, but it can be distinguished
from, for example, $\rightarrow$, which we will also use, but with a different
meaning. 

\subsection{Correspondences between Categories}\label{section:correspondences}

We first describe the correspondences between categories of partial morphisms
in cartesian categories and suitable locally posetal bicategories, known as
\emph{partial cartesian categories}. We will also describe what are known as
\emph{restriction categories} (or, more precisely, restriction categories with
weak products): partial cartesian categories are equivalent to restriction
categories together with appropriately defined finite products
\cite[\S~4.2]{cockett07:_restr_iii}.  We will need, in addition to finite
products, comma objects: restriction categories with comma objects can be
defined in an analogous way. 

We should note that, in this framework, concepts of two different sorts are
represented. The first is the representation of partiality, and the
corresponding partial order between one-cells: the second is the existence of
finite limits of various sorts. Restriction categories enable a conceptually
clean distinction between the two: a restriction category \emph{per se} only
represents partiality, and we can add suitable limits to it if we wish.  We
outline the restriction category framework, and the various equivalences
between categories, in this section.

The usual category-theoretic treatment of partiality is in terms
of spans whose left legs belong to a distinguished 
class of monos, closed under pullback. The relation between
these and restriction categories is as follows. 
\begin{definition}[\protect{\citealt[\S3.2]{cockett02:_restr_i}}]
The 2-category $\mCat$ is defined as follows:
\begin{description}
\item[Objects] are  categories, together with systems $\domains$
of monos containing the identity and closed under composition and pullbacks
\item[1-Cells] are Cartesian functors which respect $\domains$
\item[2-Cells] are natural transformations $\alpha: F \rightarrow G$ such that,
for every $m: A \rightarrow B$ in $\domains$, the following square 
is Cartesian:
\begin{displaymath}
\xy
\square[FA`FB`GA`GB;Fm`\alpha A`\alpha B`Gm]
\endxy
\end{displaymath}
\end{description}
$\mCat$, then, defines categories with a distinguished class of monos.
\end{definition}
\begin{definition}[\protect{\citealt[\S2.1.1]{cockett02:_restr_i}}]
A \emph{restriction category} is a  category together with the assignment, 
to each morphism $f: A \rightarrow B$, of
a morphism $\overline f: A \rightarrow A$ such that
\begin{enumerate}
	\item $f \overline f = f$ for all $f$,
	\item   $f\overline g = g \overline f$ whenever $\dom f = \dom g$ (i.e.
	whenever the composites 
		make sense),
	\item   $\overline{g \overline f} = \overline g \overline f$ whenever
		$\dom f = \dom g$, and
	\item $\overline g f = f \overline {gf}$ whenever $\dom g = \cod f$.
\end{enumerate}
A morphism $f$ in a restriction category is \emph{total} if $\overline f= \id$.
\end{definition}
\begin{definition}[\protect{\citealt[\S2.2.1]{cockett02:_restr_i}}]
	A \emph{restriction functor} is a functor between restriction categories 
	which commutes with restrictions. 
\end{definition}
\begin{definition}[\protect{\citealt[\S2.2.2]{cockett02:_restr_i}}]
The 2-category $\rCat$ is defined as follows:
\begin{description}
	\item[Objects] are restriction categories
	\item[1-cells] are restriction functors
	\item[2-cells] are natural transformations whose components
		are total
\end{description}
\end{definition}
\begin{definition}[\protect{\citealt[\S2.3.3]{cockett02:_restr_i}}]
A morphism $f: A \rightarrow A$ in a restriction category is a \emph{restriction
idempotent} if $f = \overline f$.

A restriction idempotent $f$ is \emph{split} if there are $r: A \rightarrow A_0$
and $i: A_0 \rightarrow A$ with $f = ri$ (in this case $f = \overline r$).

A restriction category is \emph{split} if all of its restriction idempotents
split.
\end{definition}
\begin{definition}[\protect{\citealt[\S2.3.3]{cockett02:_restr_i}}]
The 2-category $\rCatS$ is the full sub-2-category of $\rCat$ whose
objects are split restriction categories. 
\end{definition}
\begin{theorem}[\protect{\citealt[Theorem~3.4]{cockett02:_restr_i}}]
$\rCatS$ and $\mCat$ are 2-equivalent.
\label{theorem:rCatMCatEquiv}
\end{theorem}
\begin{proof}
Define functors 
\begin{align}
	\MCattoSplitR: \mCat & \rightarrow \rCatS\\
	\SplitRtoMCat: \rCatS & \rightarrow \mCat
\end{align}
as follows.

Given a category $\someCat$ together with a stable class of monos $\domains$, define
a  restriction category $\MCattoSplitR(\someCat)$ with the same objects as $\someCat$, whose
morphisms are  spans whose left legs are in $\domains$ up to commuting
isomorphism, and whose restriction sends the span $\langle m, f \rangle$ to
the span $\langle m,m \rangle  $. This assignment can easily be
extended to a 2-functor $\SplitRtoMCat$ from
$\mCat$ to $\rCatS$.

Conversely, given a split restriction category $\ccont$, consider the category $\ccontT$
whose objects are the same as those of $\ccont$ and whose morphisms are the total
morphisms of $\ccont$. The sections of the restrictions of $\ccont$ are total, and
can be shown to form a stable system of monics in $\ccontT$:
this can be shown to extend to a 2-functor $\SplitRtoMCat$ 
from $\rCatS$ to $\mCat$. These
2-functors yield the desired equivalence.
\end{proof}
So far, we have very minimal product structure: only pullbacks of a 
suitable class of monos. Next we shall discuss partial cartesian
categories, which have more product structure. 
\subsubsection{Partial Cartesian Categories}
\begin{definition}[\protect{\citealt{carboni87:_bicat_of_partial_maps}}]
A \emph{partial cartesian category} $\ccont $
is a locally posetal symmetric monoidal bicategory
such that: 
\begin{enumerate}
\item every object $A$ has a unique cocommutative comonoid structure
\begin{equation}
\Delta_A: A \rightarrow A \otimes A
\qquad
!_A: A \rightarrow I
\end{equation}
where $I$ is the monoidal unit, and where $\Delta_A$ is strict natural
and $!_A$ lax natural in $A$. 
\item $\Delta_A$ has a right adjoint $\nabla_A$ such that,
for any $A,B$ and any $f, g: A \rightarrow B$, 
\begin{align}
	\Delta_A \nabla_A &= (\nabla_A \otimes I)(I \otimes \Delta_A)\label{eq:bpmFrob}\\
\nabla_B(f \otimes g) \Delta_A &\extendedBy f
 \label{eq:meetOfMorphisms}
\end{align}
where $\extendedBy$ is the partial order on the homsets of the category. 
\end{enumerate}
\end{definition}
\begin{remark}
The operator on pairs of 1-cells $f,g:A \rightarrow B$ defined by
\eqref{eq:meetOfMorphisms} is, in fact, the meet in the poset 
$\homSet{\ccont }{A}{B}$.
\end{remark}
\begin{definition}
	A 1-cell $f: A \rightarrow B$ in a partial cartesian category is \emph{total} if
	$!_Bf = !_A$.
\end{definition}
\begin{example}
	Let $\someCat$ be a cartesian category, and let $\domains$ be a 
	stable class
	of  monics in $\someCat$ which contains the diagonal
	morphisms $\Delta_A : A \rightarrow A \times A$ for all $A$. 
	Then we define the partial cartesian category 
	$\MCattoBPM(\someCat)$ as follows:
	\begin{description}
		\item[Objects] are those of $\someCat$
		\item[1-cells] are spans in $\someCat$ whose left legs are
			in $\domains$, with composition defined in the
			usual way
		\item [The monoidal structure] is given by $\times$
		\item [The comonoid structure] on an object $A$ is defined as follows:
			\begin{description}
				\item[] $\Delta_A = \langle A, A \rangle: A \rightarrow A \times A$
				\item[] $I$ is the nullary product in $\someCat$, and $!_A$ is the
					unique total morphism $A \rightarrow I$
			\end{description}
		\item[$\nabla$] is defined by the following span:
			\begin{displaymath}
				\begin{xy}
					\morphism/<-/<350,350>[A \times A ` A; \Delta_A]
					\morphism(350,350)<350,-350>[A ` A ; \id]
				\end{xy}
			\end{displaymath}
	\end{description}
\end{example}
\begin{definition}
The 2-category $\bpm$ is defined as follows:
\begin{description}
	\item[Objects] are partial cartesian categories
	\item[1-cells] are monoidal functors  (note that because of
		the uniqueness condition such functors preserve the comonoid
		structure on objects)
	\item[2-cells] are natural transformations whose 1-cells are total
\end{description}
\end{definition}
\begin{definition}[\protect{\citealt[Def.~2.2]{carboni87:_bicat_of_partial_maps}}]
A partial cartesian category is \emph{functionally complete} if 
coreflexives split: that is, if we have $d\extendedBy \id_A: A \rightarrow A$
with $d^2 = d$, then $d = ij$ with $i: A_0 \rightarrow A$, $j: A \rightarrow A_0$,
and $i \dashv j$.  
\end{definition}
\begin{definition}
A 1-cell $f: A \rightarrow B$ in a partial cartesian category is 
\emph{total} if $!_B f = !A$.
\end{definition}
Total maps contain the identities and are closed under composition, so we
have a subcategory, $\ccontT$, of a partial cartesian category $\ccont$. 
\begin{lemma}[\protect{\citealt[Lemma~2.3.i]{carboni87:_bicat_of_partial_maps}}]
If $\ccont$ is a functionally complete bicategory of partial maps,
then $\ccontT$ is cartesian.
\label{lemma:totBCartesian}
\end{lemma}
Because $\ccontT$ is cartesian, we can form its bicategory of
partial maps with respect to the cleass $\domains$ of all
monos: call this
(slightly abusing notation) $\MCattoBPM(\ccontT)$; and we have
\begin{lemma}[\protect{\citealt[Lemma~2.3.ii]{carboni87:_bicat_of_partial_maps}}]
If $\ccontT$ is a functionally complete bicategory of partial
maps, then the natural identity-on-objects functor
\begin{displaymath}
	\ccont \quad \rightarrow \quad \MCattoBPM(\ccontT)
\end{displaymath}
is strictly monoidal and faithful.
\label{lemma:parTotBFaithful}
\end{lemma}

We can (subject to further conditions) prove that this functor
is full: for this we need some more definitions.
\begin{definition}[\protect{\citealt[Def.~2.4]{carboni87:_bicat_of_partial_maps}}]
A 1-cell $f$ in a partial cartesian category is \emph{monic} if 
$\nabla (f \otimes f) = f \nabla$.  
\end{definition}
\begin{lemma}\label{lemma:altMonicDef}
A 1-cell $f: A \rightarrow B$ is \emph{monic} iff, for any $g,h:S \rightarrow A$,
\begin{equation}
	f(g \sqcap h) \quad = \quad (fg) \sqcap (fh).
	\label{eq:altMonicDef}
\end{equation}
\end{lemma}
\begin{proof}
Suppose first that $f$ is monic. We have
\begin{align*}
	(fg) \sqcap (fh) & = \nabla (fg) \otimes (fh) \Delta\\
	& = \nabla (f \otimes f) (g \otimes h) \Delta\\
	& = f \nabla (g \otimes h) \Delta& \text{since $f$ monic}\\
	& = f(g \sqcap h).
\end{align*}

Conversely, suppose that $f$ satisfies \eqref{eq:altMonicDef}. Define the
projections $p_1: A \otimes B \rightarrow A$ and 
$p_2: A \otimes B \rightarrow B$ by
\begin{displaymath}
	\xy
		\morphism<-600,-400>[A \otimes B`A\otimes I;\id \otimes !]
		\morphism(-600,-400)<0,-400>[A \otimes I`A;\sim]
		\morphism|r|<-600,-800>[A \otimes B`A;p_1]
		\morphism<600,-400>[A \otimes B ` I \otimes  B;! \otimes \id]
		\morphism(600,-400)|r|<0,-400>[I \otimes B`B;\sim]
		\morphism|l|<600,-800>[A \otimes B`B;p_2]
	\endxy
\end{displaymath}
Easy calculations show that
$p_i \Delta = \id$, and that
$p_1 \sqcap p_2 = \nabla: A\otimes A \rightarrow A$.
We have	
\begin{align*}
	\nabla (f \otimes f) & = \nabla (f \otimes f) (p_1 \otimes p_2) \\
	&= \nabla (f \otimes f)(p_1 \otimes p_2)
		\Delta_{A} \otimes \Delta_{A}) \\
	&= \nabla (f \otimes f) (p_1 \otimes p_2) 
		(\id_A \otimes \sigma_A 
		\otimes \id_A) \Delta_{A\otimes A}\\
	&= \nabla (f \otimes f) \sigma_A (p_1 \otimes p_2) 
		(\id_A \otimes \sigma_A
		\otimes \id_A) \Delta_{A \otimes A} 
		& \text{by symmetry of $\nabla$}\\
	&=  \nabla (f \otimes f) (p_2 \otimes p_1)
		\sigma_A (\id_A \otimes \sigma_A
		\otimes \id_A) \Delta_{A \otimes A}\\
		&=\nabla (f \otimes f) (p_2 \otimes p_1)
			\sigma_{A\otimes A}\Delta_{A \otimes A}\\
	&= \nabla (f \otimes f) (p_2 \otimes p_1) 
		\Delta_{A \otimes A} \sigma_A  
		& \text{by symmetry of $\Delta$}\\
	&= (f p_2) \sqcap (f p_1) \sigma_A & \text{definition of $\sqcap$}\\
	&= f (p_2 \sqcap p_1 )\sigma_A & \text{by assumption}\\
	&= f (p_1 \sqcap p_2) \\
	&= f \nabla 
\end{align*}

\end{proof}
\begin{corollary}
$f:A \rightarrow B$ is monic in $\ccont$ iff, for any $C$,  the
postcomposition morphism
\begin{displaymath}
	f \compose \cdot : \homSet{\ccont}{C}{A} \rightarrow \homSet{\ccont}{C}{B}
\end{displaymath}
is an inclusion of posets.
\label{cor:monicYoneda}
\end{corollary}

\begin{lemma}
[\protect{\citealt[Lemma~2.5]{carboni87:_bicat_of_partial_maps}}]
If $\ccont$ is a partial cartesian category, and if $\ccontT$ is its subcategory
of total morphisms, then a 1-cell in $\ccontT$ is a mono in $\ccontT$ iff
it is monic in $\ccont$. 
\label{lemma:pccMonoMonicCorrespondence}
\end{lemma}

\begin{definition}[\protect{\citealt[Def.~2.4]{carboni87:_bicat_of_partial_maps}}]
A \emph{quasi-inverse} for a monic $f$ is a 1-cell $f^{\dag}$ such
that 
\begin{displaymath}
	\dom f = f^{\dag} f \quad \text{and} \quad \dom{ f^{\dag}} = ff^{\dag} 
\end{displaymath}
\end{definition}

\begin{lemma}
[\protect{\citealt[Lemma~2.5]{carboni87:_bicat_of_partial_maps}}]
Quasi-inverses are unique, and, if $i \dashv j$ is the 
splitting of a coreflexive in $\someCat$, then
$j = i^{\dag}$.
\label{lemma:quasiInverseProps}
\end{lemma}
\begin{definition}
The two-category $\bpms$ is the full sub-two-category of $\bpm$ given
by partial cartesian categories all of whose coreflexives split.
\end{definition}
\begin{proposition}
The two-category $\bpms$ is 2-equivalent to $\mCat$.
\end{proposition}
\begin{proof}
Define 2-functors in both directions as follows.

Given a partial cartesian category $\ccont$ with split coreflexives
the 1-category $\ccontT$ is cartesian, and the class of
morphisms 
\begin{displaymath}
	\left\{ j | i \vdash j \text{ the splitting of a coreflexive}\right\}
\end{displaymath}
is a class of monos of $\ccontT$ closed under pullback and
containing the identities. We have then an object of 
$\mCat$: we can check that this assignment is, in fact, 2-functorial.
Call this 2-functor
\begin{displaymath}
	\BPMtoMCat: \bpms \quad \rightarrow \quad \mCat
\end{displaymath}

Given an object $\langle\ccont  , \mathcal M\rangle$ of $\mCat$, 
we define an object 
of $\bpms$ as follows:
\begin{description}
	\item[Objects]\ are objects of $\ccont $,
	\item[1-cells]\ are spans in $\ccont $ whose 
		left legs are in $\mathcal M$, up to the
		usual equivalence relation, and
	\item[2-cells]\ are defined by inclusion of subobjects 
		in $\ccont $: that is, $\langle j, f\rangle \extendedBy \langle j', f'\rangle$
		iff there is a commuting diagram
		\begin{displaymath}
			\xy
			\morphism/{<-}/<200,200>[A`A_0;j]
			\morphism|b|/{<-}/<200,-200>[A`A_0';j']
			\morphism(200,200)<300,-200>[A_0`B;f]
			\morphism(200,-200)|b|<300,200>[A_0'`B;f']
			\morphism(200,200)/{..>}/<0,-400>[A_0`A_0';{}]
			\endxy
		\end{displaymath}
\end{description}
The monoidal structure on this 2-category is given by $\times$ on
objects (we need stability of $\mathcal M$ under pullbacks to
make it functorial). The conditions on the tensor product are readily
checked, as is the 2-functoriality of this assignment of an object in
$\bpms$ to an object in $\mCat$.
Call this 2-functor
\begin{displaymath}
	\MCattoBPM:\: \mCat \rightarrow \bpms
\end{displaymath}

Finally we need to check that these two 2-functors give a 2-equivalence
of categories between $\bpms$ and $\mCat$. 
\end{proof}

\subsubsection{Restriction Products}
Finally we have a characterisation of partial cartesian categories in terms of
restriction categories and suitably defined products.
\begin{definition}[\protect{\citealt[\S~4.1]{cockett07:_restr_iii}}]
	Define the two-category $\rCatl$ as follows:
	\begin{description}
		\item[Objects] are  restriction categories
		\item[One-cells] are restriction functors
		\item[Two-cells] are lax natural transformations with total components: that is,
			a natural transformation from $F: \ccont \rightarrow \ccont'$ to
			$G: \ccont \rightarrow \ccont'$ is a family of total 1-cells 
			$\alpha_X: F(X) \rightarrow G(X)$ such that, for $f: X \rightarrow Y$, we have
			\begin{displaymath}
				\begin{xy}
					\square[F(X) ` F(Y) ` G(X) ` G(Y) ; 
					Ff ` \alpha_X ` \alpha_Y ` Gf]
					\place(250,250)[\extends]
				\end{xy}
			\end{displaymath}

	\end{description}
\end{definition}

\begin{definition}
	\begin{enumerate}
		\item
	A \emph{binary restriction product} on a restriction category $\ccont$ is a
	functor $\otimes: \ccont \times \ccont \rightarrow \ccont$ right adjoint to
	$\Delta: \ccont \rightarrow \ccont \times \ccont$ in $\rCatl$. 
\item A \emph{restriction terminal object} on a restriction category $\ccont$ is
	an object I (i.e. a functor from the terminal restriction category $\mathfrak{I}$ to $\ccont$)
	which is right adjoint to the unique functor $\ccont \rightarrow \mathfrak{I}$.
\item A restriction category \emph{has restriction products} if it has 
	binary restriction products and a restriction terminal object. 
	\end{enumerate}
\end{definition}
Then we have:
\begin{theorem}(Cockett, Lack, Robinson et al.)
	A partial cartesian category is a restriction category with
	restriction products.
	\label{theorem:bpmRcRP}
\end{theorem}
\begin{proof}
	See \citet[\S~4.2]{cockett07:_restr_iii}.
\end{proof}
\subsection{Weak Comma Objects}\label{sec:weakComma}
We can now start on the material specific to this paper. 
In the total case (i.e.\ when the base is a category rather than a locally
posetal two-category) we  need conditions on the base -- namely the
existence of fibred products --
in order to prove cut elimination, together with conditions conditions on the fibration, 
known as the Beck-Chevalley conditions \citep[see][]{white:_david_reiter}. 
In the locally posetal case,
we again need a Beck-Chevalley condition, and, 
as \citet{Herm04:Descefibrastronregulcateg} 
shows, in the bicategorical case we need to formulate these
conditions with 
comma objects rather than fibre products.

In our case, we define comma objects as follows.
\begin{definition}
A \emph{bicategory of partial maps with weak comma objects} is a bicategory of partial
maps such that any diagram of the form 
\begin{displaymath}
\xy
\morphism|b|<300,-300>[A`B;f]
\morphism(600,0)|b|<-300,-300>[C`B;g]
\endxy
\end{displaymath}
can be completed to a diagram 
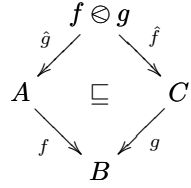
\begin{displaymath}
\xy
\morphism|b|<300,-300>[A`B;f]
\morphism(600,0)|b|<-300,-300>[C`B;g]
\morphism(300,300)<-300,-300>[f\comma g`A;\hat g]
\morphism(300,300)<300,-300>[f \comma g`C;\hat f]
	\place(300,0)[\extendedBy]
	\endxy
\end{displaymath}
with $\hat f$ and $\hat g$ total, and such that, for any 
diagram 
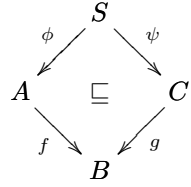
\begin{displaymath}
\xy
\morphism|b|<300,-300>[A`B;f]
\morphism(600,0)|b|<-300,-300>[C`B;g]
\morphism(300,300)<-300,-300>[S`A;\phi]
\morphism(300,300)<300,-300>[S`C;\psi]
	\place(300,0)[\extendedBy]
	\endxy
\end{displaymath}
there is a unique mediating arrow $\langle \phi, \psi \rangle: S \rightarrow f \comma g$
such that $ \hat g \langle \phi, \psi \rangle = \phi \overline \psi$ and
$\hat f \langle \phi \psi \rangle = \psi \overline \phi$. 
\end{definition}

The following is immediate:
\begin{definition}
A \emph{restriction category with weak comma objects} is a restriction 
category such that, for any morphisms $f: A \rightarrow B$ and
$g: C \rightarrow B$, there is an object $f \comma g$ with 
morphisms $\hat g: f \comma g \rightarrow A$ and 
$\hat f: f \comma g \rightarrow C$ such that:
\begin{enumerate}
	\item $\hat g$ and $\hat f$ are total
	\item $f \hat g = g \hat f\, \overline{f \hat g}$
	\item if we have $\phi: S \rightarrow A $ and
		$\psi: S \rightarrow C$, then there is 
		a unique $\langle \phi, \psi\rangle: S \rightarrow f \comma g$
		such that 
		\begin{align*}
			\hat g \langle \phi, \psi \rangle &= \phi \overline \psi\\
			\hat f \langle \phi \psi \rangle &= \psi \overline \phi
		\end{align*}
\end{enumerate}
\end{definition}
The proof of the following is elementary:
\begin{proposition}
\begin{enumerate}
	\item
		Comma objects are unique up to canonical isomorphism
	\item	
		Given pairs $\phi, \psi$ and $\phi', \psi'$ as above, we have
		$\langle \phi , \psi \rangle = \langle \phi', \psi'\rangle$ iff
		\begin{align*}
			\phi\overline{\psi} &= \phi'\,\overline{\psi'}\\
			\psi\overline{\phi} & = \psi'\, \overline{\phi'}
		\end{align*}
\end{enumerate}
\end{proposition}
We should note also the following (which is likewise elementary):
\begin{proposition}
	The following are equivalent: 
	\begin{enumerate}
    \item $f \comma g$ is a comma object of $f$ and $g$
		\item For any \emph{total} $\phi: S \rightarrow A$
			and $\psi: S \rightarrow B$ such that $f \phi \extendedBy g \psi$, there
			is a unique $\langle \phi,\psi \rangle: S \rightarrow f \comma g$
			with the usual properties.
	\end{enumerate}
\end{proposition}
This entails that comma objects are, indeed, comma objects in the usual sense (that is,
comma objects defined by weighted limits \citep{lack:2catComp}: more precisely,
\begin{proposition}
  $f \comma g$ is the weighted limit of the
	diagram 
	\begin{displaymath}
		\xy
		\morphism|b|<300,-300>[A`B;f]
		\morphism(600,0)|b|<-300,-300>[C`B;g]
		\endxy
		\quad 
		\text{with weight}
		\quad
		\xy
		\morphism|b|<300,-300>[1` 2;0]
		\morphism(600,0)|b|<-300,-300>[1`  2;1]
		\endxy
	\end{displaymath}
\end{proposition}
\subsection{Local Conditions for Comma Objects}
We can give local conditions for a bicategory of partial maps to
have comma objects; for convenience, we will do
this in the restriction category case.
	First we need to fix some vocabulary.
\begin{definition}
Let $A$ be an object of a  restriction category.
We say that $A$ has a \emph{stable Heyting operation} if,
for any two restriction idempotents $\alpha, \beta$ of $A$,
there is a restriction idempotent $\heyting{\alpha}{\beta}$ such
that, for any $f: S \rightarrow A$, and any restriction idempotents
$\alpha$, $\beta$ and $\gamma$ on $A$, we have
\begin{displaymath}
	\alpha f  \extendedBy  \beta f 
	\quad\text{iff}\quad
	f \extendedBy \heyting {\alpha}{\beta} f 
\end{displaymath}
\end{definition}
\begin{lemma}\label{lemma:commasubtensor}
Let $f: A \rightarrow B, g: C \rightarrow B$. 
The canonical morphism 
\begin{displaymath}
	f \comma g \quad \rightarrow \quad A \otimes B
\end{displaymath}
is monic.
\label{lemma:commaMonic} 
\end{lemma}
\begin{proof}
Apply Lemma~\ref{lemma:altMonicDef} to the universal characterisations
of $A \otimes B$ and $f \comma g$. 
\end{proof}
\begin{theorem}
Let $\ccont $ be a restriction category with restriction
products and terminal object whose
restriction idempotents split and whose monics
have quasi-inverses. The following are equivalent:
\begin{enumerate}
	\item $\ccont $ has weak comma objects
	\item the objects of $\ccont$ have stable Heyting operations
		on 
		their restriction idempotents
\end{enumerate}
\label{theorem:commaStableHeyting}
\end{theorem}
\begin{proof}
We first show that $1 \Rightarrow 2$. Let $\alpha$ and $\beta$ be restriction
idempotents on $A$: consider the comma 
object $\alpha \comma \beta$. The
natural morphism $i: \alpha \comma \beta \rightarrow
\alpha \otimes \beta$ is monic, by Lemma~\ref{lemma:commasubtensor}: let $j$ be its quasi-inverse. 
Then $ij: A\otimes A \rightarrow A \otimes A$ is a restriction idempotent:
let $\heyting{\beta}{\alpha} = \nabla_A ij \Delta_A: A \rightarrow A$. 
We can verify that it is a restriction idempotent. We can also
verify that, for any $f: S \rightarrow A$ for which $\langle f, f \rangle$ 
is defined,
$j \Delta_A f = \langle f, f \rangle$.

Now let $f$ be any morphism $S \rightarrow A$.
We have
\begin{align*}
	\alpha f & \extendedBy \beta f & \text{iff}\\
	\langle f, f \rangle: &S \rightarrow \alpha \comma \beta & \text{is defined, iff}\\
	\heyting{\beta}{\alpha} f & = \nabla ij \Delta f\\
	&= \nabla i \langle f,f \rangle\\
	& = \nabla \Delta f&\text{by the naturality of $i$} \\
	& = f& \text{that is, $f \heyting{\beta}{\alpha} =f$}
\end{align*}
and this is the property defining a stable Heyting operation.

To show that $2 \Rightarrow 1$, we proceed as follows. If the objects of $\ccont $ have
stable Heyting operations, consider $f: A \rightarrow B, g : C \rightarrow B$.
Define the following restriction idempotents on $A \otimes C$:
\begin{align*}
	\alpha & = \overline{f \pi_1}\\
	\beta & = \overline {\nabla_C f \otimes g}
	\intertext{Consider the restriction idempotent $\heyting{\beta}{\alpha}$: choose a splitting
	of the form} 
	&\bfig \morphism(0,0)|a|/@<5pt>/<500,0>[f \comma g`A \otimes B;i] 
	\morphism(500,0)|b|/@<5pt>/<-500,0>[A \otimes B ` f \comma g;j]\efig
	\intertext{Then $f \comma g$ will be our comma object. So we have, for any
	$\phi: S \rightarrow A, \psi: S \rightarrow C$,}
	\langle \phi, \psi\rangle & : S \rightarrow f \comma g &\text{is defined iff}\\ 
	\phi \otimes \psi \Delta_S & : S \rightarrow A \otimes C&\text{factors through $i$, iff}\\
	\phi \otimes \psi \Delta_S & = \heyting {\beta}{\alpha} \phi \otimes \psi \Delta_S&\text{iff}\\
	\alpha \phi \otimes \psi \Delta_S & \extendedBy \beta \phi \otimes \psi \Delta_S,&\text{that is}\\
	\overline{f \pi_i} \phi\otimes \psi \Delta_S & \extendedBy \overline{\nabla_C f\otimes g} \phi \otimes \psi \Delta_S&\text{iff}\\
	\phi \otimes \psi \Delta_S \overline{f \pi_1 \phi \otimes \psi \Delta_S} & \extendedBy 
	\phi \otimes \psi \Delta_S \overline{\nabla_C f \phi \otimes g \psi\Delta_S}& \text{iff}\\
	\phi \otimes \psi \Delta_S \overline{f \phi}\,\overline{\psi} &
	\extendedBy \phi \otimes \psi \Delta_S \overline{f \phi \sqcap g \psi}&\text{iff}\\
	(\phi \overline{f \phi}\,\overline{\psi})\otimes (\psi \overline{f \phi}\,\overline{\psi})\Delta_S & 
	\extendedBy 
	(\phi \overline{f\phi \sqcap g \psi}) \otimes (\psi \overline{f \phi \sqcap g \psi})\Delta_S
\end{align*}
It is easy to prove that this last inequality holds if $f \phi \extendedBy g \psi$; to prove the converse, we proceed as
follows. Applying $\pi_1$ and $\pi_2$ to both sides, and cancelling
a number of factors of the form $\overline{\psi}, \overline{\psi}$ have that the last line holds only if
\begin{align*}
	\phi \overline{f \phi}\, \overline{\psi} &\extendedBy \phi \overline{f\phi \sqcup g \psi} 
	&\text{and} && \psi \overline{f \phi}\,\overline{\psi} & \extendedBy \psi \overline{f \phi \sqcup g \psi}
	\intertext{but the converse containments clearly hold, so we have equalities. These equalities yield, after
	composition with $f$ and $g$ as appropriate, and using the identity $f \phi \overline{f \phi \sqcap g \psi} 
	g \psi \overline{f \phi \sqcap g \psi}$,}
	f\phi \overline {f \phi}\, \overline{\psi} & =  
	f \phi \overline{f\phi \sqcap g \psi} & =&& 
	g \psi \overline{f \phi \sqcap g \psi} & =
	g \psi \overline {f \phi}\, \overline{\psi}
	\intertext{This finally yields, as required,}
	&& f \phi \overline{\psi}  = g \psi \overline{f \phi}.
\end{align*}
\end{proof}
\begin{corollary}\label{prop:commaObjectCharacterisation}
Let $\ccont $ be a restriction category with 
a restriction final object whose
monics have quasi-inverses and whose restriction idempotents split.
The following
are equivalent:
\begin{enumerate}
	\item $\ccont$ has weak comma objects\label{cond1}
	\item $\ccont$ has weak comma objects and a binary restriction
		product\label{cond1a}
	\item the objects of $\ccont$\label{cond2}
		have stable Heyting operations on their restriction idempotents
\end{enumerate}
\end{corollary}

\begin{proof}
	The equivalence of \ref{cond1a} and \ref{cond2} follows from Theorem~\ref{theorem:commaStableHeyting}.
	The equivalence of \ref{cond1} and \ref{cond1a} follows from the fact that,
	in the presence of a restriction final object, $!_A \comma !_B$ is
	the binary restriction product of $A$ and $B$.
\end{proof}
\subsection{Pasting for Comma Squares}
We will need to paste comma squares: the following lemmas say
that we can do so. Note that, because of the asymmetry of a
comma square, there are two cases.
\begin{lemma}
Consider the following comma square:
\begin{displaymath}
	\begin{xy}
	\square[f\comma g ` C ` A ` B; \hat f ` \hat g ` g ` f]
	\place(250,250)[\extendedBy]
	\end{xy}
\end{displaymath}
Let $S$ be an object: then, in this diagram of posets and poset morphisms,
\begin{displaymath}
	\begin{xy}
		\morphism/{@<6pt>}/<1500,0>[\homSet{}{S}{f \comma g}` 
		\{ \phi: S \rightarrow A, \psi: S \rightarrow C | f \phi \extendedBy g \psi \}; \textsf{unpair}]	
		\morphism|b|/@{<-}@<-6pt>/<1500,0>[\homSet{}{S}{f \comma g}` 
		\{ \phi: S \rightarrow A, \psi: S \rightarrow C | f \phi \extendedBy g \psi \}; \textsf{pair}]	
	\end{xy}
\end{displaymath}
	where 
	\begin{displaymath}
		\textsf{unpair}(\chi) = (\hat g \chi, \hat f \chi)\quad \textrm{and} \quad
		\textsf{pair} (\phi, \psi) = \langle \phi, \langle \psi,
	\end{displaymath}
	we have:
	\begin{enumerate}
		\item $\textsf{unpair} \dashv \textsf{pair}$, with $\textsf{pair} \, \textsf{unpair} = \id$ 
			and $\textsf{unpair}\, \textsf{pair} \extendedBy \id$
		\item $\homSet{ }{S}{f \comma g} \cong \textsf{Im}(\textsf{unpair}) = \{\phi, \psi| f \phi = g \psi, \overline{\phi} = \overline{\psi}\}$
	\end{enumerate}
\label{lemma:whatComObjReps}
\end{lemma}
\begin{proof}
	The first part is an immediate consequence of the definition of comma objects: the second
	part follows from the first.
\end{proof}
\begin{lemma}
In the following diagram the outer rectangle is a comma square.
\label{lemma:pasting1}
\begin{displaymath}
	\begin{xy}
	\square(500,0)[f\comma g ` C ` A ` B; \hat f ` \hat g ` g ` f]
	\place(750,250)[\extendedBy]
	\square[h \comma \hat g` f \comma g ` D ` A; \hat h ` \tilde g ` ` h]
	\place(250,250)[\extendedBy]
	\end{xy}
\end{displaymath}
\end{lemma}
\begin{proof}
	We use Lemma~\ref{lemma:whatComObjReps}. 
	\begin{align*}
		\homSet{ }{S}{h \comma \hat g} \: & \cong \:
		\{\phi, \psi\,|\, h \phi \extendedBy \hat g \psi,\, \overline{\phi} = \overline{\psi} \}\\
		&=\: \{\phi, \psi_1, \psi_2\, |\,
		f \psi_1 \extendedBy g \psi_2,\, \overline{\psi_1} = \overline{\psi_2},\,
		h \phi \extendedBy \hat g \langle \psi_1, \psi_2\rangle,\,
		\overline{\phi}= \overline{\langle \psi_1,\psi_2 \rangle}\}\\
		&=\:\{\phi, \psi_1, \psi_2\, | \,
		\overline {\phi} = \overline{\psi_1} = \overline{\psi_2},\,
		f \psi_1 \extendedBy g \psi_2,\, 
		h \phi \extendedBy \psi_1\}
		\intertext{since, for any $\psi_1, \psi_2$, $\overline{\langle \psi_1, \psi_2\rangle} = \overline{\psi_1}\, \overline{\psi_2}$,
		and since $\hat g \langle \psi_1, \psi_2\rangle = \psi_1$}
		&=\: \{\phi, \psi_1, \psi_2 \,|\, 
		\overline{\phi} = \overline{\psi_1} = \overline {\psi_2},\,
		f \psi_1 \extendedBy g \psi_2,\, 
		h \phi = \psi_1\}
		\intertext{since we have $\overline {\psi_1}\, =\, \overline {\phi}\, \extendedBy \,\overline{h \phi}\, \extendedBy \,\overline{\psi_1}$,
		and so $\overline {h \phi } = \overline {\psi_1}$, so, since $h \phi \extendedBy \psi_1$, $h \phi = \phi_1$}
		&=\: \{ \phi, \psi_2 \,|\,
		\overline{\phi} = \overline {\psi_2},\, fh\phi \extendedBy g \psi_2\}\\
		& \cong \: \homSet{ } {S}{(fh) \comma g}
	\end{align*}
	and so (by a Yoneda argument) the natural morphism 
	\begin{displaymath}
		h \comma \hat g \quad \rightarrow \quad (fh) \comma g
	\end{displaymath}
	is an isomorphism (we have to check that the natural morphism induces the
	isomorphism of homsets which the argument above yields, but this is
	trivial). 
\end{proof}
\begin{lemma}
	In the following diagram, the outer rectangle is a comma square.
	\begin{displaymath}
		\begin{xy}
			\square[f\comma g` C ` A ` B; \hat f  ` \hat g ` g ` f]
			\place(250,250)[\extendedBy]
			\square(0,500)[\hat f \comma h` D ` f \comma g ` C;\tilde f` \hat h ` h ` ]
			\place(250,750)[\extendedBy]
		\end{xy}
	\end{displaymath}
	\label{lemma:pasting2}
\end{lemma}
\begin{proof}
	Note first that, since $\hat f$ and $\hat h$ are total, the top square
	is commutative. 
	We use, again, Lemma~\ref{lemma:whatComObjReps}. 
	\begin{align*}
		\homSet{ }{S}{\hat f \comma h} \:& \cong \: \{\phi, \psi\,|\, 
		\overline {\phi} = \overline{\psi}, \, 
		\hat f \phi \extendedBy h \psi
		\}\\
		&=\: \{ 
		\phi, \psi| \,
		\overline {\phi} = \overline {\psi},\,
		\hat f \phi = h \psi\}
		\intertext{since $\hat f$ is total and $\overline{\phi} = \overline{\psi}$}
		&=\: \{\phi_1, \phi_2, \psi \,|\, \overline {\phi_1} = \overline{\phi_2}, \, f \phi_1 \extendedBy g \phi_2,
		\hat f \langle \phi_1, \phi_2\rangle = h \psi,\,
		\overline {\langle \phi_1, \phi_2\rangle} = \overline{\psi}
		\} \\
		&=\: \{\phi_1, \phi_2, \psi \,|\, \overline {\phi_1} = \overline{\phi_2} = \overline{\psi}, \, f \phi_1 \extendedBy g \phi_2,
		 \phi_2 = h \psi
		\}\\
		&=\: \{\phi_1,  \psi \,|\, \overline {\phi_1}  = \overline{\psi}, \, f \phi_1 \extendedBy gh\psi 
		\}\\ 
		&\cong\: \homSet{ }{S}{f\comma(gh)}
	\end{align*}
	and the result follows as in the previous lemma.
\end{proof}
\begin{corollary}
	Consider the following diagram, where $f$ is total and where $i_C: C_0 \rightarrow C$ is
	the inclusion of the domain of $g$.
	\begin{displaymath}
		\begin{xy}
			\square(0,500)<650,500>[(f \comma g)\times_C C_0 ` C_0 ` f \comma g ` C;
			\tilde f ` \hat i_C ` i_C ` ]
			\square<650,500>[f \comma g` C ` A ` B; \hat f ` \hat g ` g ` f]
		\end{xy}
	\end{displaymath}
	Then  the top left hand object is, in fact, $A \times_B C_0$, and 
	$\hat i_C$ is an isomorphism
\end{corollary}
\begin{proof}
	Since $f$ and $\hat g$ are total, the bottom square is commutative. 
	By pasting, the outer rectangle is a comma square: since $f$ and $g i_C$ are
	both total, this square is, in fact, Cartesian. Consequently,
	$(f \comma g) \times_C C_0$ is, in fact, $A \times_B C_0$, and thus
	classifies pairs of maps $\phi, \psi$ such that $\dom(\phi) = \dom(\psi)$ 
	and $f \phi = g i_C \psi$; but this is exactly what $f \comma g$ classifies,
	since, if we have a pair $\phi, \psi'$ such that $\dom(\phi) = \dom (\psi')$
	and $f \phi = g \psi'$, then, since $f$ is total, $\psi'$ must factor
	through $i_C$ (uniquely, since $i_C$ is monic). So we have the result by 
	Yoneda.
\end{proof}
\begin{corollary}
Let $f: A\rightarrow B$, and let  $i: A_0 \rightarrow A$
be the inclusion of the domain of $f$. Let $g: C \rightarrow B$,
and let $j: C_0 \rightarrow C$ be
the inclusion of the domain of $g$ in $C$. Then the following diagram commutes:
the left square is Cartesian, the right square is comma,  the
enclosing rectangle is Cartesian, and $\hat i_A$ is the natural
morphism $A_0 \times_B C_0 \rightarrow f \comma g$:
\begin{displaymath}
	\begin{xy}
		\square[A_0 \times_B C_0 ` f \comma g ` A_0 ` A ; \hat i_A ` \tilde g ` \hat g ` i_A]
		\place(250,250)[=]
		\place(750,250)[\extendedBy]
		\square(500,0)[f \comma g ` C ` A ` B; \hat f ` ` g ` f]
	\end{xy}
\end{displaymath}
\label{prop:decompCommaSquare}
\end{corollary}
\begin{proof}
	Construct the diagram as follows:
\begin{displaymath}
	\bfig
	\square(650,0)[{f \comma g}`C`A` B ; \hat f ` \hat g ` g ` f]
	\place(325,250)[=]
	\place(900,250)[\extendedBy]
	\square<650,500>[A_0 \times_A (f \comma g) ` f \comma g ` A_0 ` A; \hat i ` \hat {\hat g} `  ` i]
	\efig
\end{displaymath}
where the right hand square is a comma square and the left hand square is
Cartesian ($i$ and $\hat g$ are total, so this makes sense). By pasting,
and since $fi_A$ is total, we can apply the previous corollary
and identify the top left object with $A_0 \times_B C_0$.
\end{proof}

\section{Fibrations over Cartesian Bicategories}
Having defined our base categories (that is, partial
cartesian categories with weak comma objects), 
we will now define the notion
of a 2-fibration.
\begin{definition}
An \emph{posetal 2-fibration} is a 2-functor
\begin{displaymath}
\pi: \econt \rightarrow \ccont
\end{displaymath}
such that 
\begin{enumerate}
\item $\econt$ and $\ccont$ are locally posetal
2-categories (i.e.\ 2-categories such that the homsets between
objects are partial orders)
\item the fibres of $\pi$ are posets, with trivial
two-cells
\item the 1-cells of $\econt$ are fibred over the 1-cells of
$\ccont$ in the standard sense, and 
\item if $P, Q \in \ob(\econt)$, then,
for all 1-cells $f,g: P \rightarrow Q$,
$f \extendedBy g$ iff $\pi(f) \extendedBy \pi(g)$, and
\item\label{fibHomSets} if $P, Q \in \ob(\econt)$,  $g: P \rightarrow Q$,
and if $\alpha \extendedBy \pi(f)$, then there is
$f \extendedBy g$ with $\pi(f) = \alpha$. 
\end{enumerate}
\end{definition}
\begin{remark}
The fibrational conditions of this definition come from
\citet[Theorem~2.8~(iii)]{hermida99:_some_proper_of_fib_as_fibred_categ},
with considerable simplifications because of our posetal case.
\end{remark}

\subsection{Correspondences between Fibrations}
We have described, in Section~\ref{section:correspondences}, the
basic correspondences between partial cartesian categories and cartesian
1-categories. We will now show how this correspondence yields correspondences
between 1-fibrations in posets over cartesian categories and posetal 2-fibrations
over partial cartesian categories. This correspondence will have
two ingredients: firstly a Grothendieck correspondence for posetal 2-fibrations, and secondly
a result of Hermida which gives a universal characterisation of the
construction of partial cartesian categories from cartesian 1-categories. 
\subsubsection{The Grothendieck correspondence}
\begin{lemma}\label{lemma:magicPullback}
Let $\pi: \econt \rightarrow \ccont$ be a posetal 2-fibration, and
let $Q \in \ob(\econt)$, $\alpha \extendedBy \beta: A \rightarrow \pi(Q)$. 
Then, in the poset $\econt$, $\beta^* Q \leq \alpha^* Q$.
\end{lemma}
\begin{proof}
Because of the fibration of 1-cells, there is a map $\hat\beta:
\beta^* Q \rightarrow Q$ over $\beta$. By condition~\ref{fibHomSets}
above, there is a morphism $f: \beta^*Q \rightarrow Q$ over
$\alpha$: a vertical-horizontal factorisation of $f$ gives the
result.
\end{proof}
\begin{lemma}\label{lemma:GrothConst}
Let $\ccont$ be a bicategory of partial maps. The 2-category of
2-fibrations in
posets over $\ccont$ is equivalent to the 2-category 
of strict functors
$\ccont\rightarrow \textsf{Poset}^{\textrm{coop}}$
\end{lemma}
\begin{proof}
This is basically the Grothendieck correspondence. Consider first a 
fibration $\pi: \econt \rightarrow \ccont$. Choose 
a cleavage of $\pi$: a 1-cell $f:A \rightarrow B$ of $\ccont$
then gives a poset morphism $f^*: \econt_B \rightarrow \econt_A$.
Composition is strict (i.e. $f^*g^* = (gf)^*$) because the vertical
structure in the fibres is posetal. Lemma~\ref{lemma:magicPullback}
gives us the 2-cells. So, given a fibration, we have a functor.

Conversely, given a functor $\mathcal{F}$, we define a bicategory as follows:
\begin{description}
\item[Objects]\ are pairs $\langle A, P\rangle$, where
$A$ is an object of $\ccont$ and $P$ is an element of
the poset $\mathcal{F}(A)$
\item[1-cells]\ between $\langle A, P\rangle$ and
$\langle B, Q \rangle$ are 1-cells $f:A \rightarrow B$
such that $P \leq f^* Q$
\item[2-cells]\ $f\extendedBy g: \langle A, P \rangle\rightarrow  \langle B, Q\rangle$ 
iff $f \extendedBy g: A \rightarrow B$.
\end{description}
It is straightforward to check that this gives an equivalence of 2-categories. 
\end{proof}

\subsubsection{The Hermida characterisation}\label{sec:hermidaCorrespondence}
Let $\ccont$ be a bicategory of partial maps, and consider a 2-fibration
(in posets, let us say)
$\econt \rightarrow \ccont$ over it: then it is easy to
check that $\econt$ restricts to a fibration over $\ccontT$, 
the subcategory of total morphisms of $\ccont$. The goal of this
section is to show that, subject to mild conditions, this
process can be inverted. 
We prove this using
a result of Hermida \citeyearpar{Herm02:categoutlorelatmodalsimul}. First
some lemmas:
\begin{lemma}\label{lemma:pullbackStuff}
Let $\ccont$ and $\ccontT$ be as above, and suppose that
coreflexives in $\ccont$ split:  let $\domains$ be
the set of monos in $\ccontT$ which split 
coreflexives in $\ccont$. Let $\alpha$ be a coreflexive, and let
$i \vdash j$ split $\alpha$. Let 
\begin{displaymath}
\xy
\square<550,350>[A_0 ` B_0 ` A ` B; \hat f ` \hat i ` i ` f]
\endxy
\end{displaymath}
be a pullback in $\ccontT$: $i \in \domains$, 
and so  $\hat i \in \domains$ and thus it has a right
adjoint $\hat j$. 
Then
\begin{enumerate}
	\item 
the following diagram commutes in $\ccont$:
\begin{equation}
	\xy
	\square/->`<-`<-` ->/<550,350>[A_0 ` B_0 ` A ` B ;\hat f ` \hat j ` j ` f]
\endxy\label{eq:bcra}
\end{equation}
and
\item we have
	\begin{displaymath}
		\overline{ij f} \quad = \quad \hat i \hat j.
	\end{displaymath}
	\end{enumerate}
\end{lemma}
\begin{proof}
First note that we can rewrite the universal property of the Cartesian square 
as follows:
\begin{quote}
	Let $f:S \rightarrow A$ be total: then, if $ijf\phi = f\phi$,
	$\hat i \hat j \phi = \phi$. 
\end{quote}
Now consider the reflection idempotent $\overline{ijf}$; let $i' \dashv j'$
split it. By properties of restriction idempotents, 
we have $\overline{ijf}i' = i'$, which is total, so that 
\begin{align*}
	f\overline{ijf}i' &= \overline{ij}fi'\\
	& = ijfi'
	\intertext{is also total: but $ijfi' \extendedBy fi'$, so}
	ijfi' &= fi'
	\intertext{and thus, by the universal property}
	\hat i\hat ji' &= i', & \text{so}\\
	\hat i\hat j i'j' &= i'j', \text{i.e.}\\
	i'j'&\extendedBy \hat i \hat j. 
	\intertext{On the other hand,}
	f & \extends ijf\\
	& = i \hat f \hat j& \quad\text{so, since $i \dashv j$,}\\
	jf & \extends \hat f \hat j, &\text{and thus}\\
	ijf & \extends i \hat f \hat j, &\text{so}\\
	i'j' = \overline{ijf} & = \overline{i \hat f \hat j}\\
	& = \overline{\hat j} & \text{since $\hat f$ and $i$ are total}\\
	& = ij
\end{align*}
Consequently, $\hat i \hat j = i' j' = \overline{ijf}$, which proves the second part.

To prove the first part, notice that, since $ij$ and $\hat i \hat j$ are
restriction idempotents, 
\begin{align*}
  ijf &=  f \overline {ijf}\\
  &= f \hat i \hat j\\
  &= i \hat f \hat j& \text{and so, precomposing with $j$,}\\
  jij f &= ji \hat f \hat j& \text{i.e.}\\
  jf &= \hat f \hat j
\end{align*}
which is the first part.
\end{proof}
\begin{lemma}\label{lemma:magicAdjoint}
Suppose that, in $\ccont$, we have $f:A \rightarrow B$ and $g: B
\rightarrow A$ with $f \dashv g$. Then, as poset morphisms between
$\econt_A$ and $\econt_B$, $f^* \dashv g^*$. 
\end{lemma}
\begin{proof}\label{lemma:adjBaseToFibres}
Because $f \dashv g$ in $\ccont$, we have the unit and counit 
\begin{align*}
\id &\extendedBy gf &\text{and} && fg &\extendedBy \id.
\intertext{Then, by Lemma~\ref{lemma:magicPullback}, we have}
(gf)^*P & \leq P&\text{and}&& Q &\leq (fg)^* Q,
\intertext{for any $P \in \ob(\econt_A)$ and 
  $Q \in \ob(\econt_B)$. By the contravariance of $(\cdot)^*$, we
have}
f^*g^* P &\leq P &\text{and} && Q &\leq f^*g^*Q,
\end{align*}
which are, respectively, the counit and unit of $f^* \dashv g^*$. 
The triangle equalities are, since we are working with posets, trivial. 
\end{proof}
\begin{lemma}
Let $\econt \rightarrow \ccont$ be a 2-fibration in posets
over a partial cartesian category $\ccont$,
suppose that monics have quasi-inverses in $\ccont$, 
and let $i:A \rightarrow B$ be
monic in $\ccont$. Then 
$i^*: \econt_B \rightarrow \econt_A$ has a right
adjoint $\rapb_i$ satisfying Beck-Chevalley.
\end{lemma}
\begin{proof}
By hypothesis, $i$ has a right adjoint $j$, and so, 
by Lemma~\ref{lemma:magicAdjoint},
$i^* \dashv j^*$: so we can identify $\rapb_i$ with $j^*$. Furthermore,
these right adjoints satisfy Beck-Chevalley by 
Lemma~\ref{lemma:pullbackStuff}.
\end{proof}

\begin{theorem}
Let $\ccontT$ be the category of total morphisms of $\ccont$,
a split restriction category,
and let $\domains(\ccont)$ be the class of monos
in $\ccontT$ which split restriction idempotents in $\ccont$. 
Then there is an equivalence of bicategories between, on
the one hand, fibrations in
posets over $\ccontT$ such that, for every  $i \in \domains(\ccont)$,
$i^*$ has a right adjoint
satisfing \ref{eq:bcra}, and, on the other hand, 2-fibrations in
posets over $\ccont$.
\end{theorem}
\begin{proof}
By Lemma~\ref{lemma:GrothConst} we can establish the equivalence on
the level of functors to $\textsc{Poset}$. 
By \cite{Herm02:categoutlorelatmodalsimul},
the functor $\MCattoBPM$ is universal
among functors to bicategories $\mathfrak{K}$ 
which send monos in $\mathcal{M}(\someCat)$
to 1-cells with right adjoints satisfying \ref{eq:bcra}. 
We apply this with $\mathfrak{K} = \textsc{Poset}$.

The 2-category of fibrations in posets over $\ccontT$ such that, for every 
$i \in \domains(\ccont)$, $i^*$ has a right adjoint satisfying
\ref{eq:bcra} is equivalent to the 
2-category $\homSet {}{\ccontT^{\textsf{op}}} {\textsc{Poset}}$,
which, by \cite{Herm02:categoutlorelatmodalsimul}, is equivalent to the 2-category 
$\homSet{}{\ccont{\textsf{op}}}{\textsc{Poset}}$, which is
in turn equivalent to 
the 2-category of 2-fibrations in in posets over $\ccont$. 
\end{proof}
\begin{definition}
	Under such circumstances, if fibrations $\pi: \econt \rightarrow \ccont$
	and $\pi_{\textsf{tot}}:\econtT \rightarrow \ccontT$ correspond, we say that
	$\pi_{\textsf{tot}}$ is a \emph{restriction} of $\pi$.
\end{definition}
\begin{proposition}\label{prop:ladjPb}
	Suppose that $\pi_{\textsf{tot}}$ is a restriction
of $\pi$. Then $\pi$ has
left adjoints to the pullbacks iff $\pi_{\textsf{tot}}$ does. 
\end{proposition}
\begin{proof}
	The direction from $\pi$ to $\pi_{\textsf{tot}}$ is clear, since
	every pullback in $\pi_{\textsf{tot}}$ is a pullback in $\pi$. Suppose,
on the other hand, that $f$ is a 1-cell in $\ccontT$: then, because
$\ccont$ is equivalent to a category of partial morphisms of
$\ccontT$, we may suppose that $f = f_0 j$, where $j$ is the 
right adjoint of a mono $i$ and $f_0$ is total. Because $f_0$ is 
total, it has, by hypothesis, a left adjoint $\lapb_{f_0}$, and $\lapb_{f_0}i$
is then the required left adjoint for $f$. 
\end{proof}
\subsection{The Total Category of a Fibration}
We now consider the structure of the total category $\econt$.
\begin{lemma}
	Let $\pi:\econt \rightarrow \ccont$ be a 2-fibration, and let $\ccont$ be
	a restriction category. Then there is a unique restriction structure
	on $\econt$ which makes $\pi$ into a restriction homomorphism.
	\label{lemma:uniqueResStrucLift}
\end{lemma}
\begin{proof}
Note first that, if $\alpha: A \rightarrow A$ is a coreflection in the base, and
if $\pi (P) = A$, then there is a unique  $\vartheta \extendedBy \id_P$
with $\pi (\vartheta) = \alpha$: the fibrational conditions on $\extendedBy$ give
us existence, and the same conditions give us equality of any two candidates. 
$\vartheta$ and $\vartheta\vartheta$ are both lifts of $\alpha$, so that,
by uniqueness of lifting, they are equal, which establishes idempotence.
We can now define a restriction structure on $\econt$ by letting
$\overline f$ be the unique lift of $\overline{\pi(f)}$: similar
uniqueness arguments give us the restriction axioms.
\end{proof}
\begin{proposition}\label{prop:tCatPM}
Suppose that $\econt$ is 2-fibred in posets over $\ccont$,
that $\ccont$ is a bicategory of patial maps
and that the fibres of $\econt$ have finite joins. Then
$\econt$ is itself a bicategory of partial maps, and
the projection $\pi: \econt \rightarrow \ccont$
is compatible with the structure. 
\end{proposition}
\begin{proof}
	By Lemma~\ref{lemma:uniqueResStrucLift}, we already
	have a restriction category structure on $\econt$: we
	now only need to show that we have a restriction final 
	object and binary restriction products. The restriction final 
	object will be the top element of the fibre over $I$, the
	restriction final object of $\ccont$: to define the
	binary restriction product, 
let $P$ and $Q$ be objects of $\econt$, with $\pi(P) = A$
and $\pi(Q) = B$. Let $p_A$ and $p_B$ be the two
projections of $\pi(P) \otimes \pi(Q)$. Then let
\begin{displaymath}
P \otimes Q \quad = \quad p_A^*P
\land p_B^*Q
\end{displaymath}
The universal property is easily verified.
\end{proof}

\begin{theorem}
Let $\pi:  \econt \rightarrow \ccont$ be a 2-fibration in posets, and suppose that
$\ccont$ has comma objects: suppose also that the fibres of $\pi$ 
have finite joins.
Then $\econt$ has comma objects.
\label{theorem:totalCommaObjects}
\end{theorem}

\begin{proof}
	As above, we can show that $\econt$ has a restriction structure and 
	a restriction final object: we only have to show that it has weak comma 
	objects. So, consider morphisms $f: P \rightarrow Q$, $g: R \rightarrow Q$
	(that is, we have $P \leq \pi(f)^* Q$ and $R \leq \pi(g)^* Q$. 
	Construct the comma object of $\pi(f)$ and $\pi(g)$:
\begin{displaymath}
	\begin{xy}
		\square[\pi(f) \comma \pi(g) ` \pi(R) ` \pi(P) ` \pi(Q) ;\widehat{\pi(f)}` \widehat{\pi(g)} ` \pi(g) ` \pi(f)]
	\end{xy}
\end{displaymath}
We now define
\begin{displaymath}
	f \comma g \: = \: \widehat{\pi(g)}^* P \land \widehat{\pi(f)}^* Q.
\end{displaymath}
and the universal property is easy to verify.
\end{proof}

\subsection{Frobenius and its Consequences}\label{section:Frobenius}
We now discuss Frobenius laws for these fibrations: they
are important in themselves, but they also have useful 
consequences. In particular, they will give us correspondence 
theorems between 2-fibrations (in Boolean
algebras or in coHeyting semilattices) over bicategories of
partial maps and 1-fibrations over their categories of total maps.

We first define Heyting and coHeyting semilattices: the latter
are important because we are concerned, in this article,
with equational reasoning. We will mainly investigate a classical
system, but the constructive variant will be based on apartness, 
and the appropriate structure in the fibres will be coHeyting.
We need 
strong and weak morphisms for both objects: this is because, over $\ccont$, pullbacks
along partial morphisms of
coHeyting semilattices will preserve $\lor$, 
but will
not, in general, preserve $\lfalse$ and will only preserve the coHeyting operation
in a rather weak sense. 
For the general correspondence theory, we will need 
both sorts of morphism, because
only pullbacks along total morphisms 
will, in general, preserve $\lfalse$ and the coHeyting operation. 

\begin{definition}
A \emph{Heyting semilattice} is a poset
with all finitary meets, the binary meet being written  $\land$ and the
nullary meeting being written $\ltrue$, together with 
a binary operation $\heyting {P}{Q}$
such that 
\begin{prooftree}
	\AxiomC{$P \land Q\:\, \leq\: R$}
	\doubleLine
	\UnaryInfC{$P\: \leq\:\, \heyting Q R$}
\end{prooftree}
\end{definition}
\begin{definition}
	A  \emph{morphism of Heyting semilattices} is
	a poset morphism which preserves $\land$, $\ltrue$, and the Heyting operation.
\end{definition}
\begin{definition}
	A \emph{weak morphism of Heyting semilattices} is a poset
	morphism $\phi$ which preserves $\land$, and for which
	$\phi(\heyting P Q) \: = \: \heyting{\phi (P)} {\phi(Q)} \land \phi(\ltrue)$
\end{definition}
\begin{definition}
A \emph{coHeyting semilattice} is a poset
with all finitary joins, the binary join being written  $\lor$ and the
nullary meet being written $\lfalse$, together with 
a binary operation $\coheyting {P}{Q}$
such that 
\begin{prooftree}
	\AxiomC{$P\: \leq Q\lor R$}
	\doubleLine
	\UnaryInfC{$\coheyting P Q \: \leq\:R$}
\end{prooftree}
\end{definition}
\begin{definition}
	A  \emph{morphism of coHeyting semilattices} is
	a poset morphism which preserves $\lor$, $\lfalse$, and the coHeyting operation.
\end{definition}
\begin{definition}
	A \emph{weak morphism of coHeyting semilattices} is a poset morphism $\phi$ which preserves
	$\lor$, and for which 
	\begin{equation}
	\phi(\coheyting P Q) \: = \: \left(\coheyting{\phi (P)} {\phi(Q)}\right) \lor \phi(\lfalse)
		\label{eq:weakCoHeyting}
	\end{equation}
\end{definition}

\begin{remark}\label{remark:motivateWeakMorphisms}
	The definitions of weak morphisms can be motivated as follows. 
	For  a given $a$, the downward closure of $a$
	can be given a Heyting structure in a natural way:
	the new $\ltrue$ is $a$, $\land$ is as before, and the new Heyting operation is
	$(\heyting{x}{y})\land a$. Then a weak Heyting morphism is just a Heyting morphism 
	with codomain  the downward closure of $a$. The situation for coHeyting morphisms is
	dual.
\end{remark}
Now we can start on Frobenius laws. 
The following is standard:
\begin{proposition}
	Let $\pi: \econt \rightarrow \ccont$ be a (1- or 2-)fibration in Heyting semilattices, 
	and suppose that, for  a 1-cell $f$ in the base, its
	pullback $f^*$ has a left adjoint and commutes with the Heyting
	operation. Then the following Frobenius
	property holds: 
	\begin{displaymath}
		(\lapb_f P) \land Q \: = \: \lapb_f(P \land f^*Q).
	\end{displaymath}
\end{proposition}
\begin{proof}
	We have to establish lattice inequalities in both directions.
	The direction
	$	(\lapb_f P) \land Q \: \geq \: \lapb_f(P \land f^*Q)$
	is easy, and only requires the adjunction $\lapb_f \dashv f^*$;
	for the other direction, we use the Heyting operation
	\citep[see][Lemma~1.9.12, p.~102]{JacobsCLTT}, \citep{white:_david_reiter},
	noting that, since $f^*$ has a left adjoint, it preserves $\ltrue$, and 
	so the notions of weak Heyting and Heyting coincide.
\end{proof}

Dually, we have the following (this will also be useful to us, since the
domain fibration, which we will study in Section~\ref{section:domFib},
is a fibration in coHeyting semilattices. 
\begin{proposition}\label{prop:frobCoHey}
	Let $f$ be a fibration in coHeyting semilattices, and suppose that,
	for a 1-cell $f$ in the base, its 
	pullback $f^*$  has a right adjoint and commutes with the coHeyting
	operation. Then the following Frobenius
	property holds: 
	\begin{equation}
		(\rapb_f P) \lor Q \: = \: \rapb_f(P \lor f^*Q).\label{eq:frobCoHey}
	\end{equation}
\end{proposition}
\begin{corollary}
	Suppose that we have a 2-fibration in coHeyting semilattices 
	$\pi: \econt \rightarrow \ccont$,
	where $\ccont$ is a bicategory of partial maps. Let $i \dashv$ in 
	$\ccont$: then 
	\begin{equation}
		j^* P \lor Q \:=\:j^* (P \lor i^*Q).\label{eq:intFrob}
	\end{equation}
\end{corollary}
\begin{proof}
	$i \dashv j$, so $i^* \dashv j^*$: $i^*$ satisfies Frobenius 
	by Proposition~\ref{prop:frobCoHey}. \eqref{eq:intFrob} follows
	from this, writing $j^*$ instead of $\forall_i$.
\end{proof}
\begin{corollary}
	If we have a fibration in coHeyting semilattices, and
	if $i \dashv j$, with 
	$P$ and $Q$ in
	the fibre over the codomain of $i$,
\begin{displaymath}
i^* P \leq i^* Q \quad \text{iff} \quad
P \leq Q \lor j^* \lfalse
\end{displaymath}
\end{corollary}
\begin{proof}
Right to left is a straightforward calculation, since
$i^*$ (having a right adjoint) preserves $\lor$ and since
$i^*j^* = \id$. For left to right, we argue as follows:
\begin{center}
\AxiomC{$i^*P \leq i^*Q$}
\RightLabel{$i \dashv j$}
\UnaryInfC{$P \leq j^*i^* Q$}
\UnaryInfC{$P \leq j^*(i^*Q\, \lor \lfalse)$}
\RightLabel{Frobenius}
\UnaryInfC{$P \leq Q \lor j^* \lfalse$}
\DisplayProof
\end{center}

\end{proof}
The following result will be important for our sequent calculus:
\begin{corollary}\label{corollary:coHeyting}
Suppose that we have a fibration in
coHeyting semilattices, and that we have, in
the base, $f \extendedBy g: A \rightarrow B$. Then, for 
any $P$ over $B$, we have 
\begin{displaymath}
f^* P \: = \: g^* P \lor f^* \lfalse.
\end{displaymath}
\end{corollary}
\begin{proof}
	We can assume,
	wlog,  that $f = g ij$, with $i \dashv j$. 
	But now 
	\begin{align*}
	g^*P \lor f^* \lfalse &= g^* P \lor j^* i^* g^* \lfalse \\
	&= j^* (i^* g^* P \lor i^* g^* \lfalse) & \text{by Frobenius}\\
	&= j^* i^* g^* P &\text{by monotonicity of $i^* g^*$}\\
	&= f^* P& \text{QED}
	\end{align*}
\end{proof}
We can now apply these results to correspondence results between fibrations
in coHeyting semilattices over
$\ccont$ and those over $\ccontT$: the following example shows that, as we claimed,
we cannot have a fibration in coHeyting semilattices and \emph{strict} morphism
over a bicategory of partial maps.
\begin{example}\label{example:coHeytingWeakness}
	Suppose that we have a fibration in coHeyting algebras over a bicategory of
	partial maps,  that we have $i \dashv j $ in the base, and that $i^*$ is a
	strict Heyting algebra morphism. $j^*$ is a right adjoint, so it preserves
	$\ltrue$: furthermore, for any $P$, $\coheyting{P} {\ltrue}  = \lfalse$. 
	So we have $j^* (\coheyting{P} {\ltrue}) = j^* \lfalse$, but 
	$\coheyting{j^* P}{j^* \ltrue} = \coheyting {j^* P} {\ltrue} = \lfalse$. 
	However, $j^* \lfalse $ will not be equal to $\lfalse$ in general: the subobject
	fibration of a category of sets and partial maps shows that.
\end{example}
\begin{proposition}
Suppose that $\ccont$ is a bicategory of partial maps.
Then the following are equivalent:
\begin{enumerate}
\item 2-fibrations in coHeyting semilattices and
	weak coHeyting semilattice
	morphisms $\econt \rightarrow \ccont$
such that
\begin{enumerate}
  \item for all objects $A$, pullbacks along $!_A$ preserve $\lfalse$
  \item pullbacks have left adjoints
\end{enumerate}
\item fibrations in coHeyting semilattices  and coHeyting
	semilattice morphisms $\econtT\rightarrow \ccontT$
such that
\begin{enumerate}
  \item pullbacks have left adjoints
  \item pullbacks along monos have right adjoints 
	which satisfy Frobenius 
\end{enumerate}
\end{enumerate}
\end{proposition}

\begin{proof}
We first show that $1 \Rightarrow 2$: the pullbacks preserve finite joins in
$\ccontT$
because they do so in $\ccont$. Similarly, the pullbacks in $\ccontT$ have
left adjoints.
The existence of right adjoints satisfying Frobenius for monos
follows from the fact that, if $i$ is a mono in $\ccontT$, it
has a right adjoint $j$ in $\ccont$: but then $j^*$ is the desired right
adjoint to
$i^*$, and the Frobenius properties correspond. Finally we need to show that
pullbacks 
along total morphisms are strict coHeyting morphisms. Firstly, they 
preserve
 $\lfalse$ because, if $f:A\rightarrow B$ is total, then, by
definition, $!_Bf = !_A$,
and so $\lfalse_A = !_A^* \lfalse_I = f^* !_B^* \lfalse_I = \lfalse_B$. 
If they preserve $\lfalse$, they, by the definition of weak 
coHeyting morphisms, they are strict coHeyting morphisms.

For the other direction we argue as follows. Proposition~\ref{prop:ladjPb}
gives us
an extension of the fibration and left adjoints to the pullbacks: the
Frobenius properties
then correspond. We now have to show that pullbacks along 1-cells preserve
binary joins: 
this is true by assumption for total 1-cells, and we have to show that it
holds
for 1-cells which are right adjoint to monos. So, let $i \dashv j$: $ji= \id$,
and so
\begin{align*}
j^*(P \lor Q) &= j^*(i^*j^* P \lor Q)\\
  & = j^* P \lor j^* Q & \text{by Frobenius.}
\end{align*}
We have now to show that the pullbacks preserve the coHeyting operation
in the required weak sense. 
Firstly, an easy calculation shows that \eqref{eq:weakCoHeyting} is
preserved under composition: so it suffices to show that it
holds for total morphism and for right adjoints to monos. 
It holds for total morphisms because they preserve $\lfalse$, and
so \eqref{eq:weakCoHeyting} requires, in this case, strict
preservation of the coHeyting operation, which we have by assumption.
So we have to show that \eqref{eq:weakCoHeyting} holds
for pullbacks along right adjoints to monos.
Note
first that, in any coHeyting semilattice, $\coheyting P Q$ is the
infimum of the $X$ such that $P \: \leq \: Q \lor X$. 
So we have, for any $S$, 
	\begin{prooftree}
		\def\fCenter{\:\leq\:}
		\alwaysDoubleLine
		\Axiom$S \fCenter j^*(\coheyting P Q)$
		\UnaryInf$i^* S \fCenter \coheyting P Q$
		\UnaryInf$i^* S \fCenter T \quad\text{for any $T$ such that $P \leq Q \lor T$}$
		\UnaryInf$S \fCenter j^* T$
		\UnaryInf$S \fCenter j^* i^* T'\quad \text{for any $T'$ such that $P \leq Q \lor i^* T'$ ($i^*$ is surjective)}$
		\UnaryInf$S \fCenter j^* i^* T' \quad \text{for any $T'$ such that $j^* P \leq j^*Q \lor j^* i^* T'$ ($j^*$ is injective)}$
		\UnaryInf$S \fCenter j^* i^* T' \quad \text{for any $T'$ such that $j^* P \leq j^*Q \lor  T' \lor j^* \lfalse$ (Frobenius)}$
		\UnaryInf$S \fCenter j^* i^* T' \quad \text{for any $T'$ such that $j^* P \leq j^*Q \lor  T'  $ ($j^* \lfalse \leq j^*Q$)}$
		\UnaryInf$S \fCenter j^* i^* T' \quad \text{for any $T'$ such that $\coheyting{(j^* P)} { (j^*Q)} \leq  T'$ }$
		\UnaryInf$S \fCenter T' \lor j^* \lfalse\quad \text{ for any $T'$ such that $\coheyting{(j^* P)}{(j^* Q)}\leq T' $,
			by Frobenius}$
			\UnaryInf$S \fCenter \coheyting{(j^* P)}{(j^* Q)} \lor j^* \lfalse$
	\end{prooftree}
	and so we have the result by Yoneda.
\end{proof}
\begin{corollary}
	Let $\ccontT$ be the subcategory of total maps of $\ccont$. 
	The following are equivalent:
	\begin{enumerate}
		\item 
			Fibrations in boolean algebras and $\land$, $\lor$ and $\ltrue$-preserving
			poset morphisms over $\ccont$  such that pullbacks along $!_A$, for 
			any $A$, preserve $\lfalse$ 
		\item 
			Fibrations in boolean algebras and boolean algebra morphisms 
			over $\ccontT$
	\end{enumerate}
\end{corollary}

\begin{proof}
	We use the obvious coHeyting structure on a boolean algebra, and 
	apply the previous proposition.
	For $1 \Rightarrow 2$, we use the fact that a lattice homomorphism of 
	a boolean algebra is a boolean algebra morphism. For $2 \Rightarrow 1$
	we extend the fibration from $\ccontT$ to $\ccont$ in the usual
	way: we factorise a given 1-cell $f$ of $\ccont$ as $f_0 j$,
	with $i \dashv j$,
	and express $f^*$ as $\rapb_i f_0^*$. Now $\rapb_i$ preserves $\land$, 
	because it is a right adjoint, and $f_0^*$ does by assumption.
	Preservation of $\lor$ 
	follows from the above proposition.
\end{proof}

\begin{remark}
	We can motivate the results of this section as follows. We have a general
	structure theory for bicategories of partial maps which describes morphisms 
	in these categories as total morphisms precomposed with partial morphisms 
	of a special form: so, every $f$ is of the form $f_0j$, with $f_0$ total
	and $j$ being the quasi-inverse of a monic. $j$ prevents $f_0j$ from being
	defined everywhere. Now if we look at a pullback along $f$, then, because 
	of contravariance, we find that $f^* = j^*f_0^* $. If we have a fibration in, let
	us say, coHeyting algebras, then, if the pullbacks were coHeyting algebra morphisms,
	then they would, by Example~\ref{example:coHeytingWeakness}, have to have
	$\lfalse$ as a value, at least in some plausible cases. This may well be
	possible for pullbacks along total morphisms, but in general we will
	have pullbacks of the form $j^*f_0^*$, and here $j^*$ is an obstruction
	to the pullback having the required values: if $j^*$ cannot hit $\lfalse$,
	then neither can $f^*$. So the best we can do is to have a morphism
	whose codomain is the segment $[f^*\lfalse,\ltrue]$, 
	i.e.\ (by Remark~\ref{remark:motivateWeakMorphisms}])
	a weak coHeyting
	morphism.
\end{remark}
\subsection{Beck-Chevalley}\label{section:BeckChevalley}
\begin{theorem}
Suppose that we have a fibration over $\ccont $ 
whose pullbacks have left adjoints, and suppose
also that pullbacks along monos have right adjoints
which satisfy Beck-Chevalley. 
Then the left adjoints of 
the fibration over $\ccontT $ satisfy Beck-Chevalley
with respect to pullback squares iff the
left adjoints of the fibration over $\ccont$
satisfy Beck-Chevalley with respect to comma squares: that is,
if we have 
\begin{displaymath}
	\bfig
	\square [f \comma g ` C ` A ` B; \hat f ` \hat g ` g ` f]
	\efig
\end{displaymath}
and if we have $P$ over $A$, then we have 
\begin{displaymath}
	g^* \lapb_f P \quad = \lapb_{\hat f} \hat g^* P.  
\end{displaymath}
\label{thm:grandBeckChevalley}
\end{theorem}
\begin{proof}
The if direction is trivial: we have to prove that, if the fibration 
over $\ccont$ satisfies Beck-Chevalley,
then the fibration over $\ccontT$ does. 

We prove the only if direction  by the usual pasting argument
as follows. 
In the diagram of Proposition~\ref{prop:decompCommaSquare}
\begin{displaymath}
	\bfig
	\square(750,0)[{f \comma g}`C`A` B ; \hat f ` \hat g ` g ` f]
	\square<750,500>[A_0 \times_A (f \comma g) ` f \comma g ` A_0 ` A; \hat i ` \hat {\hat g} `  ` i]
	\efig
\end{displaymath}
note that $\hat i$ is a mono (because it is a pullback of a mono): thus, both
$i$ and $\hat i$ have right adjoints $j$ and $\hat j$. Furthermore, we can
factorise $f$ as $f_0 j$ and $\hat f $ as $\hat f_0 \hat j$, where 
$f_0 = fj$ and $\hat f_0= \hat f \hat j$ are total. So we have a diagram
\begin{displaymath}
	\begin{xy}
		\square <750,500>[f\comma g ` A_0 \times_A (f \comma g) ` A ` A_0; \hat j ` \hat {\hat g} ` ` j]
		\square(750,0) [A_0 \times_A (f \comma g) ` C ` A_0 ` A; \hat f ` \hat {\hat g} ` g ` f_0 ]  
	\end{xy}
\end{displaymath}
The right hand square is a diagram in $\ccontT$, and we have
Beck-Chevalley for that by assumption: because $i \dashv j$, Beck-Chevalley
for the left hand square follows from the $\rapb$ Beck-Chevalley
condition of the left hand square of the previous diagram. So we have the
result by pasting. 
\end{proof}

\subsection{The Domain Fibration}\label{section:domFib}
The domain  subobject fibration is defined for a broad range
of 1-categories: the fibre over an object $A$ is the set of
subobjects of $A$, with substitution defined by pullback. 

In the case of partial cartesian categories, we have a particular class
of subobjects of $A$, namely those given by the domains of
definition of 1-cells from $A$. The corresponding
fibration is called the \emph{domain fibration}; it is a fibration
in $\land$-semilattices.
However, there are subtleties to do with the variance of
the fibration thus defined. The fibrations that we have so far studied arise
from 2-functors $\ccont^{\textsf{coop}} \rightarrow \poset$, where
$\poset$ is the two-category of posets, the morphisms ordered pointwise.
Thus, 
the substitution morphisms are contravariant on 1-cells and 2-cells. It is 
also possible to define fibrations with substitutions contravariant
on 1-cells but covariant on 2-cells, that is, fibrations corresponding to
2-functors $\ccont^{\textsf{op}} \rightarrow \textsc{Poset}$. 

This comes about as follows. Hermida's correspondence, 
described in Section~\ref{sec:hermidaCorrespondence},
shows how 1-fibrations over the total 1-cells of a split partial cartesian
category $\ccont$ can be extended to 2-fibrations over $\ccont$.
Let $\domains$ be the class of monos in $\ccont$ which split 
restriction idempotents: then a fibration $\pi: \econtT \rightarrow \ccontT$,
contravariant on 1-cells, 
extends to a fibration over $\ccont$ contravariant on 1- and 2-cells iff
for each $i$ in $\domains$, $i^*$ has a right adjoint satisfying Beck-Chevalley
with respect to pullbacks along any 1-cell in $\ccontT$.

But the same construction also yields, in exactly the
same way, a result with different variance:
\begin{proposition}
	A fibration over $\ccontT$, contravariant on 1-cells, can be extended to a fibration 
	over $\ccont$, contravariant on 1-cells and covariant on 2-cells, iff, for
	each $i$ in $\domains$, $i^*$ has a \emph{left} adjoint satisfying Beck-Chevalley.
\end{proposition}

	So we can, depending on the existence of appropriate adjoints, have
	either sort of domain fibration. We will consider each case separately: first, though,
	we show what the existence of either sort of adjoint amounts to.
	\begin{lemma}
	Let $\ccont$ be a split restriction category, and let $\domains$ be the class of monos 
	which split restriction idempotents in $\ccont$.
	\begin{enumerate}
		\item The domain fibration of $\ccont$ has left adjoints, satisfying Beck-Chevalley 
			with respect to fibred products  with total morphisms, to
			pullbacks along monos in $\domains$
		\item The domain fibration of $\ccont$ has right adjoints
			to pullbacks along
			monos in $\domains$, satisfying Beck-Chevalley with respect to
			fibred products with total morphisms,
			iff $\ccontT$ has stable Heyting operations
	\end{enumerate}
	\label{lemma:domPullMon}
\end{lemma}
\begin{proof}
	This is mostly a reformulation of standard results.
	Observe that monos are in $\ccontT$, and so left or right adjoints in 
	$\ccont$ are left or right adjoints, respectively, in $\ccontT$. We can then
	apply \citet[pp.~256ff.]{JacobsCLTT}, which shows that the domain fibration
	has left adjoints to pullbacks along monos iff it has meets: but it does 
	have meets. Beck-Chevalley corresponds to the fact that meets are stable (i.e.
	that $f^* (\alpha \land \beta) = f^* \alpha \land f^* \beta$), which follows from
	\citet[p.~254]{cockett02:_restr_i}. This establishes the first part.
	
 	For the second part, \citet[pp.~256ff.]{JacobsCLTT} shows that
	the domain fibration has right adjoints to pullbacks along monos iff
	it has Heyting operations on  the posets of domains: Beck-Chevalley 
	then corresponds to stability of the Heyting operations (we can show, 
	by suitably factoring morphisms, that it suffices to verify the
	stable Heyting condition with total morphisms).
\end{proof}

	\begin{remark}
		Terminology for this sort of thing is a disaster. ``co'' can either mean ``in the 
		same direction'' (as in covariant), or ``in the opposite direction'' (as in counit). 
		I shall abbreviate the names of these fibrations to ``2-covariant'' and ``2-contravariant'',
		which is clumsy, but I can't see any better solution.
	\end{remark}
\subsubsection{The 2-Covariant Domain Fibration}
\begin{definition}

Let $\ccont $ be a restriction category. Define the \emph{2-covariant domain fibration} of $\ccont $,
$\domf^{\textsf{co}}(\ccont)$, as follows:
\begin{description}
	\item[The fibre] over an object $A$ of $\ccont$, $\domf(\ccont)_A$, is the $\land$-semilattice 
		$\{\overline{f} | f:A \rightarrow B\}$, where 
		\begin{description}
			\item[$\alpha \extendedBy \beta$] iff $\alpha \beta = \alpha$,
			\item[$\alpha \land \beta$] is $\alpha \beta$, and
			\item[$\ltrue$] is $\overline{\id_A}$.
		\end{description}
	\item[The pullback] of $\alpha$ along $f$ is $\overline{\alpha f}$
\end{description}
\end{definition}
\begin{proposition}
	The domain fibration is a two-fibration which is contravariant on 1-cells and covariant on
	2-cells (that is, it corresponds to a 2-functor $\ccont^{\textsf{op}} \rightarrow \asl$,
	where $\asl$ is the 2-category of $\land$-semilattices).
	\label{proposition:domfibop}
\end{proposition}
\begin{proof}
	Routine calculation: we use the fact that the above definition
	when restricted to $\ccontT$ gives a 1-fibration, and then, to 
	show that it corresponds to a 2-fibration, we use the Hermida correspondence
	with the appropriate variance. For this we need to show that monos $i$ in $\domains$
	have left adjoints $\exists_i$ satisfying Beck-Chevalley with respect to
	fibred products with total morphisms. 
	This adjoint is given by the first part of Lemma~\ref{lemma:domPullMon}.
	Finally we show that the pullbacks constructed by the Hermida construction
	coincide with those given by the definition above, which is a routine calculation
	using Lemma~\ref{lemma:pullbackStuff}.
\end{proof}

\subsubsection{The 2-Contravariant Domain Fibration}
We can, as remarked above, define this when the domain posets have stable Heyting operations:
as well as the general argument given there, we can define the fibration explicitly as follows:

\begin{definition}
	Let $\ccont$ be a split bicategory of partial maps,
	and let $\domains$ be the class of monos which split 
	restriction idempotents. Suppose that $\ccontT$
	is a fibration in Heyting semilattices and weak Heyting
	semilattice morphisms: then
	define the \emph{2-contravariant category of domains},
\begin{description}
	\item[The fibre] over an object $A$ of $\ccont$, $\domf(\ccont)_A$, is the Heyting-semilattice 
		$\{\overline{f} | f:A \rightarrow B\}$, where 
		\begin{description}
			\item[$\alpha \extendedBy \beta$] iff $\alpha \beta = \alpha$,
			\item[$\alpha \land \beta$] is $\alpha \beta$, and
			\item[$\ltrue$] is $\overline{\id_A}$.
		\end{description}
	\item[The pullback] of $\alpha$ along $f$ is $\heyting{\converges{f}}{\overline{\alpha f}}$
\end{description}
\end{definition}

\begin{theorem}
	Let $\ccont$ be a restriction category with a restriction final object
where the subobject fibration
has left adjoints to the pullbacks. 
	 TFAE:
	\begin{enumerate}
		\item $\ccont$ has weak comma objects
		\item the subobject fibration of $\ccontT$ has
			stable Heyting operations
	\item the subobject fibration of $\ccontT$ has\label{cond3} 
		right adjoints to pullbacks along monics
	\item the 2-contravariant domain fibration is defined
	\end{enumerate}
	\label{prop:commaObjectSubobFib}
\end{theorem}
\begin{proof}
	This is a combination of the results of Section~\ref{sec:weakComma} 
	together with the previous proposition.
\end{proof}
\begin{example}
  Consider the category of sets and partial maps. The two domain fibrations are 
	defined as follows: domains are, in both cases, simply subsets, but the pullbacks
	are as follows.
	\begin{description}
		\item[$\domf^{\textsf{co}}$]Let $f:A \rightarrow B$ be a 1-cell, and let $V \subseteq B$. 
			Then, for $x \in A$,
			\begin{displaymath}
				x \in f^*V \quad \text{iff}\quad (f(x)\!\!\downarrow)\, \land\, f(x) \!\in V
			\end{displaymath}
		\item[$\domf^{\textsf{con}}$]
			Let $f:A \rightarrow B$ be a 1-cell, and let $V \subseteq B$. 
			Then, for $x \in A$, 
			\begin{displaymath}
        x \in f^*V \quad \text{iff} \quad (f(x)\!\!\downarrow)\, \rightarrow \,f(x) \!\in V        
			\end{displaymath}
	\end{description}
\end{example}
\begin{remark}
If our category of total morphisms had stable sums and epi-mono factorisations in addition
to the above conditions, then it would be a \emph{logos}. 
\end{remark}
\section{The Sequent Calculus}

We can now define a sequent calculus. We fix a functionally complete category
of partial maps for the base.  Our calculus will be typed: propositions are
typed by objects of the base category, and, for each one-cell of the base, we
have substitution operators on propositions of the appropriate types. Formation
rules for propositions and sequents are given in Table~\ref{tab:forSeq}: note
the substitution rules for sets of propositions $\Gamma$ on the left, and
$\Delta$ on the right.

\subsubsection{Notational Conventions}
As we have seen, the notation for tensors of objects and arrows
tends to become rather cumbersome. We will frequently abbreviate it by leaving out the names
of objects, and writing a diagram of the form
\begin{displaymath}
 \bfig
\Square[A\comma C`C`A`B;f\comma C`A\comma g`g`f]
\efig
\quad
\text{as}
\quad
   \bfig
\Square[\bullet`\bullet`\bullet`\bullet;\hat f`\hat g`g`f]
\efig
\end{displaymath}
We will hardly ever need the names of objects in our sequent calculus, and 
we will only use the superscript $\hat \cdot$  in diagrams of the above
form (or those constructed from them): with these conventions, the diagrams 
should be unambiguous.
\subsection{The Rules}
As we have said, the logic in
the fibres will be classical, and we use the standard sequent calculus 
rules for the sentential connectives and for cut: our primitives are $\lor$ and $\land$,
and $\rightarrow$ will be a defined connective. These rules are
given in Table~\ref{tab:davidsonCalcRules}. 
\begin{table}
\begin{displaymath}
\boxed{
  \begin{array}{r@{\extracolsep{25pt}}cc}
	\textbf{formulae} 
	&
	\multicolumn{2}{c}{
	  \AxiomC{$\vphantom{\ltrue}$}
	  \UnaryInfC{$\ltrue:A$}
	  \DisplayProof
	   \qquad\qquad
	   \AxiomC{$\vphantom{P}$}
		\UnaryInfC{$\lfalse:A$}
		\DisplayProof
	   \qquad\qquad
		 \AxiomC{$\phantom{\lnot}P:A$} 
		 \UnaryInfC{$\lnot P:A$}
		 \DisplayProof
}
	\\[\erh]
	&
	\AxiomC{$P:A$}
	\AxiomC{$Q:A$}
	\BinaryInfC{$P\land Q:A$}
	\DisplayProof
	&
	\AxiomC{$P:A$}
	\AxiomC{$Q:A$}
	\BinaryInfC{$P\lor Q:A$}
	\DisplayProof
	\\[\erh]
	&
	\AxiomC{$P:A$}
	\AxiomC{$f: A \rightarrow B$}
	\BinaryInfC{$\lapb_f P:B$}
	\DisplayProof
	&
	\AxiomC{$P:A$}
	\AxiomC{$f:A \rightarrow B$}
	\BinaryInfC{$\rapb_f P:B$}
	\DisplayProof
	 \\[\erh]
	 \textbf{substitution}
	 &
	 \multicolumn{1}{c}{
		\AxiomC{$f: A \rightarrow B$}
	   \AxiomC{$
		 Q:B
		 $
}
	   \BinaryInfC{$
		 f^*(Q):A
		 $
}
	   \DisplayProof
}
\\[\erh]
	   &
	   \AxiomC{$\Gamma = \{P_1, \ldots, P_n\}$}
	   \noLine
	   \UnaryInfC{$f^* \Gamma = \{f^*P_1, \ldots, f^*P_n\}$}
	   \DisplayProof
	   &
	   \AxiomC{$\Delta = \{Q_1, \ldots, Q_n\}$}
	   \noLine
	   \UnaryInfC{$f^* \Delta = \{f^*Q_1, \ldots, f^*Q_n, f^*\lfalse\}$}
	   \DisplayProof
	 \\[\erh]
	 \textbf{sequents}
	 &
	 \multicolumn{2}{c}{
	   \AxiomC{$
		 \Gamma:A
		 $
}
	  \AxiomC{$
		 \Delta:A
		 $
}
	   \BinaryInfC{$
		  A| \Gamma \entails \Delta
		 $
}
	   \DisplayProof
}
  \end{array}
}
\end{displaymath}
\caption{Formation Rules for Formulae and Sequents}\label{tab:forSeq}
\end{table}
\begin{table}
\begin{center}
  \fbox{
	\begin{tabular}{c@{\extracolsep{25pt}}c}
	  \multicolumn{2}{c}{
		\AxiomC{$\vphantom{|}$}
		\RightLabel{Ax}
		\UnaryInfC{
		  $A| \Gamma, P  
		\entails  
		\Delta, P$
		}
	  \DisplayProof 
	}
   \\[\erh] 
   \AxiomC{
	 $A| \Gamma \entails \Delta$
   }
   \AxiomC{$P:A$}
   \RightLabel{LW}
   \BinaryInfC{
	 $A |
	 \Gamma, 
	 P\entails 
	 \Delta$
   }
   \DisplayProof
   &
   \AxiomC{
	 $A|  \Gamma \entails  \Delta$
   }
   \AxiomC{$Q:A$}
   \RightLabel{RW}
   \BinaryInfC{
	 $A|
	 \Gamma \entails Q, \Delta$
   }
   \DisplayProof
   \\[\erh]
   \AxiomC{
	 $A| \Gamma, P,P \entails \Delta$
   }
   \RightLabel{LC}
   \UnaryInfC{
	 $A|\Gamma, P \entails \Delta$
   }
   \DisplayProof
   &
   \AxiomC{
	 $A|\Gamma \entails Q,Q,\Delta$
   }
   \RightLabel{RC}
   \UnaryInfC{
	 $A| \Gamma \entails Q,\Delta$
   }
   \DisplayProof
   \\[\erh]
   \AxiomC{$\vphantom{|}$}
   \RightLabel{$\lfalse$L}
   \UnaryInfC{
	 $A|\Gamma, \lfalse
	 \entails \Delta$
   }
   \DisplayProof
   &
   \AxiomC{$\vphantom{|}$}
   \RightLabel{$\ltrue$R}
   \UnaryInfC{
	 $A|\Gamma \entails \ltrue, \Delta$
   }
   \DisplayProof
   \\[\erh]
   \AxiomC{$\vphantom{|}$}
   \RightLabel{$f^*\lfalse$L}
   \UnaryInfC{
	 $A|f^*\Gamma, \lfalse
	 \entails f^*\Delta$
   }
   \DisplayProof
   &
   \AxiomC{$\vphantom{|}$}
   \RightLabel{$f^*\ltrue$R}
   \UnaryInfC{
	 $A|\Gamma \entails f^*\ltrue, \Delta$
   }
   \DisplayProof
   \\[\erh]
   \AxiomC{
	 $A|\Gamma, P_1 \entails \Delta$
   }
   \AxiomC{
	 $A|\Gamma, P_2 \entails \Delta$
   }
   \RightLabel{$\lor$L}
   \BinaryInfC{
	 $A|\Gamma, P_1 \lor P_2 \entails \Delta$
   }
   \DisplayProof
   &
   \AxiomC{
	 $A|\Gamma \entails \Delta, Q1,Q2$
   }
   \RightLabel{$\lor$R}
   \UnaryInfC{
	 $A|\Gamma \entails \Delta, Q_1 \lor Q_2$
   }
   \DisplayProof
   \\[\erh]
   \AxiomC{
	 $A|\Gamma, f^*P_1 \entails \Delta$
   }
   \AxiomC{
	 $A|\Gamma, f^*P_2 \entails \Delta$
   }
   \RightLabel{$f^*\lor$L}
   \BinaryInfC{
	 $A|\Gamma, f^*(P_1 \lor P_2) \entails \Delta$
   }
   \DisplayProof
   &
   \AxiomC{
	 $A|\Gamma \entails \Delta, f^*Q1,f^*Q2$
   }
   \RightLabel{$f^*\lor$R}
   \UnaryInfC{
	 $A|\Gamma \entails \Delta, f^*(Q_1 \lor Q_2)$
   }
   \DisplayProof
   \\[\erh]
   \AxiomC{
	 $A|\Gamma, P_1, P_2 \entails  \Delta$
   }
   \RightLabel{$\land$L}
   \UnaryInfC{
	 $A| \Gamma, P_1 \land P_2 \entails \Delta$
   }
   \DisplayProof 
   &
   \AxiomC{
	 $A | \Gamma \entails Q_1, \Delta$
   }
   \AxiomC{
	 $A | \Gamma \entails Q_2, \Delta$
   }
   \RightLabel{$\land$R}
   \BinaryInfC{
	 $A | \Gamma \entails Q_1 \land Q_2, \Delta$
   }
   \DisplayProof
   \\[\erh]
   \AxiomC{
	 $A|\Gamma, f^*P_1, f^*P_2 \entails  \Delta$
   }
   \RightLabel{$f^*\land$L}
   \UnaryInfC{
	 $A| \Gamma, f^*(P_1 \land P_2) \entails \Delta$
   }
   \DisplayProof 
   &
   \AxiomC{
	 $A | \Gamma \entails f^*Q_1, \Delta$
   }
   \AxiomC{
	 $A | \Gamma \entails f^*Q_2, \Delta$
   }
   \RightLabel{$f^*\land$R}
   \BinaryInfC{
	 $A | \Gamma \entails f^*(Q_1 \land Q_2), \Delta$
   }
   \DisplayProof
   \\[\erh]
   \AxiomC{
	 $A| \Gamma
	 \entails 
	 Q, \Delta$
   }
   \RightLabel{$\lnot$L}
   \UnaryInfC{
	 $A| \Gamma, \lnot Q
	 \entails \Delta$
   }
   \DisplayProof
   &
   \AxiomC{
	 $A|\Gamma, Q\entails \Delta$
   }
   \RightLabel{$\lnot$R}
   \UnaryInfC{
	 $A| \Gamma \entails 
	 \lnot Q,
	 \Delta$
   }
   \DisplayProof
   \\[\erh]
   \AxiomC{$A | \Gamma, f^* \lfalse \entails \Delta$}
   \AxiomC{
	 $A| \Gamma
	 \entails 
	 f^*Q, \Delta$
   }
   \RightLabel{$f^*\lnot$L}
   \BinaryInfC{
	 $A| \Gamma, f^*(\lnot Q)
	 \entails \Delta$
   }
   \DisplayProof
   &
   \AxiomC{
	 $A|\Gamma, f^*Q\entails f^*\lfalse, \Delta$
   }
   \RightLabel{$f^*\lnot$R}
   \UnaryInfC{
	 $A| \Gamma \entails 
	 f^*(\lnot Q),
	 \Delta$
   }
   \DisplayProof
   \\[\erh]
   \multicolumn{2}{c}{%
	 \AxiomC{
	   $A | 
	   \Gamma  \entails 
	   f_1^* P, 
	   \Delta$
	 }
	 \AxiomC{
	   $A |
	   \Gamma',
	   f_2^* P
	   \entails
	   \Delta'$
	 }
	 \AxiomC{$
	 f_2 \extendedBy f_1$}
	 \RightLabel{cut}
	 \TrinaryInfC{
	   $A
	   | \Gamma, \Gamma'
	   \entails
	   \Delta, \Delta'$
	 }
	 \DisplayProof
   }
 \end{tabular}
}
\end{center}
\caption{The Rules for the Sentential Connectives and Cut}%
\label{tab:davidsonCalcRules}
\end{table}
The rules specific to the bicategorical system are given 
in Table~\ref{tab:bicatRules}; note that we define
$f^*$ as follows.
\begin{definition}
Define $f^* \Gamma$, for a set of formulae $\Gamma$, as follows:
\begin{align*}
\intertext{If $\Gamma = \{\gamma_1, \ldots, \gamma_m\}$ 
is on the left, then}
f^* \Gamma \quad&= \quad\{f^* \gamma_1, \ldots, f^* \gamma_m\}
\intertext{If $\Delta = \{\delta_1, \ldots, \delta_n\}$ 
is on the right, then}
  f^* \Delta \quad&=\quad
\begin{cases}
	\{f^* \delta_1, \ldots, f^* \delta_m\} & \text{if $\Delta$ is nonempty}\\
	\{ f^* \lfalse\}&\text{otherwise.}
\end{cases}
\end{align*}
\end{definition}
The rules for $\lapb$ incorporate the Beck-Chevalley 
condition: this will make the proof of cut elimination much easier.

\begin{table}
	\begin{minipage}{\hsize}
\begin{center}
	\fbox{
  \begin{tabular}{c@{\extracolsep{25pt}}c}
	\multicolumn{2}{c}{
	  \AxiomC{$\Gamma \entails \Delta$}
	  \RightLabel{$f^*$}
	  \UnaryInfC{$ f^* \Gamma \entails f^* \Delta$}
	  \DisplayProof
				}
	\\[\erh]
	\AxiomC{$ \hat g^*\Gamma, \hat f^*P \entails \hat g^*\Delta$}
			\RightLabel{$\lapb$L\footnote{$f: A \rightarrow B$, $g: C \rightarrow B$, 
			$\hat g: f\comma g \rightarrow A$, $\hat f: f \comma g \rightarrow C$}}
	\UnaryInfC{$\Gamma, f^*(\lapb_g P) \entails \Delta$}
	\DisplayProof
			&
	 \AxiomC{$\Gamma \entails \tau ^* Q,\Delta$}
			 \AxiomC{$f \extendedBy g\tau $}
	\RightLabel{$\lapb$R}
	\BinaryInfC{$\Gamma \entails f^*\lapb_g Q,\Delta$}
	\DisplayProof          
			\\[\erh]
	\AxiomC{$f \extendedBy g$}
	\AxiomC{$ \Gamma,f^* P \entails
		  \Delta$}
					 \RightLabel{$\extendedBy\text{L}$}
	\BinaryInfC{$\Gamma, g^*P \entails 
		\Delta$}
	\DisplayProof
	&
	\AxiomC{$f \extendedBy g$}
	\AxiomC{$ \Gamma\entails
		 g^* Q, \Delta$}
					 \RightLabel{$\extendedBy\text{R}$}
	\BinaryInfC{$\Gamma \entails 
	 f^*Q,  \Delta$}
	\DisplayProof
   \\[\erh]
		 \multicolumn{2}{c}{
	\AxiomC{$f \extendedBy g$}
	\AxiomC{$ \Gamma,g^* \Gamma' \entails
		 g^* \Delta', \Delta$}
					 \RightLabel{$\extendedBy\text{LR}$}
	\BinaryInfC{$\Gamma, f^*\Gamma' \entails 
	   f^* \Delta', \Delta$}
	\DisplayProof
			}
   \\[\erh]
	\AxiomC{$\Gamma, f^*g^* P \entails \Delta$}
	\RightLabel{$\comp$L}
	\UnaryInfC{$\Gamma, (gf)^* P \entails \Delta$}
	\DisplayProof
	&
	\AxiomC{$\Gamma  \entails f^*g^* Q, \Delta$}
	\RightLabel{$\comp$R}
	\UnaryInfC{$\Gamma  \entails (gf)^* Q,\Delta$}
	\DisplayProof
	\\[\erh]
	\AxiomC{$\Gamma, (gf)^* P \entails \Delta$}
	\RightLabel{$\comp^{-1}$L}
	\UnaryInfC{$\Gamma, f^*g^* P \entails \Delta$}
	\DisplayProof
	&
	\AxiomC{$\Gamma  \entails (gf)^* Q, \Delta$}
	\RightLabel{$\comp^{-1}$R}
	\UnaryInfC{$\Gamma  \entails f^*g^* Q,\Delta$}
	\DisplayProof
	\\[\erh]
	\AxiomC{$\Gamma,  P \entails \Delta$}
	\RightLabel{$\id$L}
	\UnaryInfC{$\Gamma, \id^* P \entails \Delta$}
	\DisplayProof
	&
	\AxiomC{$\Gamma  \entails  Q, \Delta$}
	\RightLabel{$\id$R}
	\UnaryInfC{$\Gamma  \entails \id^* Q,\Delta$}
	\DisplayProof
	\\[\erh]
	\AxiomC{$\Gamma, \id^* P \entails \Delta$}
	\RightLabel{$\id^{-1}$L}
	\UnaryInfC{$\Gamma,  P \entails \Delta$}
	\DisplayProof
	&
	\AxiomC{$\Gamma  \entails \id^* Q, \Delta$}
	\RightLabel{$\id^{-1}$R}
	\UnaryInfC{$\Gamma  \entails  Q,\Delta$}
	\DisplayProof
\end{tabular}
	}
\end{center}
\end{minipage}
\caption{The Bicategorial Rules}\label{tab:bicatRules}
\end{table}
\subsection{Cut Elimination}

Proving cut elimination for systems like these faces the
following problem  \citep[see][]{GorePostTiu:09:Taming-D:rt}: if we have a cut such as 
\begin{center}
\AxiomC{$\Gamma \entails A, B, \Delta$}
\UnaryInfC{$\Gamma \entails A \lor B, \Delta $}
\AxiomC{$f^*\Gamma', f^*(A \lor B) \entails Q, f^*\Delta'$}
\UnaryInfC{$\Gamma', A \lor B \entails \lapb_f Q, \Delta'$}
\BinaryInfC{$\Gamma, \Gamma' \entails \lapb_f Q, \Delta, \Delta'$}
\DisplayProof
\end{center}
then it is not obvious how to move the cut upwards. We can deal with this
difficulty in two ways: we can either use a system with deep inference, as we
did in \citep{white:_david_reiter}, or we can, as we do here, 
use a more conventional syntax (with rules which make
the deep inference rules admissible) and prove an inversion lemma
together with an auxiliary result for the cases where the inversion lemma does not
work.
Both strategies cost about the same amount of work: the deep inference
strategy relies on unfamiliar syntax and is, as it were, more high level, whereas
the inversion lemma strategy relies on familiar syntax but is low level. 
But the inversion lemma also makes clear the role of Beck-Chevalley in 
the proof of cut elimination.
\citep{Herm04:Descefibrastronregulcateg}. 

First we prove some lemmas.
\begin{lemma}\label{lemma:admissibleRules}
The following rules are admissible in the cut-free system:
\begin{center}
\begin{tabular}{cc}
 \AxiomC{$\Gamma, P \entails \Delta$}
 \AxiomC{$f \extendedBy g$}
 \RightLabel{$f^*,g^*\text{L}$}
 \BinaryInfC{$f^*\Gamma, g^* P\entails g^* \Delta$}
 \DisplayProof
 &
 \AxiomC{$\Gamma \entails Q \Delta$}
 \AxiomC{$f \extendedBy g$}
 \RightLabel{$f^*,g^*\text{L}$}
 \BinaryInfC{$g^*\Gamma \entails f^* Q,g^* \Delta$}
 \DisplayProof
 \\[\erh]
  \AxiomC{$f^* \Gamma, P \entails f^* \Delta $}
  \RightLabel{$\lapb\text{L}'$}
  \UnaryInfC{$\Gamma, \lapb_f P \entails \Delta $}
  \DisplayProof
  &
  \AxiomC{$\Gamma \entails P, \Delta$}
  \RightLabel{$\lapb{}\text{R}'$}
  \UnaryInfC{$\Gamma \entails f^* \lapb{f}P, \Delta$}
  \DisplayProof
\end{tabular}
\end{center}
The diagram for 
$h^* \lapb \text{L}$ is as follows:
\begin{displaymath}
  \xy
  \square[h \comma \hat g ` f \comma g` \cdot ` \cdot ; 
  \hat h ` \hat {\hat g} ` \hat g ` h]
  \square(500,0)[f \comma g ` \cdot ` \cdot ` \cdot ; 
  \hat f ` ` g ` f]
  \endxy
\end{displaymath}
\end{lemma}
\begin{proof}
The $\lapb'$ rule is obtained by following the standard $\lapb\text{L}$ rule
with applications of $\id$ and $\id^{-1}$. 
The proof for $h^* \lapb \text{L}$ goes as follows:
\begin{displaymath}
	\AxiomC{$\hat h^* \hat g^* \Gamma, \hat h^* f^* \lapb g P \entails
	\hat h^* \hat g^* \Delta$}
	\RightLabel{$\extendedBy \text{R}$}
	\UnaryInfC{$\hat {\hat g}^* \hat h^* \Gamma, \hat h ^* \hat f^* \lapb_g P 
	\entails \hat {\hat g }^* \hat h^* \Delta$}
	\RightLabel{$\lapb \text{L}$}
	\UnaryInfC{$h^* \Gamma, h^* f^* \lapb_g P \entails h^* \Delta$}
	\DisplayProof
\end{displaymath}
\end{proof}

We also need to define a notion of the \emph{height} of a proof:
\begin{definition}
The \emph{height} of a proof is defined as follows:
\begin{enumerate}
	\item The height of a proof consisting of a single application of an
		axiom rule is zero
	\item	 If we have a proof of the form 
		\begin{center}
			\AxiomC{$\Pi$}
			\noLine
			\UnaryInfC{$\vdots$}
			\noLine
			\UnaryInfC{$\Gamma \entails \Delta$}
			\RightLabel{R}
			\UnaryInfC{$\Gamma' \entails \Delta'$}
			\DisplayProof
		\end{center}
		and if $R$ is one of the 
		$\extendedBy$,
		$\circ$, $\circ^{-1}$, $\id$,
		or $\id^{-1}$ rules, then the height of this proof is equal to 
		the height of $\Pi$
	\item Otherwise the height of a proof of this form is equal to
		$\mathop{\text{height}}(\Pi) + 1$. 
\end{enumerate}
\end{definition}
Similarly, the \emph{complexity} of a formula is defined as follows:
\begin{definition}
The compexity of a formula is defined inductively as follows:
\begin{description}
	\item[Axiom] The complexity of an atomic formula $P$ is 1
	\item[$f^*$] The complexity of $f^*P$ is the complexity of $P$
	\item[Unary Connectives] Other unary connectives ($\lnot$, $\lapb$)
		increase the complexity by 1
	\item [Binary Connectives] The complexity of $P \land Q$ is 
		one more than the maximum of the complexities of $P$ and $Q$. 
\end{description}
\end{definition}
We are, in this sequent calculus, not dealing with inference in the 
base category (that is, we treat data such as $f\extendedBy g$ 
as simply given), and this is consistent with our policy of treating 
$f^*$ as having no effect on the complexity of formula, and treating
the $f^*$ rule as having no effect on the height of proofs. 

We can now prove our inversion lemma: this will say that, if,
for example, we
have a proof of $\Gamma, f^*(P \land Q) \entails \Delta$, then we also
have proofs of $\Gamma, f^* P, f^* Q \entails \Delta$. We can prove such a
lemma in all cases except one: that case is when we have an 
existential quantifier on the right. In this case, directly inverting the 
$\lapb\text{R}$ would amount to giving a witness for the
existential quantifier, which, for well-known reasons, is impossible: 
we cannot permute the $\lapb\text{R}$ rule below $\lapb\text{L}$ or
$\land \text{R}$.

Note also that, apart from the missing case, the remaining classification is,
for rather trivial reasons, not exhaustive: 
we do not, for example, consider formulae like $f^* g^* (P \land Q)$, and we
also consider formulae of the form $\id^*P$ rather than those of the
form $P$. However, we will
need the inversion lemma to prove cut elimination, and, because of the way that 
we have defined the height of proofs, if we have a proof where the cutformula is 
$f^*g^*(P \land Q)$, we can find another one with subproofs of the same
height where the cutformula is $(gf)^*(P \land Q)$, and we can apply the inversion lemma
to \emph{this} proof.  
\begin{lemma}[Inversion]\label{lemma:perm}
\begin{enumerate}
	\item\label{item:invCaseAndL} If $\Gamma, P \land Q \entails \Delta$ is provable with a proof
of height $n$, then so is $\Gamma, P, Q \entails \Delta$,
and similarly for $\Gamma \entails P \lor Q, \Delta$, 
$\Gamma, \lnot P \entails \Delta$
and $\Gamma \entails \lnot P, \Delta$.
\item\label{item:invCaseAndfL} If $\Gamma, f^*(P \land Q) \entails \Delta$ is provable with a proof
of height $n$, then so is $\Gamma, f^*P, f^*Q \entails \Delta$,
and similarly for $\Gamma \entails f^*( P \lor Q), \Delta$.
\item\label{item:invCaseNotR} If $\Gamma \entails \lnot P, \Delta$ is
provable with a proof of height $n$, then
so is $\Gamma, P \entails \Delta$
\item\label{item:invCaseNotfR} If $\Gamma \entails f^*(\lnot P), \Delta$ is
provable with a proof of height $n$, then so is $\Gamma, f^* P \entails f^*\lfalse,  \Delta$.
\item\label{item:invCaseNotL} If $\Gamma, \lnot P \entails \Delta$ is provable with a proof
of height $n$, then so is 
$\Gamma \entails P, \Delta$.
\item\label{item:invCaseNotfL} If $\Gamma,  f^*(\lnot P)\entails \Delta$ is provable with a proof of height $n$,
then so is $\Gamma \entails f^* P, \Delta$.
\item\label{item:invCaseAndR} If $\Gamma \entails P\land Q, \Delta$ is provable with a proof of
height $n$, then so are $\Gamma \entails P, \Delta$ and $\Gamma \entails Q, \Delta$,
and similarly for $\Gamma, P \lor Q \entails \Delta$.
\item\label{item:invCaseAndfR} If $\Gamma \entails f^* (P \land Q), \Delta$ is provable with
a proof of height $n$, the so are $\Gamma \entails f^* P, \Delta$
and $\Gamma \entails f^* Q, \Delta$.
\item \label{item:invCaseExistsfL} If $\Gamma, f^* \lapb_g P \entails \Delta$ is provable with
a proof of height $n$, and if $f h \extendedBy g k$, for some $h,k$, then 
	\begin{displaymath}
		h^* \Gamma, k^*P \entails h^* \Delta
	\end{displaymath}
	is also provable with a proof of height $n$. 
\item\label{item:invCaseNotInsidef} If $\Gamma, f^* P \entails \Delta$ is provable with a proof of height $n$,
then so is $\Gamma, f^* \lfalse \entails \Delta$.
\end{enumerate}
\end{lemma}
\begin{proof}
This is more or less standard, with a few modifications because of
$f^*$: we will prove some illustrative cases. We first note
that the base case -- that is, axioms, $\lfalse\text{L}$, and $\ltrue\text{R}$
--
is trivial: none of the formulae in question can, in this case, be
principal, and so we can make what modifications we please. So 
we can concentrate on the inductive step.

\begin{description}
	\item[Case \ref{item:invCaseAndL}]\
The special cases are those in which the previous rule application is:
\begin{description}
  \item[$\land\text{L}$]\ with $P \land Q$ principal: we
  omit the last rule application. 
  \item[MWL]\ with $P \land Q$ principal: we weaken with the
  multiset $P, Q \cup \Gamma'$, where $\Gamma'$ is the
  multiset involved in the original weakening, apart from $P\land Q$.
  \item[$\lapb_f\text{L}'$]\ Here 
  $P \land Q$ cannot be principal, so we have
  a proof of the form 
  \begin{center}
	\AxiomC{$\Pi$}
	\noLine
	\UnaryInfC{$\vdots$}
	\noLine
	\UnaryInfC{$f^* \Gamma, f^*(P \land Q), R \entails f^* \Delta$}
	\RightLabel{$\lapb\text{L}'$}
	\UnaryInfC{$\Gamma, P \land Q, \lapb_f R \entails \Delta$}
	\DisplayProof
  \end{center}
		We apply Case~\ref{item:invCaseAndfL}  of the inductive hypothesis to the premise. 
	\item[$\circ^{-1}$] We apply Case~\ref{item:invCaseAndfL} of the inductive hypothesis to the
		premise
\end{description}
Otherwise,
$P \land Q$ persists unchanged from the premise(s), and we can apply
the inductive hypothesis to the premises.
\item[Case \ref{item:invCaseAndfL}]\ Here the special cases are  weakening, $f^*$,
		both of the $\extendedBy$ rules,  $\lapb\text{L}$, and $\lapb\text{R}$: 
in the first five cases we
apply the relevant inductive hypothesis to the premises, and then
the rule application. If we have $\lapb\text{L}$, then its premise
must be of the form 
\begin{align*}
  \hat g^* \Gamma, \hat g^*h^*(P \land Q), \hat f^* R &\entails \hat g^* \Delta;
  \intertext{we can apply the inductive hypothesis to
		the premise (together with $\circ$ and $\circ^{-1}$)
		to get}
  \hat g^* \Gamma, \hat g^*h^*P, \hat g^*h^*Q, \hat f^*R &\entails \hat g^* \Delta;
  \intertext{an application of $\lapb\text{L}$ gives
  us the conclusion.
	$\lapb\text{R}$ is very similar. 
	}
\end{align*}

Otherwise, $f^*(P \land Q)$ persists unchanged,
so we are done by induction.
\item[Case~\ref{item:invCaseNotR}]\ As above.
\item[Case~\ref{item:invCaseNotfR}]\ Here we need some care. The special cases are
$f^*$,  both cases of
$\extendedBy$, and $\lapb\text{L}$.
In the case of  $f^*$ we need care when $\Delta$ is empty. In this
case the inductive hypotheses, followed by an application of $f^*$, gives
us 
\begin{center}
  \AxiomC{$\vdots$}
  \noLine
  \UnaryInfC{$\Gamma, P \entails$}
  \RightLabel{$f^*$}
  \UnaryInfC{$f^*\Gamma, f^* P \entails f^* \lfalse$}
  \DisplayProof
\end{center}
and so we need the $f^* \lfalse$ on the right. 

With $\extendedBy\text{L}$, the premise must be of the form
\begin{align*}
  \Gamma &\entails g^* (\lnot Q), \Delta;
  \intertext{with $f \extendedBy g$. Inductively, we have}
  \Gamma, g^* Q &\entails g^*\lfalse, \Delta.
  \intertext{and now we can apply $\extendedBy \text{R}$ to conclude}
  \Gamma, f^* Q &\entails f^* \lfalse, \Delta.
\end{align*}
$\extendedBy\text{R}$ is similar.
$\lapb\text{L}$ is
	handled similarly to Case~\ref{item:invCaseAndfL}.
\end{description}
Cases~\ref{item:invCaseNotL} to \ref{item:invCaseAndfR} are similar to the above. 
\begin{description}
	\item[Case~\ref{item:invCaseExistsfL}]\ The special cases are $f^*$, $\extendedBy$, $\id$,
$\comp$,
and $\lapb\text{L}$ (with $f^* \lapb_g P$ both principal
and non-principal). We discuss each of them in turn.
\begin{description}
  \item[$f^*$]\ By hypothesis, the last inference of the proof looks as 
  follows:
  \begin{prooftree}
	\Axiom$\Gamma, g^*\lapb_h P \fCenter\entails \Delta$
	\RightLabel{$f^*$}
	\UnaryInf$f^* \Gamma, f^* g^*\lapb_h P \fCenter \entails f^*\Delta$ 
  \end{prooftree}
  What we have to show is that, if, for some $k$, $l$, we have $gfk \extendedBy h l$,
  then we have a proof of $k^* f^* \Gamma, l^* P \entails k^* f^* \Delta$; but 
  this follows immediately from the inductive hypothesis. 
	\item [$\extendedBy$]\ Here the last step is
		\begin{prooftree}
			\AxiomC{$f \extendedBy g$}
			\AxiomC{$\Gamma, f^*\lapb_h P \entails \Delta$}
			\BinaryInfC{$\Gamma, g^*\lapb_h P \entails \Delta$}
		\end{prooftree}
  where $f_1 \extendedBy f_2$.
		Suppose that $gk \extendedBy hl$: we want a proof of 
		$k^* \Gamma, l^*P \entails k^* \Delta$.  
		However, if $gh \extendedBy hl$, then $fh \extendedBy hl$, and, by
		induction, we have the needed proof
		directly.
  \item[$\id$]\ 
  Here we have a proof ending
	   \begin{prooftree}
	\Axiom$\Gamma, \lapb_g P \fCenter\entails \Delta $
	\RightLabel{$\id $}
	\UnaryInf$\Gamma, \id^* \lapb_g P \fCenter \entails \Delta $
  \end{prooftree}
		Suppose that $\id h \extendedBy g k$: then, since $h \extendedBy gk$, 
		we can use Case 9 inductively to get a proof of $h^* \Gamma, k^* P \entails
		h^* \Delta$, which is what we require. 
  \item[$\lapb\text{L}$]\ There are two cases, depending on whether
			$f^*\lapb_g P$ is principal in the last inference or not.
			
			If $f^* \lapb_g P$ is principal in this
  rule application, then the premise must be
  $\hat g ^* \Gamma, \hat f^* P \entails \hat g^* \Delta$. Suppose now
  that $f h \extendedBy g k$: then, by the universal property of
  the comma object, $\hat g \langle h, k\rangle  = h \overline h\,\overline{k}$
  and $\hat f \langle h, k\rangle = k \overline h\,\overline k$. 
  So, pulling back by $\langle h, k\rangle$, we have
  \begin{displaymath}
	  \overline h^*\overline k ^* h^* \Gamma, \overline h^*\overline k^* k^* P \entails\overline h^* \overline  k^* h^* \Delta
  \end{displaymath}
  from which we can derive $h^*\Gamma, k^*P \entails h^* \Delta$ by $\extendedBy \text{LR}$, since 
  $\overline h \overline k \extendedBy \id$. 

		Otherwise, another formula -- say
  $f_1^*\lapb_{g_1}P_1$ is principal, and so we have 
  \begin{equation}
	\Axiom$\hat g_1^*\Gamma,  \hat f_1^*P_1, \hat g_1^*f^*\lapb_g P \fCenter \entails 
	  \hat g_1^*\Delta $
	\RightLabel{$\lapb\text{L}$}
	\UnaryInf$\Gamma, f_1^* \lapb_{g_1}P_1, f^*\lapb_g P\fCenter \entails \Delta$
	\DisplayProof\label{eq:badexistind}
  \end{equation}
		What we have to show is that, if $fh \extendedBy gk$, then
		\begin{displaymath}
			h^* \Gamma, h^* f_1^* \lapb_{g_1}^* P_1, k^* P \entails h^* \Delta
		\end{displaymath}
	
		Consider the diagram
  \begin{displaymath}
	\begin{xy}
				\place(450,750)[=]
				\place(200,385)[=]
				\place(450,75)[\extends]
				\morphism(900,500)|r|<0,-500>[A` B;f]
				\morphism(0,-300)|b|<900,300>[C` B;g]
				\morphism(0,200)<0,-500>[D` C;k]
				\morphism(0,500)<0,-300>[D_0` D;i]
				\morphism(0,200)|b|<900,300>[D` A;h]
	  \square(900,500)|arrb|<900,500>[f_1 \comma g_1`C_1`A`B_1;
	 \hat  f_1` \hat g_1`g_1`f_1]
			 \place(1350,750)[\extendedBy]
			 \square(0,500)|alra|<900,500>[D_0 \otimes_{A} (f_1 \comma g_1)` f_1 \comma g_1` D_0`A;
				\pi' ` \pi` \hat g_1 ` h_0]
	\end{xy}
 \end{displaymath}
	 Here we have factorised $h$ as $h_0 j$, with $h_0$ total; let $i \dashv j$, so that
	 $hi = h_0 j i = h_0$. The top right
	 hand rectangle is a comma rectangle, generated by $f_1$ and $g_1$; the top left rectangle is
	 cartesian. Pasting in the diagram gives $f \hat g_1   \pi' \extendedBy gki\pi$. Consequently 
	 we can assume, by the inductive hypothesis, that 
	 \begin{align*}
		\pi'^* \hat g_1^* \Gamma, \pi'^* \hat f_1^* P_1, \pi^* i^* k^*  P &\entails 
		\pi'^* \hat g_1^* \Delta
		\intertext{and so, by the commutativity of the top left rectangle,}
		\pi^* h_0^* \Gamma, \pi'^* \hat f_1 P_1, \pi^* i^* k^* P & \entails
		\pi^* h_0^*\Delta
		\intertext{whence, by $\lapb \text{L}$ applied to $P_1$, since the top rectangle is
		a comma square,}
		h_0^* \Gamma, h_0^* f_1^* \lapb_{g_1} P_1, i^* k^* P & \entails 
		h_0^* \Delta
		\intertext{and, when we pull back by $j^*$, we get }
		h^* \Gamma, h^* f_1^* \lapb_{g_1} P_1, k^* P & \entails 
		h^*\Delta
		\intertext{which was what we had to prove.}
	 \end{align*}

\end{description}
\item[Case~\ref{item:invCaseNotInsidef}]\ Trivial induction.
\end{description}

\end{proof}

For the proof of cut elimination, we need the following lemma:
\begin{lemma}
	If $\Gamma \entails \Delta$ has a cut free proof of
	depth $n$, then, for any appropriate $\phi$, there is
	a cut free proof of $\phi^* \Gamma \entails \phi^* \Delta$. 
	\label{lemma:proofPullback}
\end{lemma}
\begin{proof}
	Trivial: all of the rules are stable under pullback.
\end{proof}

\begin{theorem}
	The system allows cut elimination
	\label{thm:cutElim}
\end{theorem}
\begin{proof}
	We proceed by induction on the depth of the proof and the degree of the formula. So we
	assume first that we have a cut of the form 
	\begin{prooftree}
		\AxiomC{$\displaystyle \Pi$}
		\noLine
		\UnaryInfC{$\vdots$}
		\noLine
		\UnaryInfC{$\Gamma \entails P, \Delta$}
		\AxiomC{$\displaystyle \Pi'$}
		\noLine
		\UnaryInfC{$\vdots$}
		\noLine
		\UnaryInfC{$\Gamma', P \entails \Delta'$}
		\BinaryInfC{$\Gamma, \Gamma' \entails \Delta', \Delta$}
		\RightLabel{\text{cut}}
	\end{prooftree}
	Using Lemma~\ref{lemma:proofPullback}, we can assume that $\phi^*$ is not used in either
	$\Pi$ or $\Pi'$. Using Lemma~\ref{lemma:perm}, we can also assume that, except in the
	case where $P = f^*\lapb_g Q$, $P$ is principal on both sides or an axiom: in these cases, 
	the cut can be replaced with one of lower degree or eliminated. 
	
	So we are left with the case where the cut formula is $f^* \lapb_g Q$: we argue by 
	cases on the bottom inference on the left. In the cases where the cutformula is
	unchanged by the inference (including contraction), we simply move the cut upwards. 
	So we are left with the following cases:
	\begin{description}
		\item[Axiom]\ We can directly eliminate the cut in the usual way
		\item[$\lapb \text{L}$]\ In this case the bottom inference on the left is
			\begin{prooftree}
				\AxiomC{$\hat g_1^* \Gamma, \hat f_1^* P_1 \entails \hat g_1^* \hat f^* \lapb_g Q, \hat g_1^*\Delta$}
				\UnaryInfC{$\Gamma, f_1^* \lapb_{g_1} P_1 \entails f^* \lapb_g Q, \Delta$}
			\end{prooftree}
			Here we can apply Lemma~\ref{lemma:proofPullback} to the proof on the right and move the
			cut upwards
		\item[$\lapb \text{R}$]\ Here the final inference on the left is of the form
			\begin{prooftree}
				\AxiomC{$\Gamma \entails \tau^* Q, \Delta$}
				\UnaryInfC{$\Gamma \entails f^*\lapb_g Q, \Delta$}
			\end{prooftree}
			and we can apply Lemma~\ref{lemma:perm} to the cut formula on the right
			to replace the cut with a cut on $Q$, which has lower degree.
	\end{description}
\end{proof}

\section{Semantics}
\subsection{Definitions}
The semantics of this logic should be as follows:
  Let $\ccont$ be a partial cartesian category with comma
  objects. Consider a 
  category $\econt$ 2-fibred by Boolean algebras
  and $\ltrue, \land, \lor$-preserving poset
  morphisms over $\ccont$: 
  the reindexing functors $f^*$
  should have 
  left  adjoints $\lapb_f$,
  which should satisfy the Beck-Chevalley
  conditions with respect to comma squares in Section~\ref{section:BeckChevalley}. 
  Furthermore,
  if we have $f \extendedBy g$ for 1-cells $f$ and $g$ in
  $\ccont$, we should have $g^* \leq f^*$ in the 
  pointwise order on poset morphisms. Now let
  $\lang$ be a logic as described above, typed by the objects
  and 1-cells of $\ccont$. 

  Note that the poset morphisms of Boolean algebras which we 
  consider are, when considered with the usual coHeyting structure
  on those algebras, weak coHeyting morphisms: this will allow us to use the
  results of Section~\ref{section:Frobenius}.

\begin{definition}
  An \emph{assignment} is an choice, for every $t \in
  \text{Ob}(\ccont)$ and every atomic $P \in \lang$, of 
  an element
  $\semval{P}{t} \in \text{Ob}(\econt_t)$.
\end{definition}  
Given an assignment, we can define, for each $P:t$, its semantic value
$\semval{P}{t}$ by induction on its syntactic complexity: the
sentential operators are interpreted in the usual way, $f^*$
is interpreted as the reindexing functor (also written $f^*$), and
$\lapb_f$ is interpreted as the left adjoint
to reindexing.

We also define $\semval{\Gamma}{t}$ for contexts: here
we have to make a distinction between left and
right contexts, since the comma is
interpreted differently on the left and on the right. 
\begin{definition}
  The semantic value of a left context is given by the clauses
  \begin{itemize}
    \item $\semval{\Gamma}{t} = \semval{P}{t}$ if $\Gamma =P$
    \item $\semval{\Gamma, \Gamma'}{t} = \semval{\Gamma}{t}
      \land \semval{\Gamma'}{t}$
  \end{itemize}
  The semantic value of a right context is given by the
  clauses
  \begin{itemize}
    \item $\semval{\Delta}{t} = \semval{P}{t}$ if $\Delta =P$
    \item $\semval{\Delta, \Delta'}{t} = \semval{\Delta}{t}
      \lor \semval{\Delta'}{t}$
  \end{itemize}
\end{definition}

And, finally, a definition of semantic entailment:
\begin{definition}
  \begin{align*}
    \Gamma:t &\semEnt \Delta:t
    \intertext{iff} 
    \semval{\Gamma}{t} & \leq \semval{\Delta}{t}
  \end{align*}
  for every assignment.
\end{definition}

\subsection{Soundness}
The proof of this is very standard.
\begin{proposition}
  The rules $\text{LW}$, $\text{RW}$, $\text{LC}$, $\text{RC}$,
  $\lor\text{R}$,  $\land\text{L}$, $\lfalse\text{L}$ and $\ltrue\text{R}$
  are sound for $\semEnt$. 
\end{proposition}
\begin{proof}
  Standard.
\end{proof}
\begin{proposition}
  The rules $\lor\text{L}$ and $\land\text{R}$ are sound
  for $\semEnt$.
\end{proposition}
\begin{proof}
  This follows from the distributivity of $\land$ over $\lor$
  and vice versa.
\end{proof}

\begin{proposition}
  The rules $\lnot\text{L}$ and $\lnot\text{R}$ are sound for
  $\semEnt$. 
\end{proposition}
\begin{proof}
	Standard.
\end{proof}

\begin{proposition}
	$f^* \lnot \text{L}$ and $f^* \lnot \text{R}$.
	are sound for $\semEnt$.
\end{proposition}
\begin{proof}
	Note first that, because $f^*$ is a poset morphism of Boolean algebras
	preserving $\ltrue, \land$ and $\lor$, we have $f^* \lnot A = \lnot f^* A \lor f^* \lfalse$. 
	The soundness of the rules follows from this. 
\end{proof}
\begin{proposition}
	$\extendedBy \text{L}$, $\extendedBy\text{R}$ and
	$\extendedBy \text{LR}$ are sound for $\semEnt$. 
\end{proposition}
\begin{proof}
	The first two are immediate: for the third, we use Corollary~\ref{corollary:coHeyting}. 
\end{proof}
\begin{proposition}
  $\semEnt$ is sound for
  $\lapb{f} \text{L}$ and $\lapb{f}\text{R}$ 
\end{proposition}
\begin{proof}
	For $\lapb{f}\text{R}$, this follows from the unit of the adjunction 
	together with the semantics of $\extendedBy$. For $\lapb{f}\text{L}$,
	we use the counit of the adjunction together with Beck-Chevalley.
\end{proof}

Finally
\begin{proposition}
  The cut rule is sound for $\semEnt$.
\end{proposition}
\begin{proof}
	Standard.
\end{proof}
So, putting all these results together, we have
\begin{theorem}
  $\semEnt$ is sound for our sequent calculus. 
\end{theorem}

\subsection{Completeness}
\begin{theorem}\label{thm:classComplete}
  The semantics is complete: that is, if,
  for a given base category $\ccont$, and for
  an object $t$ of $\ccont$,
  \begin{align*}
    \semval{\Gamma}{t} &\leq \semval{\Delta}{t}
    \intertext{for two contexts $\Gamma:t$ and
      $\Delta:t$, then}
    \Gamma \entails \Delta. 
  \end{align*}
\end{theorem}
This theorem will be proved by constructing a term, or
generic, model, which we define as follows.
\begin{definition}
  Let $\ccont$ be a category with fibre products.
  The term model, $\termMod{\ccont}$, over $\ccont$ is given by
  the following data:
  \begin{description}
	  \item[Objects] these are given 
      by pairs
      $P:s$, where 
      $s$ is an object of $\ccont$ and $P$ is a proposition
      of type s
  \item[Morphisms] a morphism  between $P:s$ and $Q:t$ is
      given by a proof
      \begin{displaymath}
        P \entails f^* Q
      \end{displaymath}
      for some morphism $f: s \rightarrow t$ of $\ccont$.
      Two such morphisms are equal iff their source and
      target are the same, and the corresponding 
      morphisms of $\ccont$ are equal. 
  \item[Composition] suppose we have two morphisms corresponding to
      proofs 
      \begin{displaymath}
		  \AxiomC{$\Pi$}
		  \noLine
		  \UnaryInfC{$\vdots$}
		  \noLine
		  \UnaryInfC{$P:s \entails f^*(Q:t)$}
		  \DisplayProof
        \qquad \text{and}\qquad
		\AxiomC{$\Pi'$}
		\noLine
		\UnaryInfC{$\vdots$}
		\noLine
		\UnaryInfC{$Q:t \entails g^*(R:u)$}
		\DisplayProof
      \end{displaymath}
      Their composition is given by the proof
      \begin{displaymath}
		  \AxiomC{$\Pi$}
		  \noLine
		  \UnaryInfC{$\vdots$}
		  \noLine
		  \UnaryInfC{$ P:s \entails f^*(Q:t)$}
		  \AxiomC{$ \Pi'$}
		  \noLine
		  \UnaryInfC{$\vdots$}
		  \noLine
		  \UnaryInfC{$Q:t \entails g^*(R:u)$}
		  \UnaryInfC{$f^* Q \entails f^* g^* R$}
		  \RightLabel{ cut}
		  \BinaryInfC{$P \entails f^* g^* R$}
		  \UnaryInfC{$P \entails (gf)^* R$}
		  \DisplayProof
      \end{displaymath}
  \item[Identity morphisms] these are given by the proofs
      \begin{displaymath}
		  \AxiomC{}
		  \UnaryInfC{$P:t \entails P:t$}
          \UnaryInfC{$P:t \entails \id^* P$}
		  \DisplayProof
      \end{displaymath}
  \item[2-cells] Homsets are posets, and  there is a 2-cell between $f,g:P:s \rightarrow Q:t$
	  iff $f \extendedBy g$. 
  \item[The display functor] this is the
      map $p$ which sends 
      a typed proposition $P:t$ to the object $t$,
      and a proof of $P:s \entails f^*(Q:t)$
      to the morphism $f:s \rightarrow t$. 
  \item[ Liftings] we lift 1-cells
      as follows. Let $f: s \rightarrow t$
      be a morphism in the base, and let $P:t$ be an object
      of the fibre $\econt_t$ over $t$: let the lifting
      of $f$ be the following proof:
      \begin{displaymath}
		  \AxiomC{}
		  \UnaryInfC{$f^* P \entails f^* P$}
		  \DisplayProof
      \end{displaymath}
  \item[Adjoints] Left 
   adjoints to the substitution functors $f^*$
      are given by $\lapb{f}$. 
	  \end{description}
\end{definition}
We now prove
\begin{proposition}
  $\termMod{\ccont}$ is 2-fibred in Boolean algebras over $\ccont$: the
  pullbacks along 1-cells are $\ltrue, \land, \lor$-preserving poset
  morphisms. Pullbacks along 1-cells furthermore possess left adjoints
  satisfying Beck-Chevalley with respect to comma squares.
\end{proposition}
\begin{proof}
  It is clear than $\termMod{\ccont}$ is a 
  locally posetal 2-category
  (equality of 1-cells is so strong that laws like
  associativity follow directly from the corresponding
  laws for $\ccont$). It is likewise clear
  that our ``display functor'', $p$, is actually 
  a functor. We have to check that
  the liftings are Cartesian: so, consider 
  composable morphisms $\alpha:s \rightarrow t$
  and $\beta: t \rightarrow u$ in the base,
  together with a proof $\Pi$ of 
  $P:s \entails \contranest{R:u}{\beta\compose\alpha}$
  lying over $\alpha\compose\beta$. We need to produce a proof of 
  $P:s \entails \contranest {(\wPre{\alpha}R):t}{\alpha}$
  (commutativity of the resulting diagram is trivial). 
  But this is immediate:
  \begin{displaymath}
	  \AxiomC{$\Pi$}
	  \noLine
	  \UnaryInfC{$\vdots$}
	  \UnaryInfC{$ P:s \entails gf^*(R:u)$}
	  \UnaryInfC{$P:s \entails f^*(g^* R)$}
	  \DisplayProof
  \end{displaymath}
  This establishes the functoriality of the $\wPre{\alpha}$. 
  The universal property for 2-cells follows directly from
  the corresponding universal property in the base. Consequently,
  $\termMod{\ccont}$ is 2-fibred over $\ccont$: the fibres
  are Boolean algebras, because the inference rules are a superset
  of the normal classical inference rules. 
  
  We need to show that $\lapb{f}$ is a left
  adjoint  $f^*$. Functoriality is
  easy.  For example, the following construction, which produces a
  proof of $\lapb{f} P \entails \lapb{f} Q$ from a proof of
  $P \entails Q$, establishes functoriality for $\lapb{f}$:
  \begin{displaymath}
	  \AxiomC{$\Pi$}
	\noLine
	\UnaryInfC{$\vdots$}
          \noLine
		  \UnaryInfC{$ P:s \entails Q:s$}
		  \UnaryInfC{$P \entails f^*\lapb{f}P $}
		  \DisplayProof
  \end{displaymath}
  (where we have used the admissible rules of Lemma~\ref{lemma:admissibleRules}). 

  Given functoriality, we only need to establish the unit and counit
  for the adjunction. The unit for $\lapb{f} \dashv
  f^*$ is proven using the (admissible) $\lapb{f}\text{R}'$,
  and the counit using $\lapb{f}\text{L}'$.

  Finally, we must verify the Beck-Chevalley conditions:
  for $\lapb{f}$, we prove these as follows. 
  Given the usual comma square, we have to prove that
  $f^* \lapb{g}P $ and $\lapb{\hat g}\hat f^* P$ are
  equivalent. The direction
  \begin{displaymath}
	  \lapb{\hat g} \hat f^* P \entails f^* \lapb{g} P
  \end{displaymath}
  follows from functoriality and the adjunction. We prove the other direction
  as follows:
  \begin{displaymath}
	  \AxiomC{}
	  \UnaryInfC{$\hat f^* P \entails \hat f^* P$}
	  \RightLabel{$\lapb{\hat g}\text{R}'$}
	  \UnaryInfC{$\hat f^* P \entails \hat g^*\lapb {\hat g}\hat f^* P$}
	  \RightLabel{$\lapb{f}\text{L}$}
	  \UnaryInfC{$f^*\lapb{g} P \entails \lapb{\hat g}\hat f^* P$}
	  \DisplayProof
  \end{displaymath}
  This
  concludes the proof that $\termMod{\ccont}$ is a
   category fibred over $\ccont$ with the desired
   properties.
\end{proof}
\begin{definition}
  Let $\Gamma$ be a left context: let the \emph{propositionalisation}
  of $\Gamma$, $\contraction{\Gamma}$, be defined  
  as the conjunction of its members.
 The conjunction of a right context is the disjunction of its
 memebrs.
\end{definition}
\begin{lemma}
  For any $\Gamma$ and $\Delta$,
  \begin{align*}
    \Gamma &\entails \Delta
    \intertext{iff}
    \contraction{\Gamma} & \entails \contraction{\Delta}
  \end{align*}
\end{lemma}
\begin{proof}
  The obvious induction.
\end{proof}
\begin{proof}[Proof of Theorem~\ref{thm:classComplete}]
  Suppose that $\Gamma:t \semEnt \Delta:t$. Define an
  interpretation of the language in  $\termMod{\ccont}$
  by sending $P:t$ to $P:t$ as an object of $\termMod{t}$.
  We establish, by induction, that, with respect to this
  interpretation, $\semval{\Gamma}{t} = \contraction{\Gamma}$,
  and $\semval{\Delta}{t} = \contraction{\Delta}$. Since
  $\Gamma:t \semEnt \Delta:t$, we must have 
  $\semval{\Gamma}{t} \leq \semval{\Delta}{t}$, and, consequently,
  $\contraction{\Gamma} \leq \contraction{\Delta}$: by
  the definition of $\termMod{\ccont}$, this means that
  $\contraction{\Gamma} \entails \contraction{\Delta}$.  By
  the lemma,  we have $\Gamma \entails \Delta$. 
\end{proof}
\subsubsection{Kripke Models}

We can make the model theory somewhat more specific in
the following way.
\begin{definition}
  A \emph{Reiter Kripke model} over a category $\ccont$ is
  a covariant functor $\Psi$ from $\ccont$ to the category of
  sets and partial functions. Given a Reiter Kripke model,
  the \emph{associated Reiter category} is the category,
  fibred in Boolean algebras over $\ccont$, where
  the fibre over an object $t$ 
  of $\ccont$ is the powerset of $\Psi(t)$, and where the
  pullback morphism over $\alpha: s \rightarrow t$ is
  given by the inverse image of subsets of $\Psi(t)$. 
\end{definition}
Implicit in the above definition is 
\begin{lemma}
  The associated Reiter category of a Reiter Kripke model
  is a Reiter category. 
\end{lemma}
\begin{proof}
  Routine calculation. 
\end{proof}
We can, then, define a Kripke semantics for our logic: 
given a Reiter Kripke model over a category $\ccont$,
and an object $t$ of $\ccont$, we will call the
elements of $\Psi(t)$ the \emph{possible worlds of type $t$}. 
We will assign sets of possible worlds to atomic propositions,
and it is clear how to define, inductively, the
semantic value of general propositions. 

Since each Reiter Kripke model defines a Reiter category,
it is clear that the Kripke semantics is sound. It is 
also complete:
\begin{proposition}
  The Kripke semantics is complete for our logic. 
\end{proposition}
\begin{proof}
  Suppose that we have $\Gamma:t  \notentails \Delta:t$. By
  our soundness theorem, we have a Reiter category model
  with $\semval{\Gamma}{t} \not \leq \semval{\Delta}{t}$. 
  By Stone duality \citep{johnstone82:_stone_spaces}, we
  can assume that 
  the Reiter category model is given
  by a covariant functor $\Psi$ from $\ccont$ to the
  category of Stone spaces and continuous maps: the
  fibre over an object $t$ will be the algebra of clopen 
  sets of $\Psi(t)$, and the pullbacks 
  along $\alpha:s \rightarrow t$ will be given
  by inverse images along the continuous morphism $\Psi(\alpha)$.
  If we now compose $\Psi$ with the underlying set functor, we get
  a Kripke model: it is trivial to verify that 
  the construction of semantic values
  commutes with the underlying set functor, and so, as required,
  we have a Kripke model in which
  \begin{displaymath}
    \semval{\Gamma}{t} \not \leq \semval{\Delta}{t}.
  \end{displaymath}
\end{proof}

\section{Internal Languages and Applications}

We can now use the machinery that we have used in order to 
investigate partial Cartesian categories by studying their domain fibrations. 
We obtain internal languages, both for partial Cartesian categories and also
for partial Cartesian categories with comma objects: we show that the internal language
of the former coincides with a logic of partial functions independently defined by
Palmgren and Vickers~\citet{VickersS:PHLCC}, whereas the internal logic of the 
latter is an extension of the Palmgren-Vickers logic with a Heyting operation.
Using this logic, we show how to present partial Cartesian categories, and 
partial Cartesian categories with comma objects, by means of signatures 
consisting of generators and relations. This, in turn, allows us
to express our original fibred logic with more explicit expressions for
the objects and morphisms of the base (and, in fact, with admissible
comprehension rules for internal equalities); so, finally, we can
give a formalisation of one of the key philosophical examples which 
motivated this work, namely Davidson's argument for the meaningfulness of talk
of equality of actions~\cite[p.~109]{00DavidsonD:logfas}.
So this section will be an explicit construction of
the fibrations which we have been considering abstractly in 
the previous part of this paper.

Firstly, some clarification. We have, in Section~\ref{section:domFib},
defined both 2-covariant and 2-contravariant domain fibrations: the 2-covariant domain
fibration will turn out to be the one appropriate for the semantics of internal
languages as they are usually conceived. 
\subsection{Bicategories of Partial Maps}
\subsubsection{Primitives}
We will start with the case of partial cartesian categories, i.e.\ 
bicategories of partial maps with weak products. 
The primitives of our internal languages can be motivated 
as follows. 
As we describe above (p.~\pageref{ref:actPartFun}), we describe
actions by means of partially defined functions which take possible
worlds, at the state before the action is performed, to possible
worlds at the state after the action is performed. 
We will, loosely following Scott
\citeyearpar{scott:_ident_exist_intuit_logic}, use a partially defined
equality relation to reason about partially defined functions.  

So, if we have two actions, $\alpha(x:s)$ and $\beta(y:t)$, whose
values are possible worlds at the same state, then $\alpha(x) \partEq
\beta(y)$ is a partially defined binary relation between possible
worlds at different states, $s$ and $t$. Clearly, if we go on like
this, we will need $n$-ary relations between possible worlds: it is
easiest to work with $n$-tuples of worlds, of the form
\begin{displaymath}
  \langle x_1:s_1, \ldots, x_n:s_n\rangle,
\end{displaymath}
and we shall write such a tuple, in boldface, as $\vec{x}:\vec{s}$,
or, where the states are clear from the context, as $\vec{x}$.  We
will also be interested in $k$-tuples of actions
\begin{displaymath}
  \langle \alpha_1(x_{i_1}), \ldots, \alpha_k(x_{i_k}) \rangle,
\end{displaymath}
or $\vec{\alpha}(\vec{x})$, where each of the $\alpha_i$ will have an
argument which is one of the $x_i$.

Our
intended notion of partial equality will be as follows:
as we shall show
in Section~\ref{sec:EqInFib}, it can also be defined
in terms of the abstract structure of partial cartesian 
categories.
\begin{definition}\label{def:intSemPartEq}
  Let $\vec{\alpha}(\vec{x})$ and $\vec{\beta}(\vec{y})$ 
  be two partial functions whose values are possible
  worlds at the same tuple of states:
  then 
  \begin{displaymath}
    \vec{\alpha}(\vec{x}) \partEq \vec{\beta}(\vec{y}) 
  \end{displaymath}
  is true at worlds $\vec{x}$ and $\vec{y}$ iff
  \begin{enumerate}
    \item $\vec{\alpha}$ and $\vec{\beta}$ are both
      defined at those worlds, and
    \item the values of $\vec{\alpha}$ and $\vec{\beta}$
      are equal. 
  \end{enumerate}
\end{definition}

This notion of equality can be axiomatised by a system due to Palmgren and
Vickers
\citeyearpar{VickersS:PHLCC}:
we give in in Table~\ref{tab:PalVicEnt}.
Notice that the equality $\vec{f}(\vec{x}) \partEq \vec{f}(\vec{x})$,
which we shall abbreviate to $\converges{\vec{f}}\,$, is true iff the
functions of $\vec{f}$ are all defined at $\vec{x}$.

\subsubsection{Formulae, Contexts, and Sequents}
We will first use this language to define 
locally posetal bicategories: 
we will then show that these bicategories are, firstly, models of the 
language, and, secondly, free bicategories of partial maps.
\begin{definition}
  A \emph{signature}, $\Sigma = (I,J)$, will be a pair of sets: $I$ is
  a set whose elements will be called \emph{situations}, and $J$ is a
  set whose elements will be called \emph{actions}: each action will
  be associated to a pair of situations, its source and its target.
\end{definition} 

We want to talk about partially defined equalities between
tuples of actions, so, on the basis of our signature, we define a language
$\lang_{\Sigma}$ for reasoning about such equalities. 
\begin{definition}
  Suppose that we have, for each situation $s \in I$, a supply of
  variables $x:s, x':s, \ldots$ of type $s$. A \emph{variable tuple},
  written $\vec{x}:\vec{s}$, will be a tuple of variables 
  $(x_1:s_1, \ldots, x_k:s_k)$: the zero length tuple will be written
  $()$. An \emph{action tuple}, written $\vec{\alpha}(\vec{x}:\vec{s})$,
  will be a tuple of actions $(\alpha_1, \ldots, \alpha_l)$: the source
  $s_{i_j}$ of $\alpha_i$ must be one of the $s_i$. $\vec{s}$ will
  be called the \emph{source type} of the action tuple: its 
  \emph{target type} will be $(t_1, \ldots, t_l)$, where, for all $i$, 
  $t_i$ is the target of $\alpha_i$. 

  Given a pair of action tuples $\vec{\alpha}(\vec{x})$ and
  $\vec{\beta}(\vec{x})$ with the same source and target types, the
  \emph{partial equation} (or \emph{partial equation tuple})
  $\vec{\alpha}(\vec{x}) \partEq \vec{\beta}(\vec{x})$ will be defined
  to be the tuple of partial equations
  \begin{displaymath}
    \alpha_1 \partEq \beta_1, \ldots, \alpha_l = \beta_l. 
  \end{displaymath}
  Given a tuple of situations, we define $\ltrue$ at that source typel
  to be the empty tuple of partial equations: its target type is, of
  course, the empty tuple of situations.  Given two partial equations
  $\vec{\zeta}(\vec{x}: \vec{s})$ and $\vec{\eta}(\vec{x}:\vec{s})$,
  with the same source type, we define $\vec{\zeta} \land \vec{\eta}$
  to be the concatenation of the corresponding tuples. We will
  abbreviate the equation $\vec{\alpha}(\vec{x}) \partEq
  \vec{\alpha}(\vec{x})$ to $\converges{\vec{\alpha}(\vec{x})}$.

  We define the entailment relation $\vec{\zeta}(\vec{x}: \vec{s})
  \entailsPE \vec{\eta}(\vec{x}:\vec{s})$, between partial equations
  with the same source tuples, by the rules in Table~\ref{tab:PalVicEnt}.
  We should note that, in all of the rules with premises, the 
  source tuples of premises and conclusion must be the same. 

  A \emph{theory}, $\someTheory$, will be a set of equation tuples.
  As usual, we say that two theories are the same if they have the
  same closure under entailment.
\end{definition}

\begin{example}
  If we have a variable tuple $\vec{x}:\vec{s}\: = \: x_1:s_1,
  x_2:s_2, x_3:s_3$, then
  \begin{displaymath}
    \vec{\alpha}(\vec{x}:\vec{s}) \: = \: 
    \alpha_1(x_2:s_2),\alpha_2(x_3:s_3), \alpha_3(x_1:s_1),\alpha_4(x_2:s_2)
  \end{displaymath}
  is a valid action tuple of that source type. So also is the
  empty action tuple $()$: this shows that the source type of an action
  tuple is a structure assigned to it, rather than a property.
\end{example} 

We can now define a two-category of contexts for our logic. 
\begin{definition}
  Given a signature $\Sigma$, together with a theory $\someTheory$
  in that signature, we define a locally posetal two-category
  $\freeCat{\Sigma}{\someTheory}$ as follows. 
  \begin{description}
  \item[objects] An object of $\freeCat{\Sigma}{\someTheory}$
    will be written $\obC{\zeta}{x:s}$, where 
    $x:s$ is a variable tuple and where $\zeta$ is a partial
    equation tuple with that source type.
  \item[1-cells] A 1-cell will be
    written
    \begin{displaymath}
      \morphC{\alpha(x:s)}{\zeta(x:s)}: 
      \obC{\eta}{x:s} \rightarrow \obC{\vartheta}{y:t}
    \end{displaymath}
    where we require that
    \begin{enumerate}
    \item $\vec{s}$ is the source type of
      $\vec{\alpha}$, of $\vec{\zeta}$ and of $\vec{\eta}$, 
    \item $\vec{t}$ is the target type of
      $\vec{\alpha}$ and the source type of 
      $\vec{\vartheta}$, and
    \item the  entailment
      \begin{displaymath}
        \someTheory,
        \vec{\zeta}, \vec{\eta}, \converges{\vec{\alpha}}
        \: \entailsPE \: 
        \vec{\vartheta}(\vec{\alpha})
      \end{displaymath}
      holds (informally, we require that 
      $\morphC{\alpha(x:s)}{\zeta(x:s)}$ should factor through the
      subobject of $\{\vec{y}:\vec{t}\}$ defined
      by $\vec{\vartheta}$). 
    \end{enumerate}

    Composition of 1-cells is defined as follows. Suppose that 
    \begin{align*}
      \morphC{\alpha(x)}{\zeta(x)}:\:
      &\obC{\vartheta}{x:s} \rightarrow \obC{\vartheta'}{y:t}\\
      \morphC{\beta(y)}{\eta(y)}:\:
      &\obC{\vartheta'}{y:t} \rightarrow \obC{\vartheta''}{z:u}&\text{then}\\
      \morphC{\beta(y)}{\eta(y)}\morphC{\alpha(x)}{\zeta(x)}
      \:&=\:
      \morphC{\beta(\alpha(x))}{\eta(\alpha(x)),
        \zeta(x)}
    \end{align*}
    The unit 1-cell, on an object $\obC{\zeta}{x:s}$, is
    $\morphC{x}{\zeta(x)}$. 
    
  \item[2-cells] The 2-cells of our category will
    be defined as follows. Given 
    \begin{align*}
      \morphC{\alpha}{\zeta}, 
      \morphC{\alpha'}{\zeta'}&:
      \obC{\eta}{x} \rightarrow \obC{\vartheta}{y}
      \intertext{we say that $\morphC{\alpha}{\zeta}\extendedBy
      \morphC{\alpha'}{\zeta'}$ iff}
      \someTheory,
      \converges{\vec{\alpha}},\vec{\zeta},
      \vec{\eta}\:\: &\entailsPE\, \vec{\alpha}
      \partEq \vec{\alpha'}\land \vec{\eta'};   
    \end{align*}
  informally, whenever $\morphC{\alpha}{\zeta}$ is defined, then
  so too is $\morphC{\alpha'}{\zeta}$ and they are equal. 
  \end{description}
\end{definition}
We also define the following subcategory of $\freeCat{\Sigma}{\someTheory}$:
we will need it for our proof of freeness. 
\begin{definition}
  Let $\freeUnsplitCat{\Sigma}{\someTheory}$ be the full subcategory
  of $\freeCat{\Sigma}{\someTheory}$ whose objects have no constraints,
  i.e.\ are all of the form $\{\vec{x}:\vec{s}|\:\}$.
\end{definition}
After some calculation, we can prove
\begin{lemma}
  $\freeCat{\Sigma}{\someTheory}$ and 
  $\freeUnsplitCat{\Sigma}{\someTheory}$
  are locally posetal bicategories.  
\end{lemma}
\begin{remark}
	The system of Table~\ref{tab:PalVicEnt} is due to
  \citet{VickersS:PHLCC}, and is a sound and
  complete axiomatisation for equalities between partially-defined
  functions between sets. We will generally not be pedantic about
  notation: we will often omit types (or, indeed, variables)
  where it is obvious from the context. We will also frequently
  write the object $\{\vec{x}:\vec{s}|\:\:\}$ as
  $\vec{x}:\vec{s}$, and the morphism 
  $\langle\vec{\alpha}(\vec{x}:\vec{s})|\:\:\rangle$ 
  as $\vec{\alpha}(\vec{x}:\vec{s})$ (or, for that matter,
  as $\vec{\alpha}$ when the source type is obvious). We will also 
  use $\land$ and the comma interchangeably for concatenation 
  of tuples of partial equations. 
\end{remark}

\begin{table}
  \fbox{
    \begin{tabular}{rcc}
      \textbf{Structural Rules}
      &
      \AxiomC{$\vphantom{\vartheta}$}
      \RightLabel{Ax}
      \UnaryInfC{$\vartheta(\vec{x}) \entailsPE \vartheta(\vec{x})$}
      \DisplayProof
      &
      \AxiomC{$\vartheta(\vec{x}) \entailsPE \vartheta'(\vec{x})$}
      \AxiomC{$\vartheta'(\vec{x}) \entailsPE \vartheta''(\vec{x})$}
      \RightLabel{Cut}
      \BinaryInfC{$\vartheta(\vec{x}) \entailsPE \vartheta''(\vec{x})$}
      \DisplayProof
      \\[\erh]
      &\multicolumn{2}{c}{
        \AxiomC{$\vartheta(\vec{y}) \entailsPE \vartheta'(\vec{y})$}
        \RightLabel{Substitution}
        \UnaryInfC{$\converges{\vec{\alpha}(\vec{x})}
          \land \vartheta[\vec{\alpha}/\vec{y}]
          \entailsPE  \vartheta'[\vec{\alpha}/\vec{y}]$}
        \DisplayProof
}
      \\[\erh]
      \textbf{Equality}
      &
      \AxiomC{$\vphantom{\ltrue}$}
      \RightLabel{Reflexivity}
      \UnaryInfC{$\ltrue \entailsPE\converges{\vec{x}}$}
      \DisplayProof
      &
      \AxiomC{$\vphantom{\vartheta}$}
      \RightLabel{Equality}
      \UnaryInfC{$\vec{x} \partEq \vec{y}
        \land \vartheta 
        \entailsPE \vartheta[\vec{y}/\vec{x}]$
}
      \DisplayProof
      \\[\erh]
      &
      \multicolumn{2}{c}{
        \AxiomC{$\vphantom{\vartheta}$}
        \RightLabel{Strictness 1}
        \UnaryInfC{$\vec\alpha(\vec{x}) \partEq \vec \beta(\vec{x})
          \entailsPE \converges{\vec{\alpha}(\vec{x})}
          \land
          \converges{\vec{\beta}(\vec{x})}$}
        \DisplayProof
}
      \\[\erh]
      &
      \multicolumn{2}{c}{
        \AxiomC{$\vphantom{\vartheta}$}
        \RightLabel{Strictness 2}
        \UnaryInfC{$\converges{\vec{\alpha}(\vec{\beta}(\vec{x}))}\:
          \entailsPE 
          \converges{\vec{\beta}(\vec{x})}$}
        \DisplayProof
}
      \\[\erh]
      \textbf{Conjunctions}
      &
      \AxiomC{$\vphantom{\vartheta}$}
      \RightLabel{$\ltrue$}
      \UnaryInfC{$\vartheta(\vec{x}) \entailsPE \ltrue(\vec{x})$}
      \DisplayProof
      &
      \AxiomC{$\vartheta(\vec{x})\entailsPE \vartheta'(\vec{x})$}
      \AxiomC{$\vartheta(\vec{x})\entailsPE\vartheta''(\vec{x})$}
      \RightLabel{$\land$}
      \BinaryInfC{$\vartheta(\vec{x})\entailsPE \vartheta'(\vec{x})
        \land\vartheta''(\vec{x})$}
      \DisplayProof
      \\[\erh]
      &
      \AxiomC{$\vphantom{\vartheta}$}
      \RightLabel{$\pi_1$}
      \UnaryInfC{$\vartheta(\vec{x})\land
        \vartheta'(\vec{x})\entailsPE \vartheta(\vec{x})$}
      \DisplayProof
      &
      \AxiomC{$\vphantom{\vartheta}$}
      \RightLabel{$\pi_2$}
      \UnaryInfC{$\vartheta(\vec{x})\land
        \vartheta'(\vec{x})\entailsPE \vartheta'(\vec{x})$}
      \DisplayProof
    \end{tabular}
}
  \caption{Entailment between Partial Equations}%
  \label{tab:PalVicEnt}
\end{table}

\begin{proposition}
  Let $\Sigma$ be a signature and $\someTheory$ be a theory. Then
  $\freeCat{\Sigma}{\someTheory}$ is a bicategory of partial maps.
\end{proposition}
\begin{proof}
  Given objects $\obC{\zeta}{x}$ and $\obC{\eta}{y}$, with
  $\vec{x}$ and $\vec{y}$ disjoint tuples of variables (which can always be
  achieved, up to isomorphism of objects, by renaming), we let
  \begin{align*}
    \obC{\zeta}{x} \otimes \obC{\eta}{y} &=
    \{\vec{x},\vec{y}|\vec{\zeta}(\vec{x}),\vec{\eta}(\vec{y})\}; 
    \intertext{for morphisms
      $\morphC{\alpha(x)}{\zeta(x)}$ and $\morphC{\beta(y)}{\eta(y)}$
      (with, again, disjoint tuples of source and target variables), we let}
    \morphC{\alpha(x)}{\zeta(x)}\otimes \morphC{\beta(y)}{\eta(y)} &=
    \morphC{\alpha(x),\beta(y)}{\zeta(x),\eta(y)}.  
    \intertext{The unit
      object is given by} 
    I &= \{()|\ltrue\}
    \intertext{(i.e.\ the empty
      tuple with no constraints), whereas the morphism $!$ is given,
      for any $X$, by} 
    !_X &= \morphC{()}{\ltrue(\vec{x})} 
  \end{align*}
  i.e.\ the empty tuple of function symbols (in the appropriate
  variables) with no constraints. It is straightforward, if tedious,
  to verify that these make $\freeCat{\Sigma}{\someTheory}$ into a
  strict monoidal bicategory.  For example, the fact that morphisms
  are lax $!$-homomorphisms comes down, in the case of a morphism
  $\morphC{\alpha}{\eta}$, to
  \begin{displaymath}
    \morphC{()}{\eta(\alpha(x))} \: \extendedBy \: 
    \morphC{()}{\ltrue(\alpha(x))}.
  \end{displaymath}

  We now define the comonoid structure: if $X = \obC{x}{\vartheta}$, then
  \begin{align*}
    \Delta_X & = \morphC{x,x}{\ltrue}: &\obC{\vartheta}{x}
    &\rightarrow 
    \{\vec{x},\vec{y}|\vec{\vartheta}(\vec{x}),\vec{\vartheta}(\vec{y})\}\\
    \nabla_X & = \morphC{x}{x \partEq y} : 
    &\{\vec{x},\vec{y}|\vec{\vartheta}(\vec{x}),\vec{\vartheta}(\vec{y})\}
    &\rightarrow 
    \obC{\vartheta}{x}
  \end{align*}
  This is the only possible comonoid structure
  on $X$: we can write a candidate structure in components 
  as $f,g$, and then it is clear, after some manipulation,
  that $f = g = \id_X$. 
  
  Verification of the adjunction is straightforward:
  we need
  \begin{align*}
    &\eta:& \id_X & \extendedBy \nabla_X \Delta_X 
    &&\text{i.e.}& \id_{\obC{\vartheta}{x}}
    & \extendedBy \morphC{x}{\ltrue}\\
    &\epsilon: &\Delta_X\nabla_X & \extendedBy \id_X
    &&\text{i.e.}&
    \morphC{x,y}{x\partEq y}&
    \extendedBy \id_{\obC{\vartheta}{x}\times\obC{\vartheta}{y}}
  \end{align*}
  and these clearly hold. 
  The identities \eqref{eq:bpmFrob} and \eqref{eq:meetOfMorphisms} 
  amount to 
  \begin{align*}
    \morphC{x,y}{x \partEq y} \: &= \: \morphC{x,y}{x \partEq y}
    &\text{and}\\
    \morphC{\alpha}{\zeta, \alpha\partEq \beta, \eta} 
    \: &\extendedBy \: \morphC{\alpha}{\zeta} 
  \end{align*}
  when $X = \obC{\vartheta}{x}$, $f = \morphC{\alpha}{\zeta}$,
  and $g = \morphC{\beta}{\eta}$. 
\end{proof}
\begin{lemma}
  $\freeCat{\Sigma}{\someTheory}$ is a split bicategory of partial maps.
\end{lemma}
\begin{proof}
  We first show that any coreflexive 
  $f: \obC{\zeta}{x:s} \rightarrow \obC{\zeta}{x:s}$ is 
  of the form $\morphC{x:s}{\zeta(x:s)\land \eta(x:s)}$. Suppose, then,
  that $\morphC{\alpha(x:s)}{\eta(x:s)} \extendedBy \id_{\obC{\zeta}{x:s}}$: by 
  definition of $\extendedBy$, we have
  \begin{align*}
    \someTheory, \converges{\vec{\alpha}(\vec{x})},\,
    \vec{\zeta}(\vec{x}),\,
    \vec{\eta}(\vec{x})
    &\entailsPE 
    \vec{\alpha}(\vec{x}) \partEq \vec{x}
    \intertext{and from this follows that}
    \morphC{\alpha}{\eta} & \extendedBy \morphC{\id}{\eta} & \text{and}\\
    \morphC{\id}{\eta} & \extendedBy \morphC{\alpha}{\eta}, &\text{i.e.}\\
    \morphC{\alpha}{\eta} & \partEq \morphC{\id}{\eta}
    \intertext{but now we can split this coreflexive by the 
      pair of morphisms}
    i & = \morphC{x}{} &\{\vec{x}|\vec{\zeta}(\vec{x}),
    \vec{\eta}(\vec{x})\}
    & \rightarrow \obC{\eta}{x}\\
    j & = \morphC{x}{\eta(x)}& \obC{\zeta}{x} 
    & \rightarrow \{\vec{x}|\vec{\zeta}(\vec{x}),
    \vec{\eta}(\vec{x})\}.  \end{align*}
\end{proof}
\begin{corollary}
  $\totalCat{\Sigma}{\someTheory}$, the category of objects and total
  morphisms of $\freeCat{\Sigma}{\someTheory}$, is cartesian (that is, it
  has all finite limits).
\end{corollary}
\begin{proof}
  According to  \citet[Theorem~2.3]{carboni87:_bicat_of_partial_maps},
  this is the case iff $\freeCat{\Sigma}{\someTheory}$
  is functionally complete, and we have just shown that it is. 
\end{proof}
\subsubsection{Examples}
We now need to do some work on rephrasing the concepts
of partial cartesian categories in the more concrete terms of
$\freeCat{\Sigma}{\someTheory}$. 
\begin{definition}
  We say that a morphism 
  \begin{align*}
    f: A &\rightarrow B
    \intertext{in a partial cartesian category is \emph{total} if}
    !_A f &= !_B.
  \end{align*}
\end{definition}
\begin{lemma}
  In $\freeCat{\Sigma}{\someTheory}$, a morphism
  \begin{align*}
    \morphC{\alpha}{\eta}: 
    \{\vec{x}| \vec{\vartheta}(\vec{x})\}
    & \rightarrow \{ \vec{y}|
    \vec{\vartheta}'(\vec{y})\}
    \intertext{is total iff}
    \vartheta & \entailsPE \converges{\vec{\alpha}} 
    \land \vec{\eta}
  \end{align*}
\end{lemma}
\begin{proof}
Routine. 
\end{proof}

\begin{definition}
  In a bicategory of partial maps, we define
  $\land: \hom(X,Y) \times \hom(X,Y) \rightarrow \hom(X,Y)$ to be
  the map which takes $(f,g)$ to 
  $\nabla_Y f \otimes g \Delta_X$. 
\end{definition}
\begin{lemma}
  $\land$ is a least upper bound on the poset $\hom(X,Y)$.
\end{lemma}
\begin{proof}
  See  \citet[Lemma~2.1]{carboni87:_bicat_of_partial_maps}.
\end{proof}
\begin{example}
  In $\freeCat{\Sigma}{\someTheory}$, maps $X \rightarrow I$
  are of the form $\morphC{()}{\zeta(x)}$. We find, by easy calculation,
  that $\morphC{()}{\zeta(x)} \land \morphC{()}{\eta(x)}
  = \morphC{()}{\zeta(x)\land \eta(x)}$: so the $\land$ that
  we have just defined agrees with the $\land$ in Palmgren-Vickers
  logic. 
\end{example}

\begin{example}\label{ex:and}
  If $\vec{\alpha}(\vec{x}), \vec{\beta}(\vec{x})\in
  \hom_{\freeCat{\Sigma}{\someTheory}}(\{\vec{x}:\vec{s}:\ltrue\},
  \{\vec{y}:\vec{t}|\ltrue\})$, then
  \begin{align*}
    \vec{\alpha}(\vec{x})\land \vec{\beta}(\vec{x}) \: &=\:
    \morphC{\alpha}{\alpha\partEq \beta}
    \intertext{and}
    \morphC{()}{\alpha \partEq \beta} \:&=\: !_{\vec{t}}
    (\vec{\alpha} \land \vec{\beta}) \\
    &\:=\: (!_{\vec{t}} \vec{\alpha}) \land (!_{\vec{t}}\vec{\beta});
    \intertext{if $\vec{\zeta}(\vec{x}), \vec{\eta}(\vec{x}) \in
      \hom_{\freeCat{\Sigma}{\someTheory}}(\{\vec{x}:\vec{s}:\ltrue\},I)$,
      then}
    \vec{\zeta} \land \vec{\eta} \: &= \:
    \morphC{()}{\vec{\zeta}(\vec{x}) \partEq \vec{\eta}(\vec{x})}\\
    &= \: \morphC{()}{\vec{\zeta}(\vec{x}) \land \vec{\eta}(\vec{x})}\\
    &= \: \morphC{()}{\vec{\zeta}(\vec{x}),\vec{\eta}(\vec{x})}
  \end{align*}
\end{example}
\begin{proof}
  Routine computation. 
\end{proof}
\begin{remark}
This result is perhaps not surprising: in the paradigm model
of these things -- that is, sets and partial functions -- 
$I$ is a one-element set, and all that $\partEq$ then
worries about is definedness: so the coincidence of 
$\land $ and $\partEq$ is only to be expected. 
\end{remark}

\subsubsection{The Internal Model in $\freeUnsplitCat{\Sigma}{\someTheory}$}
We will first show how $\freeUnsplitCat{\Sigma}{\someTheory}$ is 
related to the Palmgren-Vickers logic. So,
let $\someTheory$ be a theory, with signature $\Sigma$, in this
logic. 
We will now show how to associate, to each formula in the signature $\Sigma$, 
a semantic value in a homset of
$\freeUnsplitCat{\Sigma}{\someTheory}$, and we will show that this
gives a sound and complete model of $\someTheory$.
\begin{definition}
  Given a formula $\vec{\zeta}(\vec{x})$, with free
  variables $\vec{x}:\vec{s}$, of  $\lang(\Sigma)$,
  associate to it the \emph{semantic value}
  \begin{displaymath}
    \semval{\zeta}{\vec{s}}
    \quad = \quad
    \morphC{()}{\zeta(x)}\: :\:
    \vec{s} \rightarrow I.
  \end{displaymath}
  (where, as usual, we abbreviate $\{\vec{x}:\vec{s}|\ltrue\}$
  to $\vec{s}$). 
  
  Given a pair of formulae $\vec{\zeta}(\vec{x}:\vec{s})$
  and $\vec{\eta}(\vec{x}:\vec{s})$, with the
  same variables, we say that
  \begin{displaymath}
    \vec{\zeta}\:\Vdash \: \vec{\eta}
  \end{displaymath}
  (in words: $\vec{\zeta}$ semantically entails $\vec{\eta}$)
  if, in $\hom_{\freeCat{\Sigma}{\someTheory}}(\vec{s},I)$,
  \begin{displaymath}
    \semval{\vec{\zeta}}{\vec{s}}
      \leq
    \semval{\vec{\eta}}{\vec{s}} 
  \end{displaymath}
\end{definition}

Before we prove soundness and completeness, we need a lemma.%
\begin{lemma}(Cf.\ \cite[Lemma~3.3]{VickersS:PHLCC}])\label{lemma:vicKey}
  Given $\morphC{\alpha}{\:} :\;\vec{s}
  \rightarrow \vec{t} $, and a formula of the Palmgren-Vickers logic
  $\vec{\zeta}(\vec{y}:\vec{t})$, then
  \begin{displaymath}
    \semval{\vec{\zeta}(\vec{\alpha}(\vec{x}))}{\vec{s}}
    \land
    \semval{\converges{\vec{\alpha}(\vec{x})}}{\vec{s}}
    = 
    \semval{\vec{\zeta}}{\vec{t}}
    \morphC{\alpha(x)}{\;}
  \end{displaymath}
\end{lemma}
\begin{proof}
  We argue by cases. If $\vec{\zeta} = \ltrue$, then the left hand side is
  (by computation, and using Example~\ref{ex:and})
  $\morphC{()}{\ltrue} \land \morphC{()}{\converges{\alpha}}$, whereas
  the right hand side is
  $\morphC{()}{\ltrue}\morphC{\vec{\alpha}}{\ltrue}$. An easy
  computation shows that the two are equal.

  If $\vec{\zeta} = (\vec{\beta}(\vec{y}) \partEq \vec{\gamma(\vec{y})})$,
  with $\vec{\beta},\vec{\gamma}: \vec{t}\rightarrow \vec{u}$, then
  (again using Example~\ref{ex:and}, together with the definitions of
	the operations in $\freeUnsplitCat{\Sigma}{\someTheory}$) we find that the left hand side
  is $(!_{\vec{u}}\vec{\beta}(\vec{\alpha})) \land
  (!_{\vec{u}}\vec{\gamma}(\vec{\alpha}) \land
  (!_{\vec{t}}\vec{\alpha})$; the right hand side is
  $!_{\vec{u}}(\vec{\beta}\land\vec{\gamma})\vec{\alpha}$. The two can
  easily be seen to be equal using the naturality of $\land$ and the
  fact that it is a supremum, together with the lax naturality of $!$.
  
  If $\vec{\zeta} = \vec{\eta} \land \vec{\vartheta}$, then the left
  hand side is $\vec{\eta}(\vec{\alpha}) \land
  (\vec{\vartheta}(\vec{\alpha}) \land (!_{\vec{t}}\vec{\alpha})$, whereas
  the right hand side is $(\vec{\eta}\land\vec{\vartheta})\vec{\alpha}$:
  equality follows, as before, by the properties of $\land$ and 
  $!$. 
\end{proof}
\begin{theorem}
  The internal model is sound and complete for
  $\someTheory$.
\end{theorem}
\begin{proof}
  We first show that the model is sound: we start with the rules in
  Table~\ref{tab:PalVicEnt}. Axiom and cut are inherited
  from the partial order structure on the homsets of
  $\freeCat{\Sigma}{\someTheory}$: reflexivity follows
  from a routine computation. For the equality axiom, we
  argue as follows. Given $\vec{\zeta}:\vec{s}\rightarrow I$
  and tuples of variables $\vec{x},\vec{y}:\vec{s}$, then 
  \begin{align*}
    \semval{ \vec{\zeta}(\vec{x})\land\vec{x} \partEq \vec{y} }{\vec{s}}
    & =(\vec{\zeta}\land \nabla_{\vec{s}})
    (\Delta_{\vec{s}} \otimes \id_{\vec{s}})\\
    & = (\zeta\land\id_{\vec{s}})(\id_{\vec{s}} \otimes \nabla_{\vec{s}})
    (\Delta_{\vec{s}} \otimes \id_{\vec{s}})\\
    & = (\zeta \land\id_{\vec{s}}) \Delta_{\vec{s}} \nabla_{\vec{s}}
    &\text{by \eqref{eq:bpmFrob}}\\
    & = (\id_{\vec{s}}\land\vec{\zeta}) \Delta_{\vec{s}} \nabla_{\vec{s}}
    &\text{by symmetry of $\land$}\\
    & = (\zeta\land\id_{\vec{s}})(\nabla_{\vec{s}}\otimes\id_{\vec{s}})
    (\id_{\vec{s}} \otimes \Delta_{\vec{s}})&\text{by \eqref{eq:bpmFrob}}\\ 
    & = \semval{\vec{x} \partEq \vec{y} \land \vec{\zeta}(\vec{y})}{\vec{s}}\\
    & \leq \semval{\vec{\zeta}(\vec{y})}{\vec{s}}&
    \text{by properties of $\land$}
  \end{align*}
  The strictness properties follow from Lemma~\ref{lemma:vicKey}. 
  We also need to show that the elements of $\someTheory$ are
  all interpreted as tautologies, i.e.\ that, for 
  $\vec{\tau}(\vec{x}:\vec{x}) \in \someTheory$,
  we have $\semval{\vec{\tau}}{\vec{s}} = \ltrue$. But this
  trivially follows from the definition of $\leq$ in homsets. 

  Completeness is likewise trivial. Suppose that we have
  $\vec{\zeta}, \vec{\eta}$, with 
  \begin{align*}
    \semval{\vec{\zeta}}{\vec{s}}
    &\leq \semval{\vec{\eta}}{\vec{s}}.
    \intertext{Then, by definition of semantic values,}
    \morphC{()}{\vec{\zeta}} & \leq \morphC{()}{\vec{\eta}};
    \intertext{and so, by definition of $\leq$ in homsets,}
    \someTheory, \zeta &\entailsPE \eta
  \end{align*}
  which was to be proved. 
\end{proof}

\subsubsection{Models and Free Categories}
In this section we shall relate
$\freeUnsplitCat{\Sigma}{\someTheory}$  to  the theory of
bicategories by showing that it is    the free bicategory of
partial maps on the signature $(\Sigma,\someTheory)$, and that
$\freeCat{\Sigma}{\someTheory}$ is the free \emph{functionally complete}
bicategory of partial maps on that signature.

Now in general free things are produced by a left
adjoint to some forgetful functor: in this case the functor will
produce, from a bicategory of partial maps, a theory in a signature. So we need
first to define the corresponding category structure on 
theories in signatures. 
\begin{definition}
	We define a locally posetal bicategory $\text{theories}$, whose
	objects are theories in signatures, as follows. 
	Given theories $(\Sigma,\someTheory)$ and $(\Sigma', \someTheory')$,
	a morphism of signatures is a map $\Phi$ from the situations of $\Sigma$
	to the situations of $\Sigma'$, together with a map $\Psi$ from the actions
	of $\Sigma$ to those of $\Sigma'$, compatible with the typing of 
	actions and such that the induced map on formulae sends $\someTheory$ to
	a subset of $\someTheory'$. $\langle \Phi,\Psi\rangle \sqsubseteq 
	\langle \Phi', \Psi'\rangle$ iff $\Phi = \Phi'$ and if, for all
	actions $f$, $\someTheory' \entailsPE \Psi(f) \sqsubseteq \Psi'(f)$. 
\end{definition}
We define the forgetful functor $U$ as follows: if $\someCat$ is a bicategory of
partial maps, then $U(\someCat)$ will have, for situations, the
objects of $\someCat$ and for actions the 1-cells of $\someCat$. An entailment
will be in the corresponding theory if it holds in $\someCat$. 

Now in order to show the required adjunction, we must show that
the posets $\homSet{BPM}{\freeUnsplitCat{\Sigma}{\someTheory}}{\someCat}$
and $\homSet{\text{theories}}{(\Sigma,\someTheory)}{U(\someCat)}$ are
naturally isomorphic. We do this by showing that they are
both isomorphic to the poset of \emph{models} of $(\Sigma,\someTheory)$ in
$\someCat$; we define models as follows. 
\begin{definition}
  Given a signature $\Sigma$, a theory $\someTheory$, and a 
  category of partial maps $\someCat$, then a 
  $\Sigma$\emph{-structure} in $\someCat$ is
  given by the following data:
  \begin{enumerate}
  \item For each type $s$ of $\Sigma$, an object $\semval{s}{}$
    of $\someCat$
  \item for each action symbol $\alpha: s \rightarrow t$
    of $\Sigma$, a 1-cell $\semval{\alpha}{s}:\semval{s}{} \rightarrow 
    \semval{t}{}$ of
    $\someCat$
 \end{enumerate}
\end{definition}
We can now interpret formulae in $\someCat$: to be precise, we will
associate, to each variable tuple $\vec{x}:\vec{s}$ of our logic an object
$\semval{\vec{s}}{}$ of $\someCat$, to each formula
$\vec{\eta}(\vec{x}:\vec{s})$ in that context a 1-cell
$\semval{\vec{\eta}}{\vec{s}}:
\semval{\vec{s}}{} \rightarrow I$, and, to each 1-cell
$\morphC{\alpha}{\eta}$ a 1-cell (with
appropriate source and target) of $\someCat$.
 
\begin{definition}
  We define the following semantic values by mutual 
  recursion:
  \begin{enumerate}
  \item $\semval{\vec{x}}{}
    = \semval{s_1}{} \otimes \cdots \otimes \semval{s_k}{}$
  \item $\semval{\vec{\eta}(\vec{x}):\vec{s}}{\vec{s}}$
    is defined by induction on the structure of $\vec{\eta}$:
    \begin{enumerate}
    \item $\semval{\ltrue}{\vec{s}}$
      is
      \begin{displaymath}
        \bfig
        \morphism<500,0>[\semval{\vec{s}}{}`I;!_{\semval{\vec{s}}{}}]
        \efig
      \end{displaymath}
    \item $\semval{\vec{\alpha}(\vec{x}:\vec{s}) \partEq
\vec{\beta}(\vec{x}:\vec{s})}
{\vec{s}}$
      is defined as follows (suppose that $\vec{\alpha}, \vec{\beta}: \vec{s}
\rightarrow
      \vec{t}$)
      \begin{displaymath}
        \bfig
        \morphism[\semval{\vec{s}}{}`\semval{\vec{s}}{}\otimes
\semval{\vec{s}}{};
        \Delta_{\semval{\vec{s}}{}}]
        \morphism(500,0)<900,0>[\semval{\vec{s}}{}\otimes \semval{\vec{s}}{}`
        \semval{\vec{t}}{}\otimes
\semval{\vec{t}}{};\semval{\vec{\alpha}}{\vec{s}}
        \otimes\semval{\vec{\beta}}{\vec{t}}]
        \morphism(1400,0)[\semval{\vec{t}}{}\otimes
\semval{\vec{t}}{}`\semval{\vec{t}}{};
        \nabla_{\semval{\vec{t}}{}}]
        \morphism(1900,0)[\semval{\vec{t}}{}`I;!_{\semval{\vec{t}}{}}]
        \efig
      \end{displaymath}
    \item $\semval{\vec{\zeta}\land \vec{\eta}}{\vec{s}}
      = \semval{\vec{\zeta}}{\vec{s}} \land \semval{\vec{\eta}}{\vec{s}}$
		\item If the logic has Heyting operations, then
			$\semval {\heyting{\vec{\zeta}}{\vec{\eta}}}{\vec{s}} = 
			\heyting{\semval{\vec{\zeta}}{\vec{s}}}{\semval{\vec{\eta}}{\vec{s}}}$
    \end{enumerate}
  \item We define the semantic values of 1-cells
    as follows:
    \begin{enumerate}
    \item An action symbol $\alpha: s \rightarrow t$
      of $\Sigma$ has semantic value
      \begin{displaymath}
        \semval{\alpha}{s}: \semval{s}{}\rightarrow \semval{t}{}
      \end{displaymath}
    \item A 1-cell of the form 
      \begin{displaymath}
        \langle \alpha_i(x_i:s_i)|\:\rangle\: : \:
        \vec{s} 
        \rightarrow t
      \end{displaymath}
      has semantic value
      \begin{displaymath}
        \bfig
        \morphism<1700,0>[\semval{s_1}{}\otimes\cdots\otimes\semval{s_k}{}
        `I \otimes \cdots \otimes \semval{t}{}\otimes \cdots\otimes I;
        !_{s_1} \otimes \cdots \otimes \semval{\alpha}{s_i}
        \otimes \cdots \otimes I]
         \morphism(1700,0)<0,-400>%
         [I \otimes \cdots \otimes \semval{t}{}%
         \otimes \cdots\otimes I`\semval{t}{};\equiv]
        \efig
      \end{displaymath}
    \item A 1-cell of the form 
      \begin{displaymath}
        \morphC{\alpha,\alpha'}{\:}\: : \:
          \vec{s} \rightarrow \vec{t} \otimes \vec{t}'
      \end{displaymath}
      has semantic value
      \begin{displaymath}
        \bfig
       
\morphism<700,0>[\semval{\vec{s}}{}`\semval{\vec{s}}{}\otimes\semval{\vec{s}}{}
;%
        \Delta_{\semval{\vec{s}}{}}]
        \morphism(700,0)<1100,0>[\semval{\vec{s}}{}\otimes\semval{\vec{s}}{}`%
        \semval{\vec{t}}{}\otimes\semval{\vec{t}'}{};%
        \semval{\alpha}{\vec{s}}\otimes\semval{\alpha'}{\vec{s}}]
        \morphism(1800,0)<650,0>[\semval{\vec{t}}{}\otimes\semval{\vec{t}'}{}`%
        \semval{\vec{t} \otimes \vec{t}'}{};\equiv]
        \efig
      \end{displaymath}
    \end{enumerate}
  \end{enumerate}
\end{definition}

We now have the following notion of a \emph{model} of a 
theory. 
\begin{definition}
  Given a theory $\someTheory$ in a signature $\Sigma$, we 
  say that a structure $\semval{\cdot}{\cdot}$, with
  values in $\someCat$, is a \emph{model} if,
  for all 
  \begin{align*}
    \vec{\zeta}(\vec{x}:\vec{s}) &
    \entailsPE \vec{\eta}(\vec{x}:\vec{s}) &\text{in} \:
    \someTheory,\\
    \semval{\zeta}{\semval{\vec{s}}{}} & \extendedBy 
    \semval{\eta}{\semval{\vec{s}}{}} & \text{in}\: 
    \hom_{\someCat}(\semval{\vec{s}}{},I)
  \end{align*}
\end{definition}

This notion of model, which is defined by recursion on
the structure of the language, can, in fact, be 
greatly simplified: models, as we have defined them,
are the same as 2-functors from $\freeUnsplitCat{\Sigma}{\someTheory}$.
\begin{theorem}\label{thm:firstCorrespondence}
  Every model of $\someTheory$ in a structure $\Sigma$ corresponds
  to a 2-functor, preserving the structure of a bicategory of
  partial maps: 
  \begin{displaymath}
    \freeUnsplitCat{\Sigma}{\someTheory} \rightarrow \someCat
  \end{displaymath}
\end{theorem}
\begin{proof}
  Suppose that we are given a model; it will assigns semantic values
  in $\someCat$ to the objects and 1-cells of
  $\freeCat{\Sigma}{\someTheory}$, and we prove first that these
  semantic values make up a functor. Note first that, because
  $\someCat$ is a model under the given assignment of semantic values,
  it preserves the $\extendedBy$ relation on homsets, and thus, in
  particular, preserves equality of morphisms. Consider the semantic
  values assigned to $\Delta_{\vec{s}}$, $!_{\vec{s}}$, for situation
  tuples $\vec{s}$: computation shows that $\semval{\cdot}{\vec{s}}$
  is a comonoid homomorphism, so that, since
  $\Delta_{\semval{\vec{s}}{}}$ is the unique comonoid structure on
  $\semval{\vec{s}}{}$, we must have
  $\semval{\Delta_{\vec{s}}}{\vec{s}} = \Delta_{\semval{\vec{s}}{}}$
  and $\semval{!_{\vec{s}}}{\vec{s}} =
  !_{\semval{\vec{s}}{}}$. Further computation also shows that
  $\semval{\vec{\alpha}\otimes \vec{\beta}}{\vec{s}\otimes\vec{t}} =
  \semval{\vec{\alpha}}{\vec{s}} \otimes
  \semval{\vec{\beta}}{\vec{t}}$.  Furthermore, the adjunction between
  $\Delta$ and $\nabla$ is given by equations and inequalities, and
  these are, by hypothesis, preserved by $\semval{\cdot}{}$; we also
  know that, since our categories are locally posetal, adjunctions are
  unique, and so $\semval{\cdot}{}$ preserves $\nabla$. The semantic
  values of formulae $\vec{\zeta}(\vec{z})$ are all defined in terms
  of $\Delta$, $\nabla$, $\otimes$ and $!$: consequently,
  $\semval{\cdot}{}$ preserves the semantic values of formulae, in the
  sense that
  \begin{equation}
    \semval{\vec{\eta}(\vec{x}:\vec{s})}{\vec{s}}^{\someCat}
    =
    \semval{\semval{\vec{\eta}(\vec{x}:\vec{s})}{\vec{s}}^{I}}
{\vec{s}}^{\someCat}\label{eq:internalFunctorial}
  \end{equation}
  where $\semval{\cdot}{}^{\someCat}$ is the semantic value in
  our given model and $\semval{\cdot}{}^{I}$ is the semantic 
  value in the internal model. 

  Conversely, suppose that we have a functor $\phi:
  \freeCat{\Sigma}{\someTheory}\rightarrow \someCat$. Applying $\phi$
  to the semantic values of the internal model, we get semantic values
  in $\someCat$, and they are easily verified to give a
  model. Furthermore, this gives a one-to-one correspondence between 
  models of $\someTheory$ with signature $\Sigma$: the argument
  of the previous section shows that such functors are
  given by their values on $\freeCat{\Sigma}{\someTheory}^0$.
\end{proof}
\begin{theorem}
  The above semantics is sound and complete for $\cdot\entailsPE\cdot$
\end{theorem}
\begin{proof}
  Consider \eqref{eq:internalFunctorial}. Suppose that $\vec{\zeta} \entailsPE
  \vec{\eta}$ is valid in the Vickers-Palmgren logic: then
  $\semval{\vec{\zeta}}{\vec{s}}^I \extendedBy \semval{\vec{\eta}}{\vec{s}}^I$,
  since the internal model is sound. But we get the model in $\someCat$ by
  applying a suitable 2-functor to the internal model, and so we have 
  $\semval{\vec{\zeta}}{\vec{s}}^{\someCat} \extendedBy
\semval{\vec{\eta}}{\vec{s}}^{\someCat}$.
  So we have soundness. Completeness is straightforward: we already have
completeness
  for the internal model, and the internal model is a model in the above sense,
  obtained from the trivial interpretation in
$\freeUnsplitCat{\Sigma}{\someTheory}$. 
\end{proof}

We can rephrase Theorem~\ref{thm:firstCorrespondence} as
\begin{corollary}
  $\freeUnsplitCat{\Sigma}{\someTheory}$ is the 
  free bicategory of partial maps on the signature 
  $(\Sigma,\someTheory)$. 
\end{corollary}
\begin{proof}
	As explained above, we must show that
the posets $\homSet{BPM}{\freeUnsplitCat{\Sigma}{\someTheory}}{\someCat}$
and $\homSet{\text{theories}}{(\Sigma,\someTheory)}{U(\someCat)}$ are
naturally isomorphic. Theorem~\ref{thm:firstCorrespondence} shows
that the former poset is naturally isomorphic to the 
poset of models of $\someTheory$; the latter poset, on the other hand, is 
trivially isomorphic to the poset of models.
\end{proof}
\subsection{Weak Comma Objects}
The above correspondences can be extended to the case of bicategories of partial
maps with weak comma objects. As we have remarked above, we will continue using
the 2-covariant domain fibration for the semantics of the internal language of these,
even though, for bicategories with weak comma objects, we can also define a 
2-contravariant domain fibration. 

The existence of weak comma objects, as shown above, is equivalent to the 
a Heyting operation on domain posets; it is thus natural to augment the internal
language with a corresponding primitive. The Palmgren-Vickers language consisted
of conjunctions of equations between partial functions: the language for weak 
comma objects will be generated by the Heyting operation and conjunction from
equations between partial functions.  The language and the calculus are
defined in Table~\ref{tab:WeakCommaEnt}; we call it the Heyting-Palmgren-Vickers logic,
which we abbreviate to HPV.

\begin{table}
  \fbox{
    \begin{tabular}{rcc}
			\textbf{Formulae}
			&
			\multicolumn{2}{c}{
			$\vartheta(\vec{x})\quad :=\quad f(\vec{x}) \partEq g(\vec{x}) \:\:|\:\:
			\vartheta(\vec{x}) \land \vartheta{\vec{x}} \:\:|\:\:
			\heyting{\vartheta(\vec{x})}{\vartheta(\vec{x})}$}\\[\erh]
      \textbf{Structural Rules}
      &
      \AxiomC{$\vphantom{\vartheta}$}
      \RightLabel{Ax}
      \UnaryInfC{$\vartheta(\vec{x}) \entailsHPV \vartheta(\vec{x})$}
      \DisplayProof
      &
      \AxiomC{$\vartheta(\vec{x}) \entailsHPV \vartheta'(\vec{x})$}
      \AxiomC{$\vartheta'(\vec{x}) \entailsHPV \vartheta''(\vec{x})$}
      \RightLabel{Cut}
      \BinaryInfC{$\vartheta(\vec{x}) \entailsHPV \vartheta''(\vec{x})$}
      \DisplayProof
      \\[\erh]
      &\multicolumn{2}{c}{
        \AxiomC{$\vartheta(\vec{y}) \entailsHPV \vartheta'(\vec{y})$}
        \RightLabel{Substitution}
        \UnaryInfC{$\converges{\vec{\alpha}(\vec{x})}
          \land \vartheta[\vec{\alpha}/\vec{y}]
          \entailsHPV  \vartheta'[\vec{\alpha}/\vec{y}]$}
        \DisplayProof
}
      \\[\erh]
      \textbf{Equality}
      &
      \AxiomC{$\vphantom{\ltrue}$}
      \RightLabel{Reflexivity}
      \UnaryInfC{$\ltrue \entailsHPV\converges{\vec{x}}$}
      \DisplayProof
      &
      \AxiomC{$\vphantom{\vartheta}$}
      \RightLabel{Equality}
      \UnaryInfC{$\vec{x} \partEq \vec{y}
        \land \vartheta 
        \entailsHPV \vartheta[\vec{y}/\vec{x}]$
}
      \DisplayProof
      \\[\erh]
      &
      \multicolumn{2}{c}{
        \AxiomC{$\vphantom{\vartheta}$}
        \RightLabel{Strictness 1}
        \UnaryInfC{$\vec\alpha(\vec{x}) \partEq \vec \beta(\vec{x})
          \entailsHPV \converges{\vec{\alpha}(\vec{x})}
          \land
          \converges{\vec{\beta}(\vec{x})}$}
        \DisplayProof
}
      \\[\erh]
      &
      \multicolumn{2}{c}{
        \AxiomC{$\vphantom{\vartheta}$}
        \RightLabel{Strictness 2}
        \UnaryInfC{$\converges{\vec{\alpha}(\vec{\beta}(\vec{x}))}\:
          \entailsHPV 
          \converges{\vec{\beta}(\vec{x})}$}
        \DisplayProof
}
      \\[\erh]
      \textbf{Conjunctions}
      &
      \AxiomC{$\vphantom{\vartheta}$}
      \RightLabel{$\ltrue$}
      \UnaryInfC{$\vartheta(\vec{x}) \entailsHPV \ltrue(\vec{x})$}
      \DisplayProof
      &
      \AxiomC{$\vartheta(\vec{x})\entailsHPV \vartheta'(\vec{x})$}
      \AxiomC{$\vartheta(\vec{x})\entailsHPV\vartheta''(\vec{x})$}
      \RightLabel{$\land$}
      \BinaryInfC{$\vartheta(\vec{x})\entailsHPV \vartheta'(\vec{x})
        \land\vartheta''(\vec{x})$}
      \DisplayProof
      \\[\erh]
      &
      \AxiomC{$\vphantom{\vartheta}$}
      \RightLabel{$\pi_1$}
      \UnaryInfC{$\vartheta(\vec{x})\land
        \vartheta'(\vec{x})\entailsHPV \vartheta(\vec{x})$}
      \DisplayProof
      &
      \AxiomC{$\vphantom{\vartheta}$}
      \RightLabel{$\pi_2$}
      \UnaryInfC{$\vartheta(\vec{x})\land
        \vartheta'(\vec{x})\entailsHPV \vartheta'(\vec{x})$}
      \DisplayProof
			\\[\erh]
			\textbf{Heyting}
			&\multicolumn{2}{c}{\AxiomC{$\vartheta(\vec{x}),\vartheta'(\vec{x})\entailsHPV \vartheta''(\vec{x})$}
			 \RightLabel{HeytingI}
			 \UnaryInfC{$\vartheta(\vec{x})\entailsHPV \heyting{\vartheta'(\vec{x})}{\vartheta''(\vec{x})}$}
			\DisplayProof}\\[\erh]
			& \multicolumn{2}{c}{\AxiomC{$\vartheta(\vec{x})\entailsHPV \vartheta'(\vec{x})$}
			\AxiomC{$\vartheta(\vec{x}) \entailsHPV \heyting{\vartheta'(\vec{x})}{\vartheta''(\vec{x})}$}
			\RightLabel{HeytingE}
			\BinaryInfC{$\vartheta(\vec{x})\entailsHPV \vartheta''(\vec{x})$}
			\DisplayProof}
    \end{tabular}
}
  \caption{The Heyting-Palmgren-Vickers Logic}%
  \label{tab:WeakCommaEnt}
\end{table}
\subsubsection{Free Categories with Comma Objects}
We can now define free categories with comma objects: the 
definition exactly replicates the definition for partial cartesian categories, 
except that the logic is now the HPV logic.
\begin{definition}
  Given a signature $\Sigma$, together with a theory $\someTheory$
	of the HPV logic
  in that signature, we define a locally posetal two-category
  $\freeHey{\Sigma}{\someTheory}$ as follows. 
  \begin{description}
  \item[objects] An object of $\freeHey{\Sigma}{\someTheory}$
    will be written $\obC{\zeta}{x:s}$, where 
    $x:s$ is a variable tuple and where $\zeta$ is a term of
	the HPV logic with that source type.
\item[1-cells] A 1-cell will be
    written
    \begin{displaymath}
      \morphC{\alpha(x:s)}{\zeta(x:s)}: 
      \obC{\eta}{x:s} \rightarrow \obC{\vartheta}{y:t}
    \end{displaymath}
    where we require that
    \begin{enumerate}
    \item $\vec{s}$ is the source type of
      $\vec{\alpha}$, of $\vec{\zeta}$ and of $\vec{\eta}$, 
    \item $\vec{t}$ is the target type of
      $\vec{\alpha}$ and the source type of 
      $\vec{\vartheta}$, and
    \item the  entailment
      \begin{displaymath}
        \someTheory,
        \vec{\zeta}, \vec{\eta}, \converges{\vec{\alpha}}
        \: \entailsHPV \: 
        \vec{\vartheta}(\vec{\alpha})
      \end{displaymath}
      holds (informally, we require that 
      $\morphC{\alpha(x:s)}{\zeta(x:s)}$ should factor through the
      subobject of $\{\vec{y}:\vec{t}\}$ defined
      by $\vec{\vartheta}$). 
    \end{enumerate}

    Composition of 1-cells is defined as follows. Suppose that 
    \begin{align*}
      \morphC{\alpha(x)}{\zeta(x)}:\:
      &\obC{\vartheta}{x:s} \rightarrow \obC{\vartheta'}{y:t}\\
      \morphC{\beta(y)}{\eta(y)}:\:
      &\obC{\vartheta'}{y:t} \rightarrow \obC{\vartheta''}{z:u}&\text{then}\\
      \morphC{\beta(y)}{\eta(y)}\morphC{\alpha(x)}{\zeta(x)}
      \:&=\:
      \morphC{\beta(\alpha(x))}{\eta(\alpha(x)),
        \zeta(x)}
    \end{align*}
    The unit 1-cell, on an object $\obC{\zeta}{x:s}$, is
    $\morphC{x}{\zeta(x)}$. 
    
  \item[2-cells] The 2-cells of our category will
    be defined as follows. Given 
    \begin{align*}
      \morphC{\alpha}{\zeta}, 
      \morphC{\alpha'}{\zeta'}&:
      \obC{\eta}{x} \rightarrow \obC{\vartheta}{y}
      \intertext{we say that $\morphC{\alpha}{\zeta}\extendedBy
      \morphC{\alpha'}{\zeta'}$ iff}
      \someTheory,
      \converges{\vec{\alpha}},\vec{\zeta},
      \vec{\eta}\:\: &\entailsHPV\, \vec{\alpha}
      \partEq \vec{\alpha'}\land \vec{\eta'};   
    \end{align*}
  informally, whenever $\morphC{\alpha}{\zeta}$ is defined, then
  so too is $\morphC{\alpha'}{\zeta}$ and they are equal. 
  \end{description}
\end{definition}
We also define the following subcategory of $\freeHey{\Sigma}{\someTheory}$:
we will need it for our proof of freeness. 
\begin{definition}
  Let $\freeUnsplitHey{\Sigma}{\someTheory}$ be the full subcategory
  of $\freeHey{\Sigma}{\someTheory}$ whose objects have no constraints,
  i.e.\ are all of the form $\{\vec{x}:\vec{s}|\:\}$.
\end{definition}
We can, as before, prove, by routine calculation  
\begin{lemma}
  $\freeHey{\Sigma}{\someTheory}$ and 
  $\freeUnsplitHey{\Sigma}{\someTheory}$
  are locally posetal bicategories.  
\end{lemma}
We can also prove 
\begin{lemma}
  $\freeHey{\Sigma}{\someTheory}$ and 
  $\freeUnsplitHey{\Sigma}{\someTheory}$
	have terminal objects and weak comma objects.
	\label{lemma:freeWeakComma}
\end{lemma}
\begin{proof}
	Terminal objects are straightforward. We define 
	weak comma objects as follows: suppose that we have
	morphisms
	\begin{displaymath}
  \begin{xy}
		\morphism||/@{->}_(0.3){\langle f|\zeta(\vec{x})\rangle}/<350,-350>[\{\vec{x}:A | \vartheta(\vec{x})\} `\{\vec{x}':B |\vartheta'(\vec{x}')\}; ]
		\morphism(700,0)
		||
		/@{->}^(0.3){\langle g|\zeta'(\vec{x}'')\rangle}/<-350,-350>[\{\vec{x}'':C | \vartheta(\vec{x}'')\} `\{\vec{x}':B |\vartheta'(\vec{x}')\};]
  \end{xy}
\end{displaymath}
	Then we define the comma object
	\begin{displaymath}
		\langle f|\zeta(\vec{x}) \rangle  \comma \langle g| \zeta'(\vec{x}')\rangle \quad = \quad 
		\{ \vec{x}:A, \vec{y}:B | \vartheta(\vec{x}), \vartheta'(\vec{y}), \heyting{(\converges{f}\land \zeta)}
		{(f \partEq g \land \vartheta')} \}
	\end{displaymath}
	and, after some calculation, we can show that it has the required universal property.
\end{proof}
\subsection{Comprehensions}
We now revert to the fibrational setting. 
\subsubsection{Equality in the Fibres}\label{sec:EqInFib}
Because of our results on presentations of
partial cartesian categories, we can use a term-based
notation for contexts in our calculus: thus, over
an object such as $A \otimes A \otimes B$, we can
write sequents in this form:
\begin{displaymath}
	x:A , y:A, z:A | \Gamma(x,y,z) \entails \Delta(x,y,z)
\end{displaymath}
Now we define equalities in more abstract terms.
\begin{definition}
  Given a context $x:A, y: A , z: B$, define
  the proposition $x = y$ as
  \begin{displaymath}
     \lnot  (\nabla_A\otimes \id_b) * \lfalse   
  \end{displaymath}
\end{definition}
\begin{proposition}\label{prop:eqGenuine}
  The following rules for $=$ are admissible:
  \begin{center}
    \AxiomC{$x:A, y:A, z:B | \Gamma, x = y \entails \Gamma'$}
    \doubleLine
    \UnaryInfC{$x:A, z:B | \Gamma[y/x] \entails \Gamma'[y/x]$}
    \DisplayProof
  \end{center}
\end{proposition}
\begin{proof}
  Expanding the definition of $=$, this is equivalent to the 
  admissibility of 
  \begin{center}
    \AxiomC{$x:A, y:A, z:B | \Gamma \entails \nabla_A^* \lfalse, \Gamma'$}
    \doubleLine
    \UnaryInfC{$x:A, z:B | \Delta^*\Gamma \entails \lfalse,
      \Delta^*\Gamma'$}
    \DisplayProof
  \end{center}
  and this follows from the adjunction $\Delta^* \dashv \nabla^*$,
  together with Lemma~\ref{lemma:magicPullback} to show that
  $\Delta^* \Gamma$ and $\Delta^*\Gamma'$
  can be identified with $\Gamma[y/x]$ and $\Gamma'[y/x]$.
\end{proof}
By \cite[Prop.~3.2.3]{JacobsCLTT}, this is enough to show that 
$=$ is a genuine equality. 

We can also relate equality in the fibres to equality in the
base. Suppose we have an object $A$ of $\someCat$, and let
$A_0 = \{ A|\vec{\alpha} \partEq \vec{\beta}\}$ with inclusion
$i: A_0 \rightarrow A$. 
\begin{lemma}\label{lemma:inclAdj}
  $\lapb_i \ltrue$ is $\vec{\alpha} = \vec{\beta}$.  
\end{lemma}
\begin{proof}
  Because
  $A_0$ arises from the splitting of a coreflexive, we have
  \begin{align*}
  &&i: A_0 &\rightarrow A, &&&j: A &\rightarrow A_0,\\
  \text{with}&&i &\dashv j&\text{and}&& ji &= \id_{A_0}.
  \intertext{Now, by Lemma~\ref{lemma:magicAdjoint} we have} 
  &&\lapb_i \ltrue & =  \lnot j^*\lfalse\\
  &&& = \lnot j^* i^* \lfalse 
  \intertext{since $j^*$ is a boolean algebra homomorphism}
  &&& = \lnot (ij)^* \lfalse\\
  &&& = (\vec{\alpha} = \vec{\beta})
  \intertext{by definition of $=$.} 
  \end{align*}
\end{proof}
\begin{corollary}\label{cor:comprehension}
  The following rule is admissible:
  \begin{center}
    \AxiomC{$A, \vec{\alpha} \partEq \vec{\beta}
      | i^*\Gamma \entails i^* Q$}
    \doubleLine
    \UnaryInfC{$A| \vec{\alpha} = \vec{\beta},
      \Gamma \entails Q$}
    \DisplayProof
  \end{center}
\end{corollary}
\begin{proof}
This is simply the adjunction $\lapb_i \dashv i^*$, together
with Lemma~\ref{lemma:inclAdj}. 
\end{proof}
We have, by induction, 
\begin{corollary}
	The following rule is admissible, where $\vec{\zeta}$ is in the
	fragment generated by $\land$ from equalities, and where we 
	write, by abuse of notation, $\vec{\zeta}$ both in 
	the base and in the fibres:
	\def\fCenter{\entails}
	\begin{prooftree}
		\Axiom$A, \vec{\zeta} | \Gamma \fCenter \Delta$
		\doubleLine
		\UnaryInf$A | \vec{\zeta}, \Gamma \fCenter \Delta$
	\end{prooftree}
\end{corollary}

\section{Davidson's Example}
We conclude with an extended example: this is of the philosopher Davidson's
argument about the equality of actions. It is important because, in
the philosophical community, the notion of equality of action seems
to have significant consequences: roughly speaking, first-class objects 
are those which have meaningful equalities. However, we have a system in which 
we can define equalities on actions on fairly weak premises: they come 
from a well-established treatment of partiality, together with quite
weak assumptions about the existence of limits. Consequently, we can 
show that Davidson's argument probably establishes less than he takes it to.
But in order to do that, we have to show that our equalities are capable of playing the
same argumentative role as Davidson's equalities: and this is what we do in this section.
\subsection{Davidson}\label{sec:excuses}
The philosopher Donald Davidson (following Austin \cite{AustinJL:plee})
considers the following pattern of reasoning.
\begin{quote}
  `I didn't know that it was loaded' belongs to one standard pattern
  of excuse. I do not deny that I pointed the gun and pulled the
  trigger, nor that I shot the victim. My ignorance explains how it
  happens that I pointed the gun and pulled the trigger intentionally,
  but did not shoot the victim intentionally. \ldots The logic of this
  sort of excuse includes, it seems, at least this much structure: I
  am accused of doing $b$, which is deplorable. I admit I did $a$,
  which is excusable. My excuse for doing $b$ rests upon my claim that
  I did not know that $a = b$.
  \cite[p.~109]{00DavidsonD:logfas}\label{quote:Dav}
\end{quote}
Davidson, then, is arguing for two things:
\begin{enumerate}
\item equalities between actions are meaningful, and
\item we use these equalities in common-sense reasoning about action.
\end{enumerate}
These claims of Davidson's have given rise to a great deal of argument, of which
the main protagonists are Davidson \cite{DavidsonD:essae} and Kim
\cite{KimJ:Evepe}. Although 
this debate has generated a lot of high-quality philosophy, it has been
strangely inconclusive, and, perturbingly, strangely orthogonal to other
issues in the semantics of natural language. 
\begin{example}[Davidson]
The facts in Davidson's example can now  be expressed
as follows. Suppose that we have actions:  $\pullTrigger$
stands for ``pull trigger'', $\shoot$ stands for ``shoot'',
and $\kill$ stands for ``kill''. Suppose, also, that we have
propositions $\loaded$ and $\aimed$.
The semantics of these propositions will be as follows.  $\loaded(x)$
will be true in a world $x$ iff the gun is loaded in that world:
similarly, $\pullTrigger(x) = \shoot(x)$ is true in a world $x$ iff
the result of pulling the trigger is the same as the result of
shooting. Thus, \eqref{eq:loadCons} says that the gun is loaded iff
pulling the trigger is the same as shooting it: similarly,
\eqref{eq:shootCons} says that the gun is aimed iff shooting it is the
same as killing the victim. 

We formulate the effects of these actions in context in the following axioms:
we use partial equality and are careful to
  stipulate that actions are performable. Let $\shoot$ stand for
  `shoot' and $\pullTrigger$ stand for `pull trigger'. 
  \begin{equation}
    \shoot(x) \partEq \shoot(x),
    \pullTrigger(x) \partEq \pullTrigger(x)
    \;|\;
    \loaded(x) \:  \dashv\vdash \: \pullTrigger(x) = \shoot(x)
    \label{eq:loadCons}
  \end{equation}
  We also describe the effects of shooting as follows:
  \begin{equation}
    \shoot(x) \partEq \shoot(x)
    \;|\;
    \alive(x) \entails \dead(\shoot(x))\label{eq:shootCons}
  \end{equation}
  Here we assume that $\loaded$, $\alive$ and $\dead$ are predicates
  which are defined for all values of $x$. $\shoot$ and
  $\pullTrigger$, on the other hand, are actions which may not be
  performable in all circumstances (there may not be a gun to hand,
  for example), and thus these entailment have non-trivial contexts.
  
  We prove the unfortunate consequence in Table~\ref{table:davidsonShooting}. 
	We use two contexts here: $\vartheta_2$ is the base context consisting of the
	facts holding. $\vartheta_1$ is
	$\{\vartheta_2|\pullTrigger(x) = \shoot(x)\}$; let $i$ be the inclusion
	of $\vartheta_1$ in $\vartheta_2$.
	
	The inferences are annotated as follows: \textsf{subs} is
	the substitution rule from the type theory (i.e.\ $i^*$ of
	Table~\ref{tab:bicatRules}). \textsf{rep}  is the equality rule  
	of Proposition~\ref{prop:eqGenuine}, and \textsf{comp} is the 
	rule of Corollary~\ref{cor:comprehension}. \eqref{eq:loadCons}
	and \eqref{eq:shootCons} are used as axioms. $\dag$ is the 
	unfortunate fact that the gun is loaded, likewise used as an axiom.
  \begin{table}
        \begin{prooftree}
					\AxiomC{}
					\RightLabel{$\dag$}
					\UnaryInfC{
					$\vartheta_2\;|\;
					\entails \loaded(x)$}
          \AxiomC{}
          \RightLabel{\eqref{eq:loadCons}}
          \UnaryInfC{
            $\vartheta_2\;|\;
            \loaded(x) \entails \pullTrigger(x) \partEq \shoot(x)$
          }
          \AxiomC{}
          \RightLabel{\eqref{eq:shootCons}}
          \UnaryInfC{
            $\vartheta_2
            \;|\; \alive(x) 
            \entails \dead(\shoot(x))$
          }
					\RightLabel{\textsf{subs}}
          \UnaryInfC{
            $\vartheta_1 \:|\; 
            \restPred{\alive(x)}{\vartheta_1}
            \entails
            \restPred{\dead(\shoot(x))}{\vartheta_1}$
          }
					\RightLabel{\textsf{rep}}
          \UnaryInfC{
            $\vartheta_1 \:|\; 
            \restPred{\alive(x)}{\vartheta_1}
            \entails
            \restPred{\dead(\pullTrigger(x))}{\vartheta_1}$
          }
					\RightLabel{\textsf{comp}}
          \UnaryInfC{
            $\vartheta_2\;|\;
            \shoot(x) \partEq \pullTrigger(x),
            \alive(x)
            \entails
            \dead(\pullTrigger(x))$
          }
          \BinaryInfC{
            $\vartheta_2\;|\;
            \loaded(x),
            \alive(x)
            \entails
            \dead(\pullTrigger(x))$
          }
					\RightLabel{\textsf{cut}}
          \BinaryInfC{
            $\vartheta_2\;|\;
            \alive(x)
            \entails
            \dead(\pullTrigger(x))$
          }
          \end{prooftree}
     \caption{The Davidsonian Scenario}\label{table:davidsonShooting}
   \end{table}
\end{example}
This is a deduction of the eventual death, using equational reasoning,
and starting from the axioms describing the initial situation. The 
reasoning in Davidson's example is probably best regarded as
abductive, and this can be handled in this system in terms of proof
search, starting from the observed death and assuming $\dag$. 

\subsubsection{Evaluation}
Davidson uses this example to argue for the first-class status of
actions, on the basis that equalities between them are meaningful
and, in fact, used in reasoning about action. 
However,  although our formalisation accounts for the inferences
in Davidson's story, and deals with them equationally, it hardly supports Davidson's reading: 
the partial function semantics of our logic treats equality as
equality between function values. But Davidson's argument, 
based on this reasoning, would 
need  equalities between the functions themselves
rather than, as we do, between their \emph{values}.
However, our semantics of equality seems to be difficult to avoid:
given a plausible treatment of partiality -- that given by partial cartesian
categories, and independently discovered by Palmgren and Vickers -- 
we get equality for free, and it is definitely an equality between
function values.

\nocite{AustinJL:phip} 
\bibliographystyle{authordate1}
\bibliography{cat,ceb,logic,frame,WhiteGG,phil,phil2,%
VickersS,prooftheory,semantics,linguistics}

\end{document}

%% file: reiterMacros.tex
\def\entails{\vdash}
\def\notentails{\nvdash}

\def\semEnt{\Vdash}

\def\ltrue{\top}
\def\bottom{\mathord{\perp}}
\def\lfalse{\bottom}

\def\partEq{\bumpeq}
\def\entailsPE{\entails_{\!\!\!\scriptstyle\partEq}}
\def\restPred#1#2{{#1}_{|#2}}

\def\vec#1{\boldsymbol{#1}}

\def\thefibcat{\mathcal{E}}
\def\termMod#1{\thefibcat_{#1}}
\def\semval#1#2{\llbracket #1 \rrbracket_{#2}}
\def\compose{\mathbin{\circ}}
\def\comp{\compose}

\def\contraction#1{\overline{#1}}

\def\transition#1#2#3{\xymatrix@=1.5em{{#1}\ar@/^/[r]^{#2}&{#3}}}

\def\lang{\mathfrak{L}}

\def\id{\textnormal{Id}}

\def\contranest#1#2{\{#1\}^{#2}}

\def\wPre#1{#1^{*}}

\def\pred#1{\text{\scriptsize\textsf{#1}}}

\def\shoot{\pred{sh}}
\def\pullTrigger{\pred{pt}}
\def\kill{\pred{kill}}
\def\loaded{\pred{loaded}}
\def\aimed{\pred{aimed}}
\def\alive{\pred{alive}}
\def\dead{\pred{dead}}